\documentclass[reqno]{amsart}
\usepackage{float}
\usepackage{amsmath}
\usepackage{graphicx}
\usepackage{latexsym}
\usepackage{amsfonts}
\usepackage{amssymb}
\theoremstyle{plain}
\newtheorem{theorem}{Theorem}
\newtheorem{corollary}[theorem]{Corollary}
\newtheorem{lemma}[theorem]{Lemma}
\newtheorem{proposition}[theorem]{Proposition}
\theoremstyle{definition}

\newtheorem{remark}[theorem]{Remark}
\newtheorem*{remark*}{Remark}

\newcommand{\bDelta}{\Delta}
\newcommand{\Rd}{\R^{d}}

\newcommand{\pr}{\mathbf P}
\newcommand{\e}{\mathbf E}
\newcommand{\Z}{\mathbb{Z}}

\newcommand{\R}{\mathbb{R}}

\newcommand{\bL}{{\mathbf L}}

\newcommand{\cN}{{\mathcal N}}
\newcommand{\1}{\mathsf{1}}
\newcommand{\Mcal}{\mathcal{M}}
\newcommand{\mfo}{\mathcal{M}_{\rm f}}
\newcommand{\scf}[2]{\langle #1,\,#2 \rangle}

\begin{document}
\title[Regularity and irregularity of superprocesses]
{Regularity and irregularity of superprocesses with $(1+\beta)$-stable branching mechanism}
\author[Mytnik]{Leonid Mytnik} 
\address{Faculty of Industrial Engineering and Management, Technion Israel Institute of Technology, Haifa 32000, Israel}
\email{leonid@ie.technion.ac.il}

\author[Wachtel]{Vitali Wachtel} \address{Institut f\"ur Mathematik,
Universit\"at Augsburg, 86135 Augsburg, Germany}
\email{vitali.wachtel@mathematik.uni-augsburg.de}

\begin{abstract}
  We would like to give an overview of results on regularity, or better to say "irregularity",
  properties of densities at fixed times of super-Brownian motion with $(1+\beta)$-stable branching for
  $\beta<1$. First, the following dichotomy for the density is shown: it is continuous in the dimension $d=1$
  and locally unbounded in all higher dimensions where it exists. Then in $d=1$ we determine pointwise
  and local H\"older exponents of the density, and calculate the multifractal spectrum corresponding
  to pointwise H\"older exponents.
\end{abstract}
\maketitle

\section{Introduction, main results and discussion}
\label{sec:1}
 
\subsection{Model and motivation}

This paper is devoted to regularity and fractal properties of superprocesses
with $(1+\beta)$-branching. Regularity properties of functions is the most
classical question in analysis. Typically one is interested in such properties
as continuity/discontinuity and differentiability. Starting from Weierstrass,
who constructed an example of a continuous but nowhere differentiable
function, people got more and more interested in such 'strange' properties of
functions. Trajectories of stochastic processes give a rich source of such
functions. The most classical example is the Brownian motion: almost every
path of the Brownian motion is continuous but nowhere differentiable.

In order to measure the regularity of a function $f$ at point $x_0$,
we need to introduce so-called H\"older classes $C^\eta(x_0)$. One says
that $f\in C^\eta(x_0)$, $\eta>0$ if the exists a polynomial $P$ of degree
$[\eta]$ such that
$$
|f(x)-P(x-x_0)|=O(|x-x_0|^\eta).
$$
For $\eta\in (0,1)$ the above definition coincides with the definition of H\"older
continuity with index $\eta$ at $x_0$.
With the definition of $f\in C^\eta(x_0)$ at hand, let us define the 
\emph{pointwise H\"older exponent} of $f$ at $x_0$:
\begin{equation}
H_f(x_0)\,:=\,\sup\bigl\{\eta>0:\ f\in C^{\eta}(x_0)\bigr\},
\end{equation}
and we set it to $0$ if $f \not \in C^{\eta}(x_0)$ for all $\eta>0$. To simplify the
exposition we will sometimes call $H_f(x_0)$ the H\"older exponent of $f$ at $x_0$. 

It is well known that the Weierstrass functions have the same H\"older exponent at all
points. The same is true for the Brownian motion: the pointwise H\"older exponent at
all times is equal to $1/2$ with probability one. However, there exist functions with
H\"older exponent changing from point to point. In such a case one speaks of a
\emph{multifractal} function.
For some classical examples of deterministic multifractal functions we refer to
Jaffard~\cite{Jaffard1997}. In studying multifractal functions people are interested
in the 'size' of the set of points with given H\"older exponent. To measure these sizes
for different H\"older exponents of a function $f$ one introduces the following function 
\begin{equation}
\label{dim}
D(\eta)={\rm dim}\{x_0:H_f(x_0)=\eta\},
\end{equation}
where ${\rm dim}(A)$ denotes the Hausdorff dimension of the set $A$. The mapping
$\eta\mapsto D(\eta)$ reveals the so-called \emph{multifractal spectrum} related
to pointwise H\"older exponents of $f$. A standard example of a multifractal
random function is given by a Levy process with infinite Levy measure. Its
multifractal spectrum was determined by Jaffard in \cite{Jaff99}. In general, the
multifractal analysis of random functions and measures has attracted attention for
many years and has been studied for example in Dembo et al.\ \cite{DemboPeresRosenZeitouni2001},
Durand~\cite{Durand2009}, Hu and Taylor \cite{HuTaylor2000}, Klenke and M\"{o}rters
\cite{KlenkeMoerters2005}.

All these examples are related to Levy processes, which have independent increments.
Whenever the independence structure is lost, the analysis becomes much more complicated.
However, there are still some examples of stochastic processes without independent increments,
for which  the rigorous analysis of multifractal spectrum is possible. The multifractal spectrum of measures
defined on branching random trees has been studied by M{\"o}rters and Shieh~\cite{MoertersShieh2004},
Berestycki, Berestycki and
              Schweinsberg~\cite{BBS07}, and more recently
by Balan\c{c}a~\cite{Balanca15}.  
The analysis of multifractal spectrum has been also 
done in some variations for measure-valued branching processes, see, e.g,
Le\,Gall and Perkins \cite{LeGallPerkins1995}, Perkins and Taylor \cite{PerkinsTaylor1998}, and recently by 
Mytnik and Wachtel~\cite{MW14}. 
One of the aims of this review is to describe results on multifractal spectrum and other regularity properties of  
 measure-valued branching processes, and show the methods of proofs. 
We shall do it in the particular case of $(1+\beta)$-stable
super-Brownian motion, whose densities in low dimensions turn out to have a very non-trivial
regularity structure.

\vspace{12pt}

Before we start with the precise definition of these processes, we need to introduce
the following notation. $\Mcal$ is the space of all
Radon measures on $\Rd$ and $\mfo$ is the space of finite measures on $\Rd$ with weak
topology ($\Rightarrow$ denotes weak convergence). In general if $F$ is a set of functions,
write $F_{+}$ or $F^{+}$ for non-negative functions in $F$. For any metric space $E$, let
$\mathcal{C}_{E}$ (respectively, $\mathcal{D}_{E}$) denote the space of continuous
(respectively, c\`{a}dl\`{a}g) $E$-valued paths with compact-open (respectively, Skorokhod) topology. 
The integral of a function $\phi$ with respect to a measure $\mu$ is written as
$\langle \mu,\phi\rangle$ or $\scf{\phi}{\mu}$ or $\mu(\phi)$. We use $c$ (or $C$) to denote a
positive, finite constant whose value may vary from place to place.  
A constant of the form
$c(a,b,\ldots)$ means that this constant depends on parameters $a,b,\ldots$. 
Moreover, $c_{(\#)}$ will denote a constant appearing in formula line (or array) $(\#)$.

Let $(\Omega,\mathcal{F}_{t},\mathcal{F},\mathbf{P})$ be the probability space with filtration,
which is sufficiently large to contain all the processes defined below. Let
$\mathcal{C}(E)$ denote the space of continuous functions on $E$ and let  
$\mathcal{C}_b(E)$  be the space of bounded functions in  $\mathcal{C}(E)$. Let
$\mathcal{C}_b^n=\mathcal{C}_b^n(\Rd)$ denote the subspace of functions in
$\mathcal{C}_b=\mathcal{C}_b(\Rd)$ whose partial derivatives of order $n$ or less are also
in $\mathcal{C}_b$. A c\`{a}dl\`{a}g  adapted  measure-valued process $X$ is called a
super-Brownian motion with $(1+\beta)$-stable branching if $X$ satisfies the following
martingale problem. For every
$\varphi\in\mathcal{C}_b^{2}$ and every $f\in \mathcal{C}^2(\R)$,
\begin{eqnarray}
\label{equt:14_1}
&&f(\langle X_t,\varphi\rangle)-f(\langle X_0,\varphi\rangle)-{1\over
2}\int_0^t f'(\langle X_s,\varphi\rangle)
\langle X_s,\bDelta \varphi\rangle ds\\
\nonumber
&&\ -\int_0^t \Big(\int_{\Rd}\!\int_{(0,\infty)} \Big(f(\langle
X_s,\varphi\rangle+r\varphi(x))-f(\langle
X_s,\varphi\rangle)-f'(\langle X_s,\varphi\rangle)r\varphi(x) \Big)
n(dr) X_s(dx)\Big)ds
\end{eqnarray}
is an $\mathcal{F}_{t}$-martingale, where 
\begin{equation}
\label{eq:4_15_1}
n(dr)={\beta(\beta+1)\over \Gamma(1-\beta)}\,r^{-2-\beta}\,dr.
\end{equation}

There is also an analytic description of this process: For every positive
$\varphi\in\mathcal{C}_b^{2}$ one has
$$
\e e^{-\langle X_t, \varphi\rangle}=e^{-\langle X_0, u_t\rangle},
$$
where $u$ is the solution to the equation
\begin{equation}
\frac{\mathrm{d}}{\mathrm{d}t}u\ =\
\bDelta u-u^{1+\beta}, \label{logLap}
\end{equation}
with the initial condition $\varphi$.

If $\beta=1$, $X_{.}$ has continuous sample $\mfo$-valued paths, while for $0<\beta<1$,
$X_{.}$ is a.s. discontinuous and has jumps all of the form
$\Delta X_{t} =\delta_{x(t)}m(t)$ and the set of jump times is dense in $[0,\zeta)$, where
\mbox{$\zeta=\inf \{t:\; \langle X_{t},1\rangle=0\}$} is the lifetime of $X$ (see, for
example, Section 6.2.2 of~\cite{bib:daw91}). For $t>0$ fixed, $X_{t}$ is absolutely
continuous a.s. if and only if $d<2/\beta$ (see~\cite{Fleischm88}
and Theorem~8.3.1 of~\cite{bib:daw91}). If $\beta=1$, and $d=1$, then much more can be said ---
$X_{t}$ is absolutely continuous for all $t>0$ a.s. and has a density $X(t,x)$ which is jointly
continuous on $(0,\infty)\times \R$ (see~\cite{bib:konshig88}, \cite{bib:reim89}). In view of the
jumps of $X$ (described above) if $0<\beta<1$, we see that $X_{t}$ cannot have  a density for a
dense set of times a.s. and the regularity properties of the densities are very intriguing. In
this work we consider the ``stable branching'' case of $0<\beta<1$ and consider the question:
\begin{center}
{\it What are the regularity properties of the density of $X$ at fixed times $t$?}
\end{center}
The analytic methods used in~\cite{Fleischm88} to prove the existence of a
density at a fixed time do not shed any light on its regularity properties. However recently there
have been developed technique that allowed to treat these questions. To the best of our  knowledge,
the regularity properties of the densities for super-Brownian motion with $\beta$-stable
branching were first studied in Mytnik and Perkins~\cite{MP03}. It was shown there, that there
is a continuous version of the density if and only if $d=1$.  Moreover, when $d>1$ the density
is very badly behaved. Note that in the case of $\beta=1$, the density of super-Brownian motion
$X_t(dx)$ exists only in dimension $d=1$ and the density has a version which is H\"older
continuous with any exponent smaller than $1/2$ (see Konno and Shiga \cite{bib:konshig88}). 

Now consider the case $\beta<1$.  In a series of papers of Fleischmann, Mytnik, and
Wachtel~\cite{FMW10}, \cite{FMW11} and Mytnik and Wachtel~\cite{MW14} the properties of the
density were studied for a superprocess with $\beta$-stable branching with an $\alpha$-stable
motion, the so-called $(\alpha,d,\beta)$-superprocess. The case of $\alpha=2$,
clearly corresponds to the super-Brownian motion. In~\cite{FMW10}, the results of Mytnik and
Perkins~\cite{MP03} were extended to the case of $\alpha$-stable  motion. In particular, it
was shown that there is a \emph{dichotomy}\/ for the  density function of the measure (in what
follows, we just say the ``density of the measure''): There is a continuous version of the
density of $X_{t}(dx)$  if $d=1$ and $\alpha>1+\beta,$\thinspace\ but otherwise the density 
is unbounded  on open sets of positive $X_{t}(\mathrm{d}x)$-measure. Moreover, in the case of
continuity ($\,d=1$ and $\,\alpha>1+\beta$), H\"{o}lder regularity properties of the density
had been studied in \cite{FMW10}, \cite{FMW11}, \cite{MW14}. It turned out that on any set of
positive $X_t(dx)$ measure, there are points with different pointwise H\"older exponents.
In~\cite{MW14} the Hausdorff dimensions of sets containing the points with certain H\"older
exponents were computed: this reveals the multifractal spectrum related to pointwise
H\"older exponents.   

The main purpose of this paper is to give concise exposition of the results on the
regularity properties of densities of superprocesses with stable branching. On top of it
we will also prove some new results that give a more complete picture of regularity properties.
\subsection{Results on  regularity properties of the densities of super-Brownian motion with stable branching.} 
As we have mentioned above we are interested in the regularity properties of the
$(\alpha,d,\beta)$-superprocess with $\beta\in(0,1)$. In this paper we will consider
the particular case of
\begin{eqnarray}
\label{equt:12_12_1} 
\alpha =2,
\end{eqnarray}
that is, the case of super-Brownian motion. We do it in order to simplify the exposition,
however the proofs go through also in the case of $\alpha$-stable motion process. The
enthusiastic reader who is interested in this general case is invited to go through the series
of papers \cite{FMW10}, \cite{FMW11}, \cite{MW14}.

So, from now on we assume~(\ref{equt:12_12_1}) and $$\beta<1.$$  

The first result deals with the \emph{dichotomy} of the density of super-Brownian motion,
see \cite{MP03} and \cite{FMW10}. Recall that, by~\cite{Fleischm88}, the density for fixed times $t>0$, exists if 
and only if $d<2/\beta$. 
\begin{theorem}
[\textbf{Dichotomy for densities}]\label{T.dichotomy}  Let $d<2/\beta$. Fix $\,t>0$%
\thinspace\ and $\,X_{0}=\mu\in\mathcal{M}_{\mathrm{f}\,}.$\thinspace\ 

\begin{description}
\item[(a)] If\/ $\,d=1$\thinspace\ then with probability one, there is a
continuous version $\tilde{X}_{t}$ of the density function of the measure\/
$\,X_{t}(\mathrm{d}x).$\thinspace\

\item[(b)] If\/ $\,d>1,$\thinspace then with probability one, for all open
$\,U\subseteq\R^{d},$%
\[
\Vert X_{t}\Vert_{U}\ =\ \infty\ \,\text{whenever }\,X_{t}(U)>0.
\]

\end{description}
\end{theorem}

\medskip

For the later results, we assume that $d=1$, $\beta<1$, that is, there is a continuous 
version of the density at fixed time $t$. This density, with a slight abuse of notation,
will be also denoted by $X_t(x)$, $x\in \R$. 

\medskip

In the next theorem the first regularity properties of the density $X_t(\cdot)$, in 
dimension $d=1$, are revealed (see \cite{FMW10}).

\begin{theorem}
[\textbf{Local H\"{o}lder continuity}]\label{T.loc.Hold}  Let $d=1$. Fix $\,t>0$%
\thinspace\ and $\,X_{0}=\mu\in\mathcal{M}_{\mathrm{f}\,}.$\thinspace\ 
\begin{description}
\item[(a)] For each $\,\eta<\eta
_{\mathrm{c}}:=\frac{2}{1+\beta}-1,$\thinspace\ this version ${X}_{t}(\cdot)$ is
locally H\"{o}lder continuous of index $\,\eta:$%
\[
\sup_{x_{1},x_{2}\in K,\,x_{1}\neq x_{2}}\frac{\bigl|{X}_{t}%
(x_{1})-{X}_{t}(x_{2})\bigr|}{|x_{1}-x_{2}|^{\eta}}\,<\,\infty
,\quad\text{compact }\,K\subset\R.
\]
\item[(b)] For every $\,\eta\geq\eta_{\mathrm{c}}%
$\thinspace\ with probability one, for any open $\,U\subseteq\R,$%
\[
\sup_{x_{1},x_{2}\in U,\,x_{1}\neq x_{2}}\frac{\bigl|{X}_{t}%
(x_{1})-{X}_{t}(x_{2})\bigr|}{|x_{1}-x_{2}|^{\eta}}\,=\,\infty
\quad\text{whenever }\,X_{t}(U)>0.
\]
\end{description}
\end{theorem}
One of the consequences of the above theorem is  that  the so-called \emph{optimal index}
for \emph{local} H\"{o}lder continuity of $X_{t}$ equals to $\eta_{\mathrm{c}}$ (see e.g.
Section 2.2 in~\cite{bib:seuret02} for the discussion about local Holder index of continuity). 

The main part of this paper is devoted to studying pointwise regularity properties of $X_t$
which differ drastically from their local regularity properties.  In particular, we are interested
in pointwise H\"older exponent at fixed points and in multifractal spectrum.

Let us fix $t>0$ and  return to the continuous density $X_{t}$ of the $(2,1,\beta)$-superprocess.
In what follows, $H_X(x)$ will denote the 
pointwise H\"{o}lder exponent of $X_{t}$ at $x\in \R.$ 
The next result describes $H_X(x)$ at {\it fixed} points $x\in \R$.  
\begin{theorem}
[\textbf{Pointwise H\"{o}lder exponent at fixed points}]\label{T.fixed}  Let $d=1$. Fix $\,t>0, x\in \R$%
\thinspace\ and $\,X_{0}=\mu\in\mathcal{M}_{\mathrm{f}\,}.$\thinspace\ Define 
$\bar\eta_{\mathrm{c}}:=\frac{3}{1+\beta}-1.$ If $\bar\eta_{\rm c}\neq1$ then, for every fixed $x$,
\[  H_X(x)= \bar\eta_{\mathrm{c}},\;
\quad\mathbf{P}-a.s.\ on\ \{X_{t}(x)>0\}.
\]
\end{theorem}

\begin{remark}
The above result was proved in~\cite{FMW11} for the case of  $\beta>1/2$. This is the case for
which $H_X(x)<1$. In the case of $\beta\leq 1/2$, it was shown in \cite{FMW11}, that for any
fixed point $x\in\R$ 
$$
H_X(x)\geq 1\quad \mathbf{P}-a.s.\ on \ \{X_t(x)>0\}.
$$
Thus, Theorem~\ref{T.fixed} strengthens the result from~\cite{FMW11} by determining the 
pointwise H\"older exponent at fixed points for any $\beta\in(0,1)\setminus\{\frac{1}{2}\}$. 
\end{remark}

\medskip

The above results immediately imply that almost every realization of $X_t$ has points
with different pointwise exponents of continuity. For example, it follows from
Theorem~\ref{T.loc.Hold} that one can find (random) points $x$ where $H_X(x)=\eta_{\rm c}$.
Moreover, it follows from Theorem~\ref{T.fixed} that there are also points $x$ where
$H_X(x)=\bar\eta_{\mathrm{c}}>\eta_{\mathrm{c}}$. This indicates that we are dealing with
random multifractal function $x\mapsto X_t(x)$. 
To study its multifractal spectrum,
for any open $U\subset\R$ and any
$\eta\in(\eta_{\mathrm{c}\,},\bar{\eta}_{\mathrm{c}}]$ define a random set
$$
{\mathcal E}_{U,X,\eta}:=\bigl\{x\in U:\,H_X(x)=\eta\bigr\}
$$
and let $D_{U}(\eta)$ denote its Hausdorff dimension (similarly to \eqref{dim}).

The function $\,\eta\mapsto D_{U}(\eta)$ reveals the
multifractal spectrum related to pointwise H\"{o}lder exponents of $X_t(\cdot)$. This spectrum is determined in the next 
theorem (see~\cite{MW14}) which also claims its independence on $U$. 

\begin{theorem}
[\textbf{Multifractal spectrum}]\label{thm:mfractal}%
Fix $\,\,t>0,\,$\thinspace\ and $\,X_{0}=\mu\in\mathcal{M}_{\mathrm{f}\,}.$
Let $\,d=1$. Then, for any $\eta\in [\eta_{\mathrm{c}}, \bar\eta_{\mathrm{c}}]\setminus\{1\}$
and any open set $U$, with probability one,
\begin{equation}
\label{mfr}
D_U(\eta)= (\beta+1)(\eta-\eta_{\rm c})
\end{equation}
whenever  $X_t(U)>0$.  
\end{theorem}

\begin{remark}
It should be emphasized that the result in Theorem \ref{thm:mfractal} is not uniform in $\eta$.
More precisely, an event of zero probability, on which \eqref{mfr} can fail, is not necessarily
the same for different values of the exponent $\eta$. The question, whether there exists a zero
set $M$ such that \eqref{mfr} holds for all $\eta$ and all $\omega\in M^c$, remains open. 
Note that the uniformity of multifractal spectrum in $\eta$ has been obtained by 
Balan\c ca~\cite{Balanca15} 
for closely related model  
 --- this has been done for level sets of stable random trees. 
\end{remark}

\begin{remark}
The proof of the above theorem fails in the case $\eta=1$, and it is a bit disappointing. Formally, it happens
for some technical reasons, but one has also to note, that this point is critical: it is the
borderline between differentiable and non-differentiable functions. However we still believe that
the function $D_U(\cdot)$ can be continuously extended to $\eta=1$, i.e., 
$D_U(1)=(\beta+1)(1-\eta_{\mathrm{c}})$ almost surely on $\{X_t(U)>0\}$.
\end{remark}
\begin{remark}
The condition $\beta<1$ excludes the case of the quadratic super-Brownian motion, i.e., $\beta=1$.
But it is a known ``folklore" result that the super-Brownian motion $X_t(\cdot)$ is almost
surely monofractal on any open set of strictly positive density. That is, $\pr$-a.s., for any $x$
with $X_t(x)>0$ we have $H_X(x)=1/2$. For the fact that $H_X(x)\geq1/2$, for any $x$, see Konno and
Shiga \cite{bib:konshig88} and Walsh~\cite{bib:wal86}. To get that $H_X(x)\leq 1/2$ on the event
$\{X_t(x)>0\}$ one can show that
$$
\limsup_{\delta\to0}\frac{|X_t(x+\delta)-X_t(x)|}{\delta^\eta}=\infty\quad
\text{for all }x\text{ such that }X_t(x)>0,\ \pr-\text{a.s.,}
$$
for every $\eta>1/2$. This result follows from the fact that for $\beta=1$ the noise driving the
corresponding stochastic equation for $X_t$ is Gaussian (see (0.4) in \cite{bib:konshig88}) in
contrast to the case of $\beta<1$ considered here, where we have driving discontinuous noise with
L\'evy type intensity of jumps.
\end{remark}

\bigskip

\noindent{\bf Organization of the article.}\;
Beyond the description of the regularity properties of the densities of $(2,d,\beta)$-superprocesses,
which  is given above, one of the main goals of the article is to provide the approach for proving
these properties. We will also show how to use this approach to verify some of the results mentioned
above. 
In particular we will give main elements of the proofs of Theorems~\ref{T.dichotomy}, \ref{T.fixed},
\ref{thm:mfractal}. As for the missing details and the proof of Theorem~\ref{T.loc.Hold} we refer
the reader to corresponding papers. 

Now we will say a few of words about the organization of the material in the following sections. 
In Section~\ref{sec:stoch_int_rep}, we give 
the representation of $(2,d,\beta)$-superprocess as a solution to certain martingale problem and 
describe the approach for studying the regularity of the superprocess.
Section~\ref{sec:simple_properties} collects certain properties of  $(2,d,\beta)$-superprocesses 
which later are used for the proofs. Section~\ref{sec:jumps} is very important: it gives precise
estimates on the sizes of jumps of $(2,d,\beta)$-superprocesses, which in turn are crucial for
deriving the regularity properties. Sections~\ref{sec:dichotomy}, \ref{sec:fixed},
\ref{sec:multifactal} are devoted to the partial proofs of Theorems~\ref{T.dichotomy},
\ref{T.fixed}, \ref{thm:mfractal}. Since many of the proofs are technical, at the beginning of
several sections and subsections we give heuristic explanations of our results, which, as we hope,
will provide the reader with some intuition about the results and their proofs.


\section{Stochastic representation for $X$ and description of the approach for determining regularity}
\label{sec:stoch_int_rep}
Let us start with formal definition of $(2,d,\beta)$ superprocess. A c\`{a}dl\`{a}g  adapted  measure-valued
process $X$ is called an $(2,d,\beta)$ superprocess, or super-Brownian motion with $(1+\beta)$-stable
branching, if $X$ satisfies the following martingale problem. For every
$\varphi\in\mathcal{C}_b^{2}(\R^d)$ and every $f\in \mathcal{C}^2(\R)$,
\begin{eqnarray}
\label{equt:14_2}
&&f(\langle X_t,\varphi\rangle)-f(\langle X_0,\varphi\rangle)-{1\over
2}\int_0^t f'(\langle X_s,\varphi\rangle)
\langle X_s,\bDelta
 \varphi\rangle ds\\
\nonumber
&&\ -\int_0^t \Big(\int_D\!\int_{(0,\infty)} \Big(f(\langle
X_s,\varphi\rangle+r\varphi(x))-f(\langle
X_s,\varphi\rangle)- \\
\nonumber
&&  \hspace*{2cm} f'(\langle X_s,\varphi\rangle)r\varphi(x) \Big)
n(dr) X_s(dx)\Big)ds
\end{eqnarray}
is an $\mathcal{F}_{t}$-martingale. The jump measure $n(dr)$ is defined in~(\ref{eq:4_15_1}). 

The following lemma contains a semimartingale decomposition of $X$ which
includes stochastic integrals with respect to discontinuous martingale
measures.
\begin{lemma}
\label{L.mart.dec}Fix $\,X_{0}=\mu\in\mathcal{M}_{\mathrm{f}\,}.$\thinspace\ 
\begin{description}
\item[(a) (Discontinuities)] 
Define the  random measure
\begin{eqnarray}   
\label{equt:7_12_4} 
\cN:=\sum_{s\in J} \delta_{(s,\Delta X_s)},
\end{eqnarray}
where  $J$
denotes the set of all  jump times of $X$. Then
there exists a random counting measure $\,N\!\left(  _{\!_{\!_{\,}}%
}\mathrm{d}(s,x,r)\right)  $\thinspace\ on $\,\R_{+}\times
\R^{d}\times\R_{+}$\thinspace\ 
such that 
\begin{eqnarray}
  \int_{\R_+}\int_{\mathcal {M}_{\rm f}} G(s,\mu)\, \cN(ds,d\mu)
= \int_0^\infty \int_{\R_+} \int_{\R^d} \,G(s,r\,\delta_x) N\!\left(  _{\!_{\!_{\,}}
}\mathrm{d}(s,x,r)\right), 
\end{eqnarray}
for any bounded continuous $G$ on $\R_+\times \Mcal_{\rm f}$.
That is, all discontinuities of the process $\,X$ are jumps upwards of the form $\,r\delta_{x\,}.$

\item[(b) (Jump intensities)] The compensator $\,\widehat{N}$\thinspace\ of\/
$\,N$\thinspace\ is given by%
\begin{equation}
\label{decomp}
\widehat{N}\!\left(  _{\!_{\!_{\,}}}\mathrm{d}(s,x,r)\right)  \,=\,
\mathrm{d}s\,X_{s}(\mathrm{d}x)\,n(\mathrm{d}r),
\end{equation}
that is, \thinspace$\widetilde{N}\,:=\,N-\widehat{N}$\thinspace\ is a martingale
measure on $\,\R_{+}\times\R^{d}\times\R_{+\,}%
.$\vspace{2pt}

\item[(c) (Martingale decomposition)] For all
$\varphi\in\mathcal{C}_{\mathrm{b}}^{2,+}$\thinspace\ and $\,t\geq0,$%
\begin{equation}
\label{mart.dec}
\left\langle X_{t},\varphi\right\rangle \,=\,\left\langle \mu,\varphi
\right\rangle +\int_{0}^{t}\!\mathrm{d}s\,\left\langle X_{s},%
\bDelta
\varphi\right\rangle +M_{t}(\varphi)
\end{equation}
with discontinuous martingale%
\begin{eqnarray}
\label{mart}
\nonumber
t\,\mapsto\,M_{t}(\varphi)\,&:=&
\, \int_{0}^t\int_{\mathcal {M}_{\rm f}}\!
\langle\mu,\varphi\rangle
(\cN- \widehat\cN)\!\left(  _{\!_{\!_{\,}}}\mathrm{d}(ds,d\mu)\right)
\\
\label{equt:7_12_6}
&=&
\,\int_{(0,t]\times\R^{d}\times
\R_{+}}\!r\,\varphi(x)\widetilde{N}\left(  _{\!_{\!_{\,}}}\mathrm{d}(s,x,r)\right)
\end{eqnarray}
{}
\end{description}
\end{lemma}

The martingale decomposition of $X$ in the above lemma is basically proven in Dawson
\cite[Section~6.1]{bib:daw91}. However for the sake of completeness we will reprove it here. 
Some ideas are taken also from~\cite{bib:lm03}. 

\begin{proof}
Since $X$ satisfies the martingale problem~(\ref{equt:14_2}) one can easily get (by formally
taking $f(x)=x$)  that 
\[%
\displaystyle
\left\langle X_{t},\varphi\right\rangle \,=\,\left\langle \mu,\varphi
\right\rangle +\int_{0}^{t}\!\mathrm{d}s\,\left\langle X_{s},%
\bDelta
\varphi\right\rangle +\widetilde M_{t}(\varphi)
\]
where $\widetilde  M_t(\varphi)$ is a local martingale. Moreover, by taking again
$f(\langle X_t,\varphi\rangle)$ for $f\in \mathcal{C}_{\mathrm{b}}^{2}(\R)$,
applying the It\^o formula and comparing the terms with~(\ref{equt:14_2}) one can
easily see that for each $\varphi\in\mathcal{C}_{\mathrm{b}}^{2,+}(\R^d)$,
$\widetilde M_t(\varphi)$ is a purely discontinuous martingale, with the compensator
measure $\widehat\cN_{\varphi}$ given by 
\begin{eqnarray}
\label{equt:7_12_2}
\int_0^t\int_{\R_+} f(s,v) \widehat\cN_{\varphi}(ds,dv) = 
\int_0^t\int_{\R_+}\int_{\R^d} f(s,r\varphi(x))\,X_s(dx)n(dr)\,ds,
\  t\geq 0\hspace{0.5cm}
\end{eqnarray}
for any bounded continuous $f: \R_+\times\R_+\mapsto \R$. This means that if $s$ is
a jump time for $\left\langle X_{\cdot},\varphi\right\rangle$, then
\begin{eqnarray}
\label{equt:7_12_1}
\Delta \left\langle X_{s},\varphi\right\rangle 
= \Delta \widetilde  M_s(\varphi) = r\varphi(x)
\end{eqnarray}
for some $r$ and $x\in \R^d$ (``distributed according  to $\varrho X_{s-}(dx)n(dr)$).
Since this holds for any test function $\varphi$, by putting all  together we infer that
if $s>0$ is the jump time for the measure-valued process $X$, then $\Delta X_s=r\delta_x$
for some $r$ and $x\in \R^d$. 
 
Let $\cN$ be as in~(\ref{equt:7_12_4}). Then by above description of of jumps of $X$, it
is clear that there exists a random counting measure $N$ such that 
\begin{eqnarray}
N := \sum_{(s,x,r): s\in J, \Delta X_{s}=r\delta_x} \delta_{(s,x,r)},
\end{eqnarray}
and {\bf (a)} follows. 

In order to obtain~(\ref{decomp}), we first get $\widehat \cN$ --- the compensator of 
$\cN$. It is defined as follows. For any nonnegative predictable function $F$ on
$\R_+\times\Omega\times \Mcal_{\rm f}$, $\widehat \cN$ satisfies the following quality 
\begin{equation}
\label{compens}
\e_\mu\Big[\int_{\R_+}\int_{\mathcal {M}_{\rm f}} 
 F(s,\omega,\mu) \cN(ds,d\mu) \Big]
=\e_\mu\Big[\int_{\R_+}\int_{\mathcal {M}_{\rm f}}  F(s,\omega,\mu)\,\widehat \cN(ds,d\mu)\Big]. 
\end{equation}
We will show that, in fact, $\widehat \cN$  is defined by the equality 
\begin{eqnarray} 
\label{equt:7_12_3} 
\int_{\R_+}\int_{\mathcal {M}_{\rm f}} G(s,\mu)\,\widehat \cN(ds,d\mu)
=\int_0^\infty ds\int_{\R_+} n(dr)\int_{\R^d} X_s(dx)\,G(s,r\,\delta_x),
\end{eqnarray}
which holds for  any
bounded continuous function $G$ on
$\R_+\times \Mcal_{\rm f}$.
To show~(\ref{equt:7_12_3}),
for any bounded nonnegative continuous function $\varphi$ define random measure 
$$\cN_{\varphi}:=\sum_{s\in J} \delta_{(s,\Delta \left\langle X_{s},\varphi\right\rangle)},$$
By~(\ref{equt:7_12_1}) we have that the compensator measure  of $\cN_{\varphi}$ is $\widehat\cN_{\varphi}$. 
Clearly for any bounded continuous $f: \R_+\times\R_+\mapsto \R$, and 
$\,\varphi\in\mathcal{C}_{\mathrm{b}}^{2,+}$
\begin{eqnarray}
\nonumber
\int_0^t\int_{\mathcal {M}_{\rm f}}  f(s,\left\langle \mu,\varphi\right\rangle)\, \cN(ds,d\mu)
&=&\int_0^t\int_{\R_+}  f(s,v\rangle)\, \cN_{\varphi}(ds,dv), \;\;\forall t\geq 0,
\end{eqnarray}
This immediately  implies, that the corresponding compensator measures also  satisfy
\begin{eqnarray} 
\nonumber 
\lefteqn{\int_0^t\int_{\mathcal {M}_{\rm f}}  f(s,\left\langle \mu,\varphi\right\rangle)\, \widehat\cN(ds,d\mu)}\\
\nonumber
&=&\int_0^t\int_{\R_+}  f(s,v)\, \widehat \cN_{\varphi}(ds,dv)
\\
&=& \int_0^t\int_{\R_+}\int_{\R^d} f(s,r\varphi(x))\,X_s(dx)n(dr)\,ds,
\end{eqnarray}
where the second equality follows by~(\ref{equt:7_12_2}). 
Since collection of  functions in the form $f(s,\left\langle \mu,\varphi\right\rangle)$, 
is dense in the space of bounded continuous functions on  $\R_+\times \Mcal_{\rm f}$, (\ref{equt:7_12_3}) 
follows.  Now~(\ref{decomp}) follows by {\bf (a)} and the definition of $N$. This finishes the 
proof of {\bf (b)}.  

To show~(\ref{equt:7_12_6}) it is enough to derive only the first inequality since the second one is immediate 
by the definition of $N$. Let us first identify the class of functions for which the stochastic integral with respect to 
$(\cN - \widehat \cN)(ds,d\mu)$ is well defined. 
Let $F$ be a measurable function on $\R_+\times \Mcal_{\rm f}$ such
that, for every $t\geq 0$,
\begin{equation}
\label{increpro}
\e_\mu\Big[\Big(\sum_{s\in J\cap [0,t]} F(s,\Delta
X_s)^2\Big)^{1/2}\Big]<\infty\ .
\end{equation}
Following \cite{bib:jacshir87} (Section II.1d), we can then define the
stochastic integral of $F$ with respect to the compensated measure $\cN-\widehat \cN$,
$$\int_0^t F(s,\mu)\,(\cN-\widehat \cN)(ds,d\mu),$$
as the unique purely discontinuous martingale (vanishing at time $0$)
whose jumps are indistinguishable
of the process $1_J(s)\,F(s,\Delta X_s)$.

We shall be interested in the special case where
$F(s,\mu)=F_\phi(s,\mu)\equiv\int \phi(s,x)\mu(dx)$
for some measurable function $\phi$ on $\R_+\times \R^d$ (some
convention is needed when
$\int |\phi(s,x)|\mu(dx)=\infty$, but this will be irrelevant in what
follows). If $\phi$ is bounded, then it is
easy to see that condition (\ref{increpro}) holds. Indeed, we can
bound  separately
\begin{eqnarray*}
\e_\mu\Big[\Big(\sum_{s\leq t} \langle \Delta
X_s,1\rangle^2\,1_{\{\langle \Delta X_s,1\rangle\leq
1\}}\Big)^{1/2}\Big]
&\leq& \e_\mu\Big[\sum_{s\leq t} \langle \Delta
X_s,1\rangle^2\,1_{\{\langle \Delta X_s,1\rangle\leq
1\}}\Big]^{1/2}\\
&=&\Big(\int_{(0,1]} r^2\,n(dr)\,\e_\mu\Big[\int_0^t \langle
X_s,1\rangle ds\Big]\Big)^{1/2}<\infty,
\end{eqnarray*}
and, using the simple inequality
$a_1^2+\cdots+a_n^2\leq(a_1+\cdots +a_n)^2$ for any nonnegative
reals $a_1,\ldots,a_n$,
\begin{eqnarray*}
\e_\mu\Big[\Big(\sum_{s\leq t} \langle \Delta
X_s,1\rangle^2\,1_{\{\langle \Delta X_s,1\rangle>
1\}}\Big)^{1/2}\Big]
&\leq& \e_\mu\Big[\sum_{s\leq t} \langle \Delta
X_s,1\rangle\,1_{\{\langle \Delta X_s,1\rangle >
1\}}\Big]\\
&=&\int_{(1,\infty)} r\,n(dr)\,\e_\mu\Big[\int_0^t \langle
X_s,1\rangle ds\Big]<\infty.
\end{eqnarray*}
In both cases, we have used (\ref{compens}) and the fact that
$\e_\mu[\langle
X_t,1\rangle]\leq \langle
\mu,1\rangle$.

To simplify notation, we write
$$M_t(\phi)=\int_0^t\int_{\R^d}\phi(s,x)\,M(ds,dx)\equiv \int_0^t
F_\phi(s,\mu)\,(\cN-\widehat \cN)(ds,d\mu),$$
whenever (\ref{increpro}) holds for $F=F_\phi$. This is consistent
with the notation of the introduction. Indeed,
if $\phi(s,x)=\varphi(x)$ where $\varphi\in \mathcal{C}^{2}_b(\R)$, then by
the very definition, $M_t(\phi)$
is a purely discontinuous martingale with the same jumps as the
process $\langle X_t,\varphi\rangle$.
Since the same holds for the process
$$\widetilde M_t(\varphi)=\langle X_t,\varphi\rangle -\langle X_0,\varphi\rangle -
{1\over 2}\int_0^t \langle X_s,\bDelta \varphi\rangle ds$$
(see Th\'eor\` eme 7 in \cite{EKR}) we get that $M_t(\phi)=\widetilde
M_t(\varphi)$.\hfill$\square$
\end{proof}

Let $\,\{p_t(x), t\geq 0, x\in \R^d\}$ denote the
continuous transition kernel related to the Laplacian $\bDelta$ in $\R^d,$ and
$(S_t,t\geq 0)$ the related semigroup, that is,
$$
S_tf(x)=\int_{\R^d}p_t(x-y)f(y)dy\text{ for any bounded function }f
$$
and
$$
S_t\nu(x)=\int_{\R^d}p_t(x-y)\nu(dy)\text{ for any finite measure }\nu.
$$
Fix $\,X_{0}=\mu\in\mathcal{M}_{\mathrm{f}\,}\backslash\{0\}.$ Recall that if
$d=1$ then $\,X_{t}(\mathrm{d}x)$ is a.s. absolutely continuous for every fixed
$t>0$ (see \cite{Fleischm88}). In what follows till the end of the section we
will consider the case of $d=1$. \textrm{F}rom the Green function representation
related to (\ref{mart.dec}) (see, e.g., \cite[(1.9)]{FMW10}) we obtain the following
representation of a version of the density function of $\,X_{t}(\mathrm{d}x)$ in $d=1$
(see, e.g., \cite[(1.12)]{FMW10}):
\begin{equation}
\begin{array}
[c]{l}
\displaystyle
X_{t}(x)\,=\,\mu\!\ast\!p_{t}\,(x)\,+\,\int_{(0,t]\times
\R}\!M\!\left(  _{\!_{\!_{\,}}}\mathrm{d}(u,y)\right)
p_{t-u}(y-x)\vspace{6pt}\\%
\displaystyle
\hspace{0.9cm}=:\mu\!\ast\!p_{t}\,(x)\,+Z_t(x),\quad x\in\R,
\end{array}
\label{rep.dens}%
\end{equation}
where
\begin{equation}
\label{def_Zt}
Z_s(x) =  \int_{(0,s]\times
\R}\!M\!\left(  _{\!_{\!_{\,}}}\mathrm{d}(u,y)\right)
p_{t-u}(y-x),\;\;0\leq s\leq t. 
\end{equation}

Note that although $\{Z_s\}_{s\leq t}$ depends on $t$, it does not appear in the notation 
since $t$ is fixed throughout the paper. $M\!\left(  _{\!_{\!_{\,}}}\mathrm{d}(s,y)\right)$
in (\ref{rep.dens}) is the martingale measure related to (\ref{mart}). Note that by Lemma~1.7
of~\cite{FMW10} the class of  ``legitimate'' integrands with respect to the martingale measure
$M\!\left(  _{\!_{\!_{\,}}}\mathrm{d}(s,y)\right) $ includes the set of functions $\psi$ such
that for some $p\in (1+\beta, 2)$, 
\begin{equation}
\label{eq:11_10_1}
\int_0^T ds \int_{\R}dx S_s\mu(x)|\psi(s,x)|^p <\infty, \quad \forall T>0.  
\end{equation} 
We let ${\mathcal L}^p_{\rm loc}$ denote the space of  equivalence classes of 
measurable functions satisfying~(\ref{eq:11_10_1}). For $\beta<1$, 
it is easy to check that, for any $t>0, z\in \R$, 
$$
(s,x)\mapsto p_{t-s}(z-x)\1_{s<t}
$$
is in ${\mathcal L}^p_{\rm loc}$ for any $p\in (1+\beta, 2)$, and hence the stochastic integral
in the representation~(\ref{rep.dens}) is well defined.

$\mu\ast p_t(x)$ is obviously twice differentiable. Thus, the regularity properties  of $X_t(\cdot)$
including its multifractal structure coincide with that of $Z$. Recalling the definitions of
$Z$ and $M(ds,dy)$, we see that there is a ``competition'' between branching and motion:
jumps of the martingale measure $M$ try to destroy smoothness of $X_t(\cdot)$ and $p$ tries to
make $X_t(\cdot)$ smoother. Thus, it is natural to expect, that $\{x:H_{Z}(x)=\eta\}$ can be
described by jumps of a certain order depending on  $\eta$.

Next we connect the martingale measure $M$ with spectrally positive
$1+\beta$-stable processes.

Let $L=\{L_{t}:\,t\geq0\}$ denote a spectrally positive stable process of index $1+\beta.$
Per definition, $L$ is an $\R$-valued time-homogeneous process with independent increments
and with Laplace transform given by%
\begin{equation}
\mathbf{E}\,\mathrm{e}^{-\lambda L_{t}}\,=\,\mathrm{e}^{t\lambda^{1+\beta}%
},\quad\lambda,t\geq0. \label{Laplace}%
\end{equation}
Note that $L$ is the unique (in law) solution to the following martingale
problem:%
\begin{equation}
t\mapsto\mathrm{e}^{-\lambda L_{t}}-\int_{0}^{t}\!\mathrm{d}s\ \mathrm{e}%
^{-\lambda L_{s}}\lambda^{1+\beta}\,\ \text{is a martingale for any}%
\,\ \lambda>0. \label{MP}%
\end{equation}

\begin{lemma}
\label{L9}
Let $d=1$. Suppose
$\,p\in(1+\beta,2)$\thinspace\ and let $\,\psi\in\mathcal{L}_{\mathrm{loc}%
}^{p}(\mu)$\vspace{1pt} with $\psi\geq0$.\thinspace\ Then there exists a
spectrally positive $(1+\beta)$-stable process \thinspace$\{L_{t}:\,t\geq0\}$
such that%
\begin{equation*}
U_{t}(\psi)\,:=\,\int_{(0,t]\times\R}\!M\!\left(  _{\!_{\!_{\,}}%
}\mathrm{d}(s,y)\right)  \psi(s,y)\,=\,L_{T(t)\,},\quad t\geq0,
\end{equation*}
where \thinspace$T(t)$\thinspace$:=$\thinspace$\int_{0}^{t}\mathrm{d}%
s\int_{\R}X_{s}(\mathrm{d}y)\,\bigl(\psi(s,y)\bigr)^{1+\beta}.$
\end{lemma}

\begin{proof}
Let us write It\^{o}'s formula for $\mathrm{e}^{-U_{t}(\psi)}:$%
\begin{align*}
&  \mathrm{e}^{-U_{t}(\psi)}-1\,=\,\text{local martingale}\\
&  \quad+\int_{0}^{t}\mathrm{d}s\ \mathrm{e}^{-U_{s}(\psi)}%
\int_{\R}X_{s}(\mathrm{d}y)\int_{0}^{\infty}n(\mathrm{d}r)
\,\bigl(\mathrm{e}^{-r\psi(s,y)}-1+r\,\psi(s,y)\bigr).\nonumber
\end{align*}
Define $\,\tau(t):=T^{-1}(t)$\thinspace\ and put $\,t^{\ast}:=\inf\!\left\{
_{\!_{\!_{\,}}}t:\,\tau(t)=\infty\right\}  .$ Then it is easy to get for every
$v>0$,
\begin{align*}
\mathrm{e}^{-vU_{\tau(t)}(\psi)}\,  &  =\,1+\int_{0}^{t}\mathrm{d}%
s\ \mathrm{e}^{-vU_{\tau(s)}(\psi)}\,\frac{X_{\tau(s)}\bigl(v^{1+\beta}%
\psi^{1+\beta}(s,\cdot)\bigr)}{X_{\tau(s)}\bigl(\psi^{1+\beta}(s,\cdot
)\bigr)}+\text{loc.\ mart.}\nonumber\\
&  =\,1+\int_{0}^{t}\mathrm{d}s\ \mathrm{e}^{-vU_{\tau(s)}(\psi)}\,v^{1+\beta
}+\text{loc. mart.,}\quad t\leq t^{\ast}.
\end{align*}
Since the local martingale is bounded, it is in fact a martingale. Let
$\,\tilde{L}$\thinspace\ denote a spectrally positive process of index
$1+\beta,$ independent of $X.$ Define%
\begin{equation*}
L_{t}\,:=\,\left\{  \!\!%
\begin{array}
[c]{ll}%
U_{\tau(t)}(\psi), & t\leq t^{\ast},\\
U_{\tau(t^{\ast})}(\psi)+\tilde{L}_{t-t^{\ast}\,}, & t>t^{\ast}\ \text{(if
}\,t^{\ast}<\infty).
\end{array}
\right.
\end{equation*}
Then we can easily get that $L$ satisfies the martingale problem (\ref{MP})
with $\kappa$ replaced by $1+\beta.$ Now by time change back we obtain%
\begin{equation*}
U_{t}(\psi)=\tilde{L}_{T(t)}=L_{T(t)\,},
\end{equation*}
finishing the proof.\hfill$\square$
\end{proof}
Having this result we may represent the increment $Z_t(x_1)-Z_t(x_2)$ as a difference
of two stable processes. More precisely, for every fixed pair $(x_1,x_2)$ there
exist spectrally positive stable processes $L^+$ and $L^-$ such that
\begin{equation}
\label{Levy.rep}
Z_t(x_1)-Z_t(x_2)=L^+_{T_+(x_1,x_2)}-L^-_{T_-(x_1,x_2)}
\end{equation}
where
\begin{equation}
\label{Levy.rep1}
T_\pm(x_1,x_2)=\int_0^t\mathrm{d}u\int_{\R}X_u(\mathrm{d}y)
\left((p_{t-u}(x_1-y)-p_{t-u}(x_2-y))^\pm\right)^{1+\beta}.
\end{equation}
It is clear from Lemma \ref{L9} that every jump $r\delta_{s,y}$ of the martingale
measure $M$ produces a jump of one of those stable processes:
\begin{itemize}
 \item If $p_{t-s}(x_1-y)>p_{t-s}(x_2-y)$ then $L^+$ has a jump of size
       $r(p_{t-u}(x_1-y)-p_{t-u}(x_2-y))$;
 \item If $p_{t-s}(x_1-y)<p_{t-s}(x_2-y)$ then $L^-$ has a jump of size
       $r(p_{t-u}(x_2-y)-p_{t-u}(x_1-y))$.     
\end{itemize}
Therefore, representation \eqref{Levy.rep} gives a handy tool for the
study of the influence of jumps of $M$ on the behavior of the increment
$Z_t(x_1)-Z_t(x_2)$. Moreover, it becomes clear that one needs to know good
estimates for the difference of kernels $p_{t-u}(x_1-y)-p_{t-u}(x_2-y)$
and for the tails of spectrally positive stable processes. 

For H\"older exponents $\eta>1$ we cannot use \eqref{Levy.rep}, since for exponents greater
than $1$ one has to subtract a polynomial correction. Instead of $Z(x_1)-Z(x_2)$ we shall
consider
\begin{align}
\label{def.Z.new}
\nonumber
Z_s(x_1,x_2)
&:=Z_s(x_1)-Z_s(x_2)-(x_1-x_2)\int_0^s\int_{\R}
M\left(\mathrm{d}(u,y)\right)\frac{\partial}{\partial y}p_{t-u}(x_2-y)\\
&=\int_0^s\int_{\R} M\left(\mathrm{d}(u,y)\right)q_{t-u}(x_1-y,x_2-y),
\quad 0\leq s\leq t,
\end{align}
where
\begin{equation}
\label{def_qq}
q_s(x,y):=p_s(x)-p_s(y)-(x-y)\frac{\partial}{\partial y}p_s(y).
\end{equation}
Here we may again apply Lemma \ref{L9} to obtain a representation for $Z(x_1,x_2)$ in terms
of difference of spectrally positive stable processes, similarly to \eqref{Levy.rep} which
gives the representation for $Z_s(x_1)-Z_s(x_2)$. The only difference to \eqref{Levy.rep} is 
that $p_{t-s}(x_1-y)-p_{t-s}(x_2-y)$ in \eqref{Levy.rep1} is replaced by $q_{t-s}(x_1-y,x_2-y)$.
\section{Some simple properties of $(2,d,\beta)$-superprocesses}
\label{sec:simple_properties}
In this section we collect some estimates on $(2,d,\beta)$-superprocesses which are needed
for the implementation of the program described in Section~\ref{sec:1}.

We start with a lemma where we give some  
left continuity properties of $(2,d,\beta)$-superprocess at fixed times, in dimensions $d<2/\beta$. 
\begin{lemma}
\label{lem:1}
Let $d<2/\beta$, and $B$ be an arbitrary open ball in $\R^d$. Then,
for a fixed $t>0$,
\begin{equation*}
\lim_{s\rightarrow t}X_{s}(B)\,=\,X_{t}(B),\quad \mathbf{P}-a.s.
\end{equation*}
\end{lemma}

\begin{proof}
Since $t$ is fixed, $X$ is continuous at $t$ with probability $1$. Therefore,
\begin{equation*}
X_{t}(B)\leq\liminf_{s\rightarrow t}X_{s}(B)
\leq\limsup_{s\rightarrow t}X_{s}(B)
\leq\limsup_{s\rightarrow t}X_{s}(\overline{B})
\leq X_{t}(\overline{B})
\end{equation*}
with $\overline{B}$ denoting the closure of $B.$ But since $X_{t}(\mathrm{d}x)$
is absolutely continuous with respect to Lebesgue measure, we have 
$X_{t}(B)=X_{t}(\overline{B}).$ Thus the proof is finished.
\hfill$\square$
\end{proof}

In the next lemma we give a simple test for  explosion of an integral involving
$\{X_s(B)\}_{s\leq t}$ whereas $B$ in an open ball in $\R^d$. 
\begin{lemma}
\label{L.explosion}
Let  $d<2/\beta$, and $B$ be an arbitrary open ball in $\R^d$.
Let $f:(0,t)\rightarrow(0,\infty)$ be measurable and assume that
\begin{equation*}
\int_{t-\delta}^{t}\mathrm{d}s\ f(t-s)=\infty\ \,\text{for all sufficiently
small }\,\delta\in(0,t).
\end{equation*}
Then for these $\delta$
\begin{equation*}
\int_{t-\delta}^{t}\mathrm{d}s\ X_{s}(B)f(t-s)=\infty
\quad\mathbf{P}\text{-a.s. on the event }\bigl\{X_{t}(B)>0\bigr\}.
\end{equation*}
\end{lemma}

\begin{proof}
Fix $\delta$ as in the lemma. Fix also $\omega$ such that $\,X_{t}(B)>0$
and $\,X_{s}(B)\rightarrow X_{t}(B)$ as $s\uparrow t$. For this $\omega$,
there is an $\varepsilon\in(0,\delta)$ such that $X_{s}(B)>\varepsilon$ 
for all $s\in(t-\varepsilon,t)$. Hence
\begin{equation*}
\int_{t-\delta}^{t}\mathrm{d}s\ X_{s}(B)f(t-s)
\geq\varepsilon\int_{t-\varepsilon}^{t}\mathrm{d}s f(t-s)=\infty,
\end{equation*}
and we are done.\hfill$\square$
\end{proof}

Now we will study the properties of $(2,1,\beta)$-superprocess in dimension
$d=1$. We start with moment estimates on the spatial increments of $Z_t$
defined in \eqref{def_Zt}.
Until the end of this section we consider the case of 
$$
d=1.
$$  
\begin{lemma}
\label{L.moment_bound}
Let $d=1$. {F}or each $q\in(1,1+\beta)$ and $\delta<\min\bigl\{1,(2-\beta)/(1+\beta)\bigr\}$,
\begin{equation*}
\mathbf{E}\bigl|Z_t(x_{1})-Z_t(x_{2})\bigr|^{q}\,\leq
\,C\,|x_{1}-x_{2}|^{\delta q},\quad x_{1},x_{2}\in\mathbb{R}.%
\end{equation*}
\end{lemma}

\begin{proof}
Applying (3.1) from \cite{MP03} with
$$
\phi(s,y)=p_{t-s}(x_1-y)-p_{t-s}(x_2-y),
$$
we get
\begin{align}
&  \mathbf{E}\bigl|Z_t(x_{1})-Z_t(x_{2})\bigr|^{q}\nonumber\\
&  \leq\,C\,\bigg[\Bigl(\int_{0}^{t}\mathrm{d}s\int_{\R}S_{s}\mu(\mathrm{d}y)
\bigl|p_{t-s}(x_{1}-y)-p_{t-s}(x_{2}-y)\bigr|^{\theta}\Bigr)^{\!q/\theta}\nonumber\\
&  \quad\ +\int_{0}^{t}\mathrm{d}s\int_{\R}S_{s}\mu(\mathrm{d}y)
\bigl|p_{t-s}(x_{1}-y)-p_{t-s}(x_{2}-y)\bigr|^{q}\bigg]. \label{L4.1}%
\end{align}
For every $\varepsilon\in(1,3)$,
\begin{align*}
&  \int_{0}^{t}\mathrm{d}s\int_{\R}S_{s}\mu(\mathrm{d}y)
\bigl|p_{t-s}(x_{1}-y)-p_{t-s}(x_{2}-y)\bigr|^{\varepsilon}\nonumber\\
& =\int_{\R}\mu(\mathrm{d}z)\int_{0}^{t}\mathrm{d}s
\int_{\R}\mathrm{d}y\ p_{s}(y-z)
\bigl|p_{t-s}(x_{1}-z)-p_{t-s}(x_{2}-z)\bigr|^{\varepsilon}\nonumber\\
&  =\int_{\R}\mu(\mathrm{d}z)
\int_{0}^{t}\mathrm{d}s\int_{\R}\mathrm{d}y\ p_{s}(y)
\bigl|p_{t-s}(x_{1}-z-y)-p_{t-s}(x_{2}-z-y)\bigr|^{\varepsilon}.
\end{align*}
Using Lemma~\ref{L.kernel2}, we get for every positive
$\delta<\min\bigl\{1,(3-\varepsilon)/\varepsilon\bigr\}$,
\begin{align*}
&  \int_{0}^{t}\mathrm{d}s\int_{\R}S_{s}\mu(\mathrm{d}y)
\bigl|p_{t-s}(x_{1}-y)-p_{t-s}(x_{2}-y)\bigr|^{\varepsilon}\\
&  \leq C|x_{1}-x_{2}|^{\delta\varepsilon}\int_{\R}\mu(\mathrm{d}z)
\Big(p_{t}\bigl((x_{1}-z)/2\bigr)+p_{t}\bigl((x_{2}-z)/2\bigr)\Big)
\leq C|x_{1}-x_{2}|^{\delta\varepsilon},
\end{align*}
since $\,\mu,t$\thinspace\ are fixed. Applying this bound to both summands at
the right hand side of \eqref{L4.1} finishes the proof of the lemma.
\hfill$\square$
\end{proof}

Bounds on moments of spatial increments of $Z_t$, from the previous lemma, clearly give the same bounds 
on spatial increments of $X_t$ itself. However, on top of this, they immediately give the bounds on 
the moments of the supremum of $X_t(\cdot)$ on compact spatial sets: this is done in the next lemma. 

\begin{lemma}
\label{L.fixed.1} 
Let $d=1$. If $K\subset \R$ is a compact and $q\in(1,1+\beta)$ then
\begin{equation*}
\mathbf{E}\Big(\sup_{x\in K}X_{t}(x)\Big)^{q}<\infty.
\end{equation*}
\end{lemma}
\begin{proof}
By Jensen's inequality, we may additionally assume that $q>1.$ It follows from
(\ref{rep.dens}) that%
\begin{equation*}
\Big(\sup_{x\in K}X_{t}(x)\Big)^{\!q}\,\leq\,4\,\biggl(\!\Big(\sup_{x\in K}%
\mu\!\ast\!p_{t}\,(x)\Big)^{\!q}+\sup_{x\in K}\bigl|Z_t(x)\bigr|^{q}\biggr)\!.
\end{equation*}
Clearly, the first term at the right hand side is finite. Furthermore,
according to Corollary~1.2 of Walsh \cite{bib:wal86}, Lemma \ref{L.moment_bound}
implies that
\begin{equation*}
\mathbf{E}\sup_{x\in K}\bigl|Z_t(x)\bigr|^{q}<\infty.
\end{equation*}
This completes the proof.\hfill$\square$
\end{proof}

From the above lemma one can immediately see that for any fixed $t$, $X_t(\cdot)$ is bounded on 
compacts. However, this is clearly not the case, if one start considering $X_s(x)$ as a function 
of $(s,x)$ with $s\leq t$. The reason is obvious: as we have discussed in the introduction, 
the measure-valued process $X_s(dx)$ has jumps in the form of atomic measures, and if 
$\langle X_t, 1\rangle>0$, the set of jump times is dense in $[0,t]$. However, it turns out that if one ``smooths''
a bit $X_s$ by taking its convolution with the heat kernel $p_{c(t-s)}(x-\cdot)$, then the resulting function of 
$(s,x)$  is 
a.s. bounded on compacts for $c$ large enough. This not obvious result is given in the next lemma.

\begin{lemma}
\label{L.fixed.2}
Let $d=1$. 
Fix a non-empty compact $K\subset\R$. Then
\begin{equation*}
V(K):=\sup_{0\leq s\leq t,\,x\in K}
S_{4(t-s)}X_{s}(x)<\infty\quad\mathbf{P}-a.s.
\end{equation*}
\end{lemma}

\begin{proof}
Assume that the statement of the lemma does not hold, i.e. there exists an
event $A$ of positive probability such that
$\sup_{0\leq s\leq t,\,x\in K}S_{4(t-s)}X_{s}\,(x)=\infty$ for every
$\omega\in A$. Let $n\geq1.$ Put
\[
\tau_{n}\,:=\,\left\{  \!%
\begin{array}
[c]{l}%
\inf\Big\{s<t:\text{ there exists }x\in K\text{ such that }S_{4(t-s)}X_{s}
\,(x)>n\Big\},\quad\omega\in A,\vspace{4pt}\\
t,\quad\omega\in A^{\mathrm{c}}.
\end{array}
\right.
\]
If $\,\omega\in A,$ choose $x_{n}=x_{n}(\omega)\in K$ such that
$\,S_{4(t-\tau_{n})} X_{\tau_{n}}(x_{n})>n,$ whereas if
$\,\omega\in A^{\mathrm{c}},$ take any $x_{n}=x_{n}(\omega)\in K.$
Using the strong Markov property gives
\begin{align}
&  \mathbf{E}S_{3(t-\tau_{n})}X_{t}\,(x_{n})
=\mathbf{EE}\bigl[S_{3(t-\tau_{n})}X_{t}\,(x_{n})\,\bigl|
\mathcal{F}_{\tau_{n}}\bigr]\label{56}\\[2pt]
&  =\mathbf{E}S_{3(t-\tau_{n})}
S_{(t-\tau_{n})}X_{\tau_{n}}(x_{n})=
\mathbf{E}S_{4(t-\tau_{n})}X_{\tau_{n}}(x_{n}).\nonumber
\end{align}
\textrm{F}rom the definition of $(\tau_{n},x_{n})$ we get
\begin{equation*}
\mathbf{E}S_{4(t-\tau_{n})}X_{\tau_{n}}(x_{n})
\geq n\mathbf{P}(A)\rightarrow\infty\ \,\text{as}\ \,n\uparrow\infty.
\end{equation*}
In order to get a contradiction, we want to prove boundedness in $n$ of the
expectation in (\ref{56}). Choosing a compact
$K_{1}\supset K$ satisfying 
$\mathrm{dist}\bigl(K,(K_{1})^{\mathrm{c}}\bigr)\geq1,$ we have
\begin{align*}
&  \mathbf{E}S_{3(t-\tau_{n})}X_{t}\,(x_{n})\\
&  =\mathbf{E}\int_{K_{1}}\mathrm{d}y\ X_{t}(y)\,p_{3(t-\tau_{n})}(x_{n}-y)
+\mathbf{E}\int_{(K_{1})^{\mathrm{c}}}\mathrm{d}y\ X_{t}(y)\,p_{3(t-\tau_{n})}(x_{n}-y)\\
&  \leq\mathbf{E}\sup_{y\in K_{1}}X_{t}(y)+
\mathbf{E}X_{t}(\R)\sup_{y\in(K_{1})^{\mathrm{c}},
\,x\in K,\,0\leq s\leq t\,}p_{3s}(x-y).
\end{align*}
By our choice of $\,K_{1}$\thinspace\ we obtain the bound
\begin{equation}
\mathbf{E}S_{3(t-\tau_{n})}X_{t}\,(x_{n})
\leq\,\mathbf{E}\sup_{y\in K_{1}}X_{t}(y)+C=C,
\end{equation}
the last step again by Lemma~\ref{L.fixed.1}. Altogether, (\ref{56}) is bounded in $n,$
and the proof is finished.
\hfill$\square$
\end{proof}

An easy application of the previous lemma is the following result.

\begin{lemma}
\label{n.L2}
Let $d=1$. 
Fix any non-empty bounded $K\subset\R$. Then
\begin{equation*}
W_{K}:=\sup_{(c,s,x): c\geq1,\,\,0\vee(t-c^{-2})\leq s<t,\,\,\\*x\in K 
                    }\frac
{X_{s}\bigl(B_{c\,(t-s)^{1/2}}(x)\bigr)}{c\,(t-s)^{1/2}}<\infty
\quad\mathbf{P}-a.s.
\end{equation*}
\end{lemma}

\begin{proof}
Every ball of radius $c\,(t-s)^{1/2}$ can be covered with at most $[c]+1$
balls of radius $(t-s)^{1/2}$. Therefore,
\begin{align*}
\sup_{(c,s,x): c\geq1,\,\,0\vee(t-c^{-1/2})\leq s<t,\,\,x\in K}\frac
{X_{s}\bigl(B_{c\,(t-s)^{1/2}}(x)\bigr)}{c\,(t-s)^{1/2}}\hspace{1cm}\\
\,\leq\,2\sup_{(s,x): 0<s\leq t,\,x\in K_{1}}\frac{X_{s}\bigl(B_{(t-s)^{1/2}%
}(x)\bigr)}{(t-s)^{1/2}}\,,
\end{align*}
where $K_{1}:=\bigl\{x:\ \mathrm{dist}(x,\overline{K})\leq1\bigr\}$
with $\overline{K}$ denoting the closure of $K$. (The restriction 
$s\geq t-c^{-1/2}$ is imposed to have all centers $x$ of the balls
$B_{(t-s)^{1/2}}(x)$ in $K_{1\,}.)$ We further note that
\begin{equation*}
S_{t-s}X_{s}\,(x)=
\int_{\R}\mathrm{d}y\ p_{t-s}(x-y)X_{s}(y)
\geq\int_{B_{(t-s)^{1/2}}(x)}\mathrm{d}y\ p_{t-s}(x-y)X_{s}(y).
\end{equation*}
Using the monotonicity and the scaling property of $p$, we get the bound
\begin{equation*}
S_{t-s}X_{s}\,(x)
\geq(t-s)^{-1/2}p_{1}(1)X_{s}\bigl(B_{(t-s)^{1/2}}(x)\bigr).
\end{equation*}
Consequently,
\begin{equation*}
\sup_{(s,x):0<s\leq t,\,x\in K_{1}}\frac{X_{s}\bigl(B_{(t-s)^{1/2}}%
(x)\bigr)}{(t-s)^{1/2}}
\leq\frac{1}{p_{1}(1)}\sup_{(s,x):0<s\leq t,\,x\in K_{1}}S_{t-s}X_{s}\,(x).
\end{equation*}
It was proved in Lemma \ref{L.fixed.2} that the random variable at
the right hand side is finite. Thus, the lemma is proved.
\hfill$\square$
\end{proof}

The boundedness of the smoothed density will play a crucial role in the analysis
of the time changes $T_\pm(x_1,x_2)$ described in the previous section, see
\eqref{Levy.rep}--\eqref{def_qq} and discussion there. 
The next lemma provides necessary tools to obtain poitwise upper bounds
for $T_\pm(x_1,x_2)$: by taking $\theta=1+\beta$ in \eqref{L.fixed.3.1} and
\eqref{L.fixed.3.2} below we get estimates for $T_\pm(x_1,x_2)$.

\begin{lemma}
\label{L.fixed.3} 
Let $d=1$. 
Fix\/ $\,\theta\in\lbrack1,3)$, $\delta\in\lbrack0,1]$ with
$\delta<(3-\theta)/\theta,$ and a non-empty compact $K\subset\R$. Then
\begin{gather}
\int_{0}^{t}\mathrm{d}s\int_{\R}X_{s}(\mathrm{d}y)
\bigl|p_{t-s}(x_{1}-y)-p_{t-s}(x_{2}-y)\bigr|^{\theta}\,\nonumber\\
\leq C V|x_{1}-x_{2}|^{\delta\theta},\quad x_{1},x_{2}\in
K,\quad \mathbf{P}-a.s.,\label{L.fixed.3.1}
\end{gather}
with $V=V(K)$ from Lemma~\ref{L.fixed.2}.\\
Moreover, for every $\theta\in[1,2)$ and $\delta\in(0,(3-2\theta)/\theta]$,
\begin{gather}
\int_{0}^{t}\mathrm{d}s\int_{\R}X_{s}(\mathrm{d}y)
\bigl|p_{t-s}(x_{1}-y)-p_{t-s}(x_{2}-y)-(x_1-x_2)\frac{\partial}{\partial x_2}p_{t-s}(x_{2}-y)\bigr|^\theta\,\nonumber\\
\leq C V|x_{1}-x_{2}|^{1+\delta},\quad x_{1},x_{2}\in
K,\quad \mathbf{P}-a.s.\label{L.fixed.3.2}
\end{gather}
\end{lemma}

\begin{proof}
Using \eqref{L.kernel1.1} gives
\begin{align*}
&\int_{0}^{t}\mathrm{d}s\int_{\R}X_{s}(\mathrm{d}y)
\bigl|p_{t-s}(x_{1}-y)-p_{t-s}(x_{2}-y)\bigr|^{\theta}
\leq C |x_{1}-x_{2}|^{\delta\theta}\times\\
& \times\int_{0}^{t}\mathrm{d}s\ (t-s)^{-(\delta\theta+\theta-1)/2}%
\int_{\R}X_{s}(\mathrm{d}y)
\Big(p_{t-s}\bigl((x_{1}-y)/2\bigr)+p_{t-s}\bigl((x_{2}-y)/2\bigr)\Big),
\end{align*}
uniformly in $x_{1},x_{2}\in\R$. Recalling the scaling
property of the kernel $p$, we get
\begin{align*}
&\int_{0}^{t}\mathrm{d}s\int_{\R}X_{s}(\mathrm{d}y)
\bigl|p_{t-s}(x_{1}-y)-p_{t-s}(x_{2}-y)\bigr|^{\theta}\\
&\leq C|x_{1}-x_{2}|^{\delta\theta}
\int_{0}^{t}\mathrm{d}s\ (t-s)^{-(\delta\theta+\theta-1)/2}
\Bigl(S_{4(t-s)}X_{s}(x_{1})+S_{4(t-s)}X_{s}(x_{2})\Bigr).
\end{align*}
We complete the proof of \eqref{L.fixed.3.1} by applying Lemma~\ref{L.fixed.2}.
To derive \eqref{L.fixed.3.2} it suffices to replace
\eqref{L.kernel1.1} by \eqref{L.kernel1.4} 
in the computations we used to prove \eqref{L.fixed.3.1}.
\hfill$\square$
\end{proof}

In Lemmas \ref{L.fixed.2} and \ref{n.L2} we have obtained uniform on compact sets
upper bounds for the ``smoothed'' densities. Now we turn to the analysis
of this smoothed density near a fixed spatial point. Without loss of
generality we choose fixed point $x=0$ in the next lemma.
\begin{lemma}
\label{n.L1} Let $d=1$.  For all $c,\theta>0$,
\begin{equation*}
\mathbf{P}\!\left(  X_{t}(0)>\theta,\,\ \liminf_{s\uparrow t}
S_{t-s}X_{s}\bigl(c\,(t-s)^{1/2}\bigr)\leq\theta\right)  =\,0.
\end{equation*}
\end{lemma}

\begin{proof}
For brevity, set
\begin{equation*}
A\,:=\left\{  \liminf_{s\uparrow t}S_{t-s}X_{s}
\bigl(c\,(t-s)^{1/2}\bigr)\leq\theta\right\}
\end{equation*}
and for $n>1/t$ define the stopping times
\begin{equation*}
\tau_{n}\,:=\,\left\{  \!%
\begin{array}
[c]{l}%
\inf\Big\{s\in(t-1/n,t):S_{t-s}X_{s}\bigl(c\,(t-s)^{1/2}\bigr)
\leq\theta+1/n\Big\},\quad\omega\in A,\vspace{4pt}\\
t,\quad\omega\in A^{\mathrm{c}}.
\end{array}
\right.
\end{equation*}
Define also
\begin{equation*}
x_{n}:=c\,(t-\tau_{n})^{1/2}.
\end{equation*}
Then, using the strong Markov property, we get
\begin{equation}
\mathbf{E}\!\left[  X_{t}(x_{n})\,\big|\,\mathcal{F}_{\tau_{n}}\right]
=S_{t-\tau_{n}}X_{\tau_{n}}(x_{n})=X_{t}(0)\mathsf{1}_{A^{\mathrm{c}%
}}+S_{t-\tau_{n}}X_{\tau_{n}}(x_{n})\mathsf{1}_{A\,}. \label{n1}%
\end{equation}
We next note that $x_{n}\rightarrow0$ almost surely as $n\uparrow\infty$. This
implies, in view of the continuity of $X_{t}$ at zero, that $X_{t}%
(x_{n})\rightarrow X_{t}(0)$ almost surely. Recalling that 
$$
\mathbf{E}\sup_{|x|\leq1}X_{t}(x)<\infty,
$$
in view of Corollary~2.8 of
\cite{FMW10}, we conclude that
\begin{equation*}
X_{t}(x_{n})\,\underset{n\uparrow\infty}{\longrightarrow}\,X_{t}%
(0)\quad\text{in }\mathcal{L}_{1\,}.
\end{equation*}
This, in its turn, implies that
\begin{equation}
\mathbf{E}\left[  X_{t}(x_{n})|\mathcal{F}_{\tau_{n}}\right]  -\mathbf{E}%
\left[  X_{t}(0)|\mathcal{F}_{\tau_{n}}\right]  \rightarrow0\quad\text{in
}\mathcal{L}_{1\,}. \label{n2}%
\end{equation}
Furthermore, it follows from the well known Levy theorem on convergence of
conditional expectations that
\begin{equation*}
\mathbf{E}\!\left[  X_{t}(0)\,\big|\,\mathcal{F}_{\tau_{n}}\right]
\,\underset{n\uparrow\infty}{\longrightarrow}\,\mathbf{E}\!\left[
X_{t}(0)\big|\,\mathcal{F}_{\infty}\right]  \quad\text{in }\mathcal{L}_{1\,},
\end{equation*}
where $\mathcal{F}_{\infty}:=\sigma\left(  \cup_{n>1/t}\mathcal{F}_{\tau_{n}%
}\right)  $.

Noting that $\tau_{n}\uparrow t$, we conclude that
\begin{equation*}
\mathcal{F}_{t-}\subseteq\mathcal{F}_{\infty}\subseteq\mathcal{F}_{t\,}.
\end{equation*}
Since $X_{\cdot}(0)$ is continuous at fixed $t$ a.s., we have
$X_{t}(0)=\mathbf{E}\!\left[  X_{t}(0)\,\big|\,\mathcal{F}_{t-}\right]  $
almost surely. Consequently, $\mathbf{E}\!\left[  X_{t}(0)\,\big|\,\mathcal{F}%
_{\infty}\right]  =X_{t}(0)$ almost surely, and we get, as a result,
\begin{equation}
\mathbf{E}\!\left[  X_{t}(0)\,\big|\,\mathcal{F}_{\tau_{n}}\right]
\,\underset{n\uparrow\infty}{\longrightarrow}\,X_{t}(0)\quad\text{in
}\mathcal{L}_{1\,}.\label{n3}%
\end{equation}
Combining (\ref{n2}) and (\ref{n3}), we have
\begin{equation*}
\mathbf{E}\!\left[  X_{t}(x_{n})\,\big|\,\mathcal{F}_{\tau_{n}}\right]
\,\underset{n\uparrow\infty}{\longrightarrow}\,X_{t}(0)\quad\text{in
}\mathcal{L}_{1\,}.
\end{equation*}
\textrm{F}rom this convergence and from (\ref{n1}) we finally get
\begin{equation*}
\mathbf{E}\Big[\mathsf{1}_{A}\left\vert X_{t}(0)-
S_{t-\tau_{n}}X_{\tau_{n}}(x_{n})\right\vert \Big]\,\underset{n\uparrow\infty
}{\longrightarrow}\,0.
\end{equation*}
Since $\,S_{t-\tau_{n}} X_{\tau_{n}}(x_{n})\leq\theta+1/n$%
\thinspace\ on $A$, for all $n>1/t$, the latter convergence implies that
$X_{t}(0)\leq\theta$ almost surely on the event $A$. Thus, the proof is finished.
\hfill$\square$
\end{proof}

\section{Analysis of jumps of superprocesses.}
\label{sec:jumps}
This section is devoted to the analysis of jumps of $(2,d,\beta)$-superprocesses.
The results of this section will be crucial for proofs of main theorems, since
regularity properties of $(2,d,\beta)$-superprocesses depend heavily on presence
and intensity of big jumps at certain locations.

Let $\,\Delta X_{s}:=X_{s}-X_{s-}$\thinspace\ denote the jumps of the
measure-valued process $\,X$. Also let $|\Delta X_s|= \langle \Delta X_s\,,1\rangle$ be the 
size of the jump, and with some abuse of notation $|\Delta X_s(x)|$ denotes the size of jump at a point $(s,x)$. 

The results of the section are a bit technical, however let us explain briefly the main bounds we are going to obtain. 
Recall that $t>0$ is fixed. First we would like to verify that the largest jump at the proximity of time $t$ is of the
order 
\begin{equation} 
 |\Delta X_s| \sim (t-s)^{1/(1+\beta)}, 
\end{equation}
for $s<t$. 
The exact lower and upper bounds are given in Lemmas~\ref{lem:3} and~\ref{L8}. Note that the jump of order 
$(t-s)^{1/(1+\beta)}$ happens at the some ``random'' spatial point. Whenever one asks about the size of 
the maximal jump at the proximity of a given time-space point $(t,x)$ one gets other estimates. 
For a moment fix a spatial point $x=0$. It turns out that, if $X_t(0)>0$, then   the size of the maximal jump at the proximity of the 
time-space point $(t,0)$ is of order 
\begin{equation} 
 |\Delta X_s(x)| \sim |(t-s)x|^{1/(1+\beta)}, 
\end{equation}
for  $x$ close to $0$. This is shown in Lemma~\ref{6.00}, Corollary~\ref{6.1} and Lemma~\ref{n.L3}.

We start with the lemma where we show that on any open subset of $\R^d$
big jumps, which are  ``a bit`` larger  than $(t-s)^{1/(1+\beta)}$,  will occur with probability one.
In fact, we give a lower bound on the size of the largest jump. 
 
\begin{lemma}
\label{lem:3} 
Let $d<2/\beta$ and $B$ be an open ball in $\R^d$.
For each $\varepsilon\in(0,t)$,
\begin{equation*}
\mathbf{P}\left(\Delta X_{s}(B)>
(t-s)^{\frac{1}{1+\beta}}\log^{\frac{1}{1+\beta}}\Big(\frac{1}{t-s}\Big)
\text{ for some }s\in(t-\varepsilon,t)\Big|X_t(B)>0\right)=1
\end{equation*}
\end{lemma}

\begin{proof}
It suffices to show that
\begin{equation}
\mathbf{P}\left(\Delta X_{s}(B)\leq
(t-s)^{\frac{1}{1+\beta}}\log^{\frac{1}{1+\beta}}\Big(\frac{1}{t-s}\Big)
\text{ for all }s\in(t-\varepsilon,t),X_t(B)>0\right)=0.
\label{equt:11}%
\end{equation}
For $u\in(0,\varepsilon]$ define
$$
\Pi_{u}:=N\biggl((s,x,r):\ s\in(t-\varepsilon,\,t-\varepsilon+u),\ x\in
B,\ r>(t-s)^{\frac{1}{1+\beta}}\log^{\frac{1}{1+\beta}}\Big(\frac{1}{t-s}\Big)\biggr),
$$
with the random measure $N$ introduced in Lemma~\ref{L.mart.dec}. Then
\begin{equation}
\left\{  \Delta X_{s}(B)\leq
(t-s)^{\frac{1}{1+\beta}}\log^{\frac{1}{1+\beta}}\Big(\frac{1}{t-s}\Big)\right\}  \,=\,\{\Pi_{\varepsilon}=0\}. \label{equt:10}%
\end{equation}
{F}rom a classical time change result for counting processes (see e.g. 
Theorem 10.33 in Jacod \cite{Jacod79}), we conclude that there exists a standard
Poisson process $A=\bigl\{A(v):\,v\geq0\bigr\}$ such
that
\begin{align*}
\Pi_{u}&=A\biggl(\int_{t-\varepsilon}^{t-\varepsilon+u}%
\mathrm{d}s\ X_{s}(B)\int_{(t-s)^{\frac{1}{1+\beta}}\log^{\frac{1}{1+\beta}}\bigl(\frac
{1}{t-s}\bigr)}^{\infty}n(\mathrm{d}r)\biggr)\\
&=A\biggl(\frac{c_\beta}{1+\beta}\int_{t-\varepsilon}^{t-\varepsilon
+u}\mathrm{d}s\ X_{s}(B)\,\frac{1}{(t-s)\log\bigl(\frac{1}{t-s}\bigr)}%
\biggr),
\end{align*}
where $c_\beta:=\frac{\beta(\beta+1)}{\Gamma(1-\beta)}$. (By this definition,
$n(dr)=c_\beta r^{-2-\beta}dr$.)
Then
\begin{align*}
\mathbf{P}\bigl(\Pi_{\varepsilon}=0,\,X_{t}(B)>0\bigr)
\leq\mathbf{P}\biggl(\int_{t-\varepsilon}^{t}\mathrm{d}s\ X_{s}%
(B)\frac{1}{(t-s)\log\bigl(\frac{1}{t-s}\bigr)}<\infty,X_{t}(B)>0\biggr).
\end{align*}
It is easy to check that
\begin{equation*}
\int_{t-\delta}^{t}\mathrm{d}s\ \frac{1}{(t-s)\log\bigl(\frac{1}{t-s}%
\bigr)}\,=\,\infty\,\ \text{for all}\,\ \delta\in(0,\varepsilon).
\end{equation*}
Therefore, by Lemma~\ref{L.explosion},
\begin{equation*}
\int_{t-\varepsilon}^{t}\mathrm{d}s\ X_{s}(B)\,\frac{1}{(t-s)\log
\bigl(\frac{1}{t-s}\bigr)}\,=\,\infty\,\ \ \text{on}\,\ \bigl\{X_{t}%
(B)>0\bigr\}.
\end{equation*}
As a result we have $\mathbf{P}\bigl(\Pi_{\varepsilon}=0,\,X_{t}(B)>0\bigr)=0$.
Combining this with (\ref{equt:10}) we get (\ref{equt:11}).
\hfill$\square$
\end{proof}

The next result complements the previous lemma: it gives with probability
close to one an upper bound for the sizes of jumps.
\begin{lemma}
\label{L8} 
Let $d=1$. Let \thinspace$\varepsilon>0$ and $\,\gamma\in\bigl(0,(1+\beta)^{-1}\bigr).$
There exists a constant $c_{(\ref{inL8})}=c_{(\ref{inL8})}(\varepsilon,\gamma)$ such that
\begin{equation}
\mathbf{P}\Big(|\Delta X_{s}|>c_{(\ref{inL8})}\,(t-s)^{(1+\beta)^{-1}-\gamma
}\text{ for some }s<t\Big)\leq\,\varepsilon. \label{inL8}%
\end{equation}
\end{lemma}

\begin{proof}
Recall the random measure $N$ from Lemma~\ref{L.mart.dec}(a). For any $c>0,$
set
\begin{equation}
Y_{0}\,:=\,N\Bigl([0,2^{-1}t)\times\R^d\times\bigl(c\,2^{-\lambda
}t^{\lambda},\infty\bigr)\Bigr),
\end{equation}%
\begin{equation}
Y_{n}\,:=\,N\Bigl(\bigl[(1-2^{-n})t,(1-2^{-n-1})t\bigr)\times\R^d
\times\bigl(c\,2^{-\lambda(n+1)}t^{\lambda},\infty\bigr)\Bigr),\ \,n\geq1,
\end{equation}
where $\lambda:=(1+\beta)^{-1}-\gamma$. It is easy to see that
\begin{equation}
\mathbf{P}\Big(|\Delta X_{s}|>c\,(t-s)^{\lambda}\text{ for some }%
s<t\Big)\,\leq\,\mathbf{P}\Bigl(\sum_{n=0}^{\infty}Y_{n}\geq1\Bigr)\,\leq
\,\sum_{n=0}^{\infty}\mathbf{E}Y_{n\,}, \label{L12.1}%
\end{equation}
where in the last step we have used the classical Markov inequality.
\textrm{F}rom the formula for the compensator $\hat{N}$ of $N$ in
Lemma~\ref{L.mart.dec}(b),
\begin{equation}
\mathbf{E}Y_{n}\,=\,c_\beta\int_{(1-2^{-n})t}^{(1-2^{-n-1})t}\mathrm{d}%
s\ \mathbf{E}X_{s}(\R^d)\int_{c\,2^{-\lambda(n+1)}t^{\lambda}}^{\infty
}\mathrm{d}r\ r^{-2-\beta},\quad n\geq1.
\end{equation}
Now%
\begin{equation}
\mathbf{E}X_{s}(\R^d)\,=\,X_{0}(\R^d)=:\,c_{(\ref{65})}. \label{65}%
\end{equation}
Consequently,
\begin{equation}
\mathbf{E}Y_{n}\,\leq\,\frac{c_\beta}{1+\beta}\,c_{(\ref{65})}c^{-1-\beta
}\,2^{-(n+1)\gamma(1+\beta)}\,t^{\gamma(1+\beta)}. \label{L12.2}%
\end{equation}
Analogous calculations show that (\ref{L12.2}) remains valid also in the case
$n=0$. Therefore,%
\begin{align}
\sum_{n=0}^{\infty}\mathbf{E}Y_{n}\,  &  \leq\,\frac{c_\beta}{1+\beta
}\,c_{(\ref{65})}c^{-1-\beta}\,t^{\gamma(1+\beta)}\sum_{n=0}^{\infty
}2^{-(n+1)\gamma(1+\beta)}\nonumber\\
\,  &  =\,\frac{c_\beta}{1+\beta}\,c_{(\ref{65})}c^{-1-\beta}\,t^{\gamma
(1+\beta)}\,\frac{2^{-\gamma(1+\beta)}}{1-2^{-\gamma(1+\beta)}}\,.
\label{L12.3}%
\end{align}
Choosing $c=c_{(\ref{inL8})}$ such that the expression in~(\ref{L12.3}) equals
$\varepsilon,$ and combining with~(\ref{L12.1}), the proof is complete.
\hfill$\square$
\end{proof}

Put
\begin{equation}
f_{s,x}\,:=\,\log\bigl((t-s)^{-1}\bigr)\,\mathsf{1}_{\{x\neq0\}}%
\log\bigl(|x|^{-1}\bigr). \label{def_f}%
\end{equation}
In the following lemma and corollary we obatin suitable upper bounds
for maximal jumps which occur close to $0$.
\begin{lemma}
\label{6.00}
Let $d=1$.
Fix $\,\,X_{0}=\mu\in\mathcal{M}_{\mathrm{f}}\backslash\{0\}.$ 
Let $\varepsilon>0$ and $q>0$. Then there exists a constant 
$c_{(\ref{in6.1a})}=c_{(\ref{in6.1a})}(\varepsilon,q)$ such that%
\begin{equation}
\mathbf{P}\Big(\Delta X_{s}(x)>c_{(\ref{in6.1a})}\bigl((t-s)|x|\bigr)^{\frac
{1}{1+\beta}}(f_{s,x})^{\ell}\text{ for some }s<t,\text{\ }\ x\in
B_{1/\mathrm{e}}(0)\Big)\!\leq\varepsilon, \label{in6.1a}%
\end{equation}
where
\begin{equation}
\ell\,:=\,\frac{1}{1+\beta}+q. \label{def.lambda1}%
\end{equation}
\end{lemma}

\begin{proof}
For any $c>0$ (later to be specialized to some $c_{(\ref{6.4})})$ set
\[
Y\,:=\,N\Bigl((s,x,r):\;(s,x)\in(0,t)\times B_{1/\mathrm{e}}(0),\ r\geq
c\,\bigl((t-s)|x|\bigr)^{1/(1+\beta)}(f_{s,x})^{\ell}\Bigr),
\]
Clearly,%
\begin{equation}%
\begin{array}
[c]{c}%
\displaystyle
\mathbf{P}\Big(\Delta X_{s}(x)>c\,\bigl((t-s)|x|\bigr)^{1/(1+\beta)}%
(f_{s,x})^{\ell}\text{ for some }s<t\ \text{and}\ x\in B_{1/\mathrm{e}%
}(0)\Big)\vspace{6pt}\\
=\,\mathbf{P}(Y\geq1)\,\leq\,\mathbf{E}Y,
\end{array}
\label{6.20}%
\end{equation}
where in the last step we have used the classical Markov inequality.
\textrm{F}rom (\ref{decomp}),
\begin{align}
\mathbf{E}Y\,  &  =\,c_\beta\,\mathbf{E}\int_{0}^{t}\mathrm{d}s\int
_{\R}X_{s}(\mathrm{d}x)\,\mathsf{1}_{B_{1/\mathrm{e}}(0)}%
(x)\int_{c\,\bigl((t-s)|x|\bigr)^{1/(1+\beta)}(f_{s,x})^{\ell}}^{\infty
}\mathrm{d}r\ r^{-2-\beta}\nonumber\\
&  =\,c_\beta\,\frac{c^{-1-\beta}}{1+\beta}\int_{0}^{t}\mathrm{d}%
s\ (t-s)^{-1}\log^{-1-q(1+\beta)}\bigl((t-s)^{-1}\bigr)\\
&  \hspace{1cm}\times\int_{\R}\mathbf{E}X_{s}(\mathrm{d}%
x)\,\mathsf{1}_{B_{1/\mathrm{e}}(0)}(x)\,|x|^{-1}\log^{-1-q(1+\beta
)}\bigl(|x|^{-1}\bigr).\nonumber
\end{align}
Now, writing $C$ for a generic constant (which may change from place to
place),%
\begin{align}
&  \int_{\R}\,\mathbf{E}X_{s}(\mathrm{d}x)\,\mathsf{1}%
_{B_{1/\mathrm{e}}(0)}(x)\,|x|^{-1}\log^{-1-q(1+\beta)}\bigl(|x|^{-1}%
\bigr)\,\nonumber\\
&  \leq\ \int_{\R}\mu(\mathrm{d}y)\!\int_{\R%
}\mathrm{d}x\ p_{s}(x-y)\,\mathsf{1}_{B_{1/\mathrm{e}}(0)}%
(x)\,|x|^{-1}\log^{-1-q(1+\beta)}\bigl(|x|^{-1}\bigr)\nonumber\\
&  \leq\ C\,\mu(\R)\,s^{-1/2}\!\int_{\R}\mathrm{d}%
x\ \mathsf{1}_{B_{1/\mathrm{e}}(0)}(x)\,|x|^{-1}\log^{-1-q(1+\beta
)}\bigl(|x|^{-1}\bigr)\nonumber\\
&  =:\ c_{(\ref{6.3})}s^{-1/2}, \label{6.3}%
\end{align}
where $\,c_{(\ref{6.3})}=c_{(\ref{6.3})}(q)$\thinspace\ (recall that $t$ is
fixed). Consequently,%
\begin{align}
\mathbf{E}Y\,  &  \leq\,c_\beta\,c_{(\ref{6.3})}\,c^{-1-\beta}\int_{0}%
^{t}\mathrm{d}s\text{\thinspace}s^{-1/2}\,(t-s)^{-1}\log^{-1-q(1+\beta
)}\bigl((t-s)^{-1}\bigr)\,\nonumber\\[1pt]
&  =:\,c_{(\ref{6.4})}\,c^{-1-\beta} \label{6.4}%
\end{align}
with $\,c_{(\ref{6.4})}=c_{(\ref{6.4})}(q).$\thinspace\ Choose now $c\,\ $such
that the latter expression equals $\varepsilon$ and write $c_{(\ref{6.4})}$
instead of $\,c.$ Recalling (\ref{6.20}), the proof is complete.
\end{proof}

Since $\,\sup_{0<y<1}y^{\gamma}\log^{\ell}\frac{1}{y}<\infty$ for
every $\gamma>0$, the following corollary is immediate from Lemma~\ref{6.00}.

\begin{corollary}
\label{6.1}Fix $\,\,X_{0}=\mu\in\mathcal{M}%
_{\mathrm{f}}\backslash\{0\}.$ Let $d=1$. Let \thinspace$\varepsilon>0$\thinspace\ and
$\,\gamma\in\bigl(0,(1+\beta)^{-1}\bigr).$\thinspace\ There exists a constant
\thinspace$c_{(\ref{in6.1})}=c_{(\ref{in6.1})}(\varepsilon,\gamma)$ such that
\begin{equation}
\mathbf{P}\Big(\Delta X_{s}(x)>c_{(\ref{in6.1})}%
\,\bigl((t-s)|x|\bigr)^{\frac{1}{1+\beta}-\gamma}\text{ for some }s<t\ \text{and}\ x\in
B_{2}(0)\Big)\leq\,\varepsilon. \label{in6.1}%
\end{equation}
\end{corollary}

Introduce the event
\begin{equation*}
D_{\theta}:=\left\{  X_{t}(0)>\theta,\ \sup_{0<s\leq t}X_{s}(\R%
)\leq\theta^{-1},\ W_{B_{3}}(0)\leq\theta^{-1}\right\}  \!.
\end{equation*}
In the next lemma we obtain a lower bound on a jump which occur
close to time $t$ and to the spatial point $z=0$.

\begin{lemma}
\label{n.L3} 
Let $d=1$. For each $\theta>0$ there exists a constant 
$c_{(\ref{nn.L3})}=c_{(\ref{nn.L3})}(\theta)\geq1$ such that
\begin{align}
\label{nn.L3}
\mathbf{P}\Bigl( &\Delta X_{s}(y) >Q\,\bigl(y\,(t-s)\bigr)^{1/(1+\beta)}\log^{1/(1+\beta)}\bigl((t-s)^{-1}
\bigr)\\
\nonumber
&\text{for some }s\in(t-\varepsilon,t)\text{ and }\frac{c_{(\ref{nn.L3})}}{2}\,(t-s)^{1/2}\leq y\leq\frac{3c_{(\ref{nn.L3})}}{2}\,(t-s)^{1/2}
\big|\,D_{\theta}\Bigr)  =\,1
\end{align}
for all $\varepsilon\in(0,t\wedge1/8),\,\ Q>0$.
\end{lemma}

\begin{proof}
Analogously to the proof of Lemma~\ref{lem:3}, to
show  that the number of jumps is greater than zero almost
surely on some event, it is enough to show the divergence
of a certain integral on that event or even on a bigger
one. Specifically here, it suffices to verify that there
exists $c=c_{\eqref{nn.L3}}$ such that
\begin{equation*}
I_{\varepsilon,c}\,:=\int_{t-\varepsilon}^{t}\frac{\mathrm{d}s}{(t-s)\log
\bigl((t-s)^{-1}\bigr)}\int_{\frac{c}{2}(t-s)^{1/2}}^{\frac{3c}%
{2}(t-s)^{1/2}}\mathrm{d}y\ y^{-1}X_{s}(y)=\infty
\end{equation*}
almost surely on the event $D_\theta$.

The mapping $\varepsilon\mapsto I_{\varepsilon,c}$ is nonincreasing.
Therefore, we shall additionally assume, without loss of generality,
that $\varepsilon\leq c^{-1/2}$ and this in turn implies that
$c(t-s)^{1/2}\leq1$ for all $s\in(t-\varepsilon, t)$. So, in what
follows, in the proof of the lemma we will assume without loss of
generality that given $c$, we choose $\varepsilon$ so that
$$
c(t-s)^{1/2}\leq1,\;\;\; \forall \;s\in(t-\varepsilon,t).
$$

Since $y\leq\frac{3c}{2}\,(t-s)^{1/2}$ and $p_{s}(x)\leq p_{s}(0)$
for all $x\in\R$, we have
\begin{align*}
I_{\varepsilon,c}\,\geq\frac{2}{3c}\int_{t-\varepsilon}^{t} &  \frac
{\mathrm{d}s}{(t-s)^{3/2}\log\bigl((t-s)^{-1}\bigr)}\\
&  \times\int_{\frac{c}{2}(t-s)^{1/2}}^{\frac{3c}{2}(t-s)^{1/2}%
}\!\!\mathrm{d}y\ \frac{p_{t-s}\bigl(c\,(t-s)^{1/2}-y\bigr)}{p_{t-s}(0)}\,X_{s}(y).
\end{align*}
Then, using the scaling property of the kernel $p$, we obtain
\begin{align}
I_{\varepsilon,c}\,\geq\frac{2}{3c\,p_{1}(0)}\int_{t-\varepsilon}^{t}
&\frac{\mathrm{d}s}{(t-s)\log\bigl((t-s)^{-1}\bigr)}\,\biggl(S_{t-s}%
X_{s}\,\bigl(c\,(t-s)^{1/2}\bigr)\nonumber\label{n4}\\
- &  \int_{\left\vert _{\!_{\!_{\,}}}y-c\,(t-s)^{1/2}\right\vert
>\frac{c}{2}(t-s)^{1/2}}\mathrm{d}y\ p_{t-s}%
\bigl(c\,(t-s)^{1/2}-y\bigr)X_{s}(y)\!\biggr)\!.
\end{align}
Since we are in dimension one, if
\begin{align}
\nonumber
y\in \widetilde{D}_{s,j}:=&\left\{z:\; c\left(  \frac{1}{2}+j\right)  (t-s)^{1/2}<\left\vert _{\!_{\!_{\,}}%
}z-c\,(t-s)^{1/2}\right\vert\right.\\
&\left.<c\left(  2+\frac{1}{2}+j\right)
(t-s)^{1/2}\right\},\label{1dim}%
\end{align}
then%
\begin{align*}
&  p_{t-s}\bigl(c\,(t-s)^{1/2}-y\bigr)
\leq p_{t-s}\bigl(c\,(j+1/2)(t-s)^{1/2}\bigr)\\
&  =(t-s)^{-1/2}p_{1}\bigl(c\,(j+1/2)\bigr)
=(2\pi)^{-1/2} (t-s)^{-1/2}e^{-c^2(1/2+j)^{2}/2}.
\end{align*}
\textrm{F}rom this bound we conclude that
\begin{align*}
&\int_{\left\vert _{\!_{\!_{\,}}}y-c\,(t-s)^{1/2}\right\vert >\frac{c}%
{2}(t-s)^{1/2}}\mathrm{d}y\ p_{t-s}\bigl(c\,(t-s)^{1/2}-y\bigr)
\mathsf{1}_{B_{2}(0)}(y)X_{s}(y)\\
&\hspace{1cm}\leq (2\pi)^{-1/2}(t-s)^{-1/2} \sum_{j=0}^{\infty}
e^{-c^2(1/2+j)^{2}/2}
 \int_{\widetilde{D}_{s,j}}\mathrm{d}y \mathsf{1}_{B_{2}(0)}(y)X_{s}(y).
\end{align*}
Now recall again that the spatial dimension equals to one and hence for
any $j\geq 0$  the set $\widetilde{D}_{s,j}$ in (\ref{1dim}) is the union
of two  balls of radius $c(t-s)^{1/2}$. If furthermore 
$\widetilde{D}_{s,j}\cap B_{2}(0)\neq\emptyset$, then, in view of the
assumption $c(t-s)^{1/2}\leq1$, the centers of those balls lie in $B_{3}(0)$.
Therefore, we can apply Lemma~\ref{n.L2} to bound the integral
$\int_{\widetilde{D}_{s,j}}\mathrm{d}y \mathsf{1}_{B_{2}(0)}(y)X_{s}(y)$ by 
 $2 c(t-s)^{1/2} W_{B_{3}(0)}$ and obtain
\begin{gather}
\int_{\left\vert _{\!_{\!_{\,}}}y-c\,(t-s)^{1/2}\right\vert >\frac{c}%
{2}(t-s)^{1/2}}\mathrm{d}y\ p_{t-s}\bigl(c\,(t-s)^{1/2}-y\bigr)
\mathsf{1}_{B_{2}(0)}(y)X_{s}(y)\nonumber
\\
\leq\frac{2c}{(2\pi)^{1/2}}W_{B_{3}(0)}
\sum_{j=0}^{\infty}e^{-c^2(1/2+j)^{2}/2}
\leq CW_{B_{3}(0)}c^{-2}.\label{n5}%
\end{gather}
Furthermore, if $|y|\geq2$ and $(t-s)\leq c^{-2}$, then
\begin{align*}
p_{t-s}\bigl(c\,(t-s)^{1/2}-y\bigr)&\leq p_{t-s}(1)
=(t-s)^{-1/2}p_{1}\bigl((t-s)^{-1/2}\bigr)\\
&=(2\pi)^{-1/2}(t-s)^{-1/2}e^{-1/2(t-s)}.
\end{align*}
This implies that
\begin{align*}
\int_{\R\setminus B_{2}(0)}\mathrm{d}y\ p_{t-s}\bigl(c\,(t-s)^{1/2}-y\bigr)X_{s}(y)
&\leq (2\pi)^{-1/2}(t-s)^{-1/2}e^{-1/2(t-s)}X_{s}(\R)\nonumber\\
&\leq Cc^{-2}X_{s}(\R).
\end{align*}
Combining this bound with (\ref{n5}), we obtain
\begin{align*}
\int_{\left\vert _{\!_{\!_{\,}}}y-c\,(t-s)^{1/2}\right\vert >\frac{c}%
{2}(t-s)^{1/2}}\mathrm{d}y\ p_{t-s}\bigl(c\,(t-s)^{1/2}-y\bigr)X_{s}(y)\\
\leq\,Cc^{-2}\Bigl(W_{B_{3}(0)}+\sup_{0<s\leq t}X_{s}(\R%
)\Bigr).
\end{align*}
Thus, we can choose $c$ so large that the right hand side in the previous
inequality does not exceed $\theta/2$. Since, in view of Lemma~\ref{n.L1},
$$
\liminf_{s\uparrow t}S_{t-s}X_{s}\bigl(c\,(t-s)^{1/2}\bigr)>\theta,
$$
we finally get
\begin{align*}
&  \liminf_{s\uparrow t}\Biggl(S_{t-s}X_{s}\bigl(c\,(t-s)^{1/2}\bigr)\\
&  \qquad\quad-\int_{|y-c\,(t-s)^{1/2}|>\frac{c}{2}(t-s)^{1/2}%
}\mathrm{d}y\ p_{t-s}\bigl(c\,(t-s)^{1/2}-y\bigr)X_{s}(y)\Biggr)\geq\theta/2.
\end{align*}
\textrm{F}rom this bound and (\ref{n4}) the desired property of
$I_{\varepsilon,c}$ follows.
\hfill$\square$
\end{proof}

Fix any $\theta>0$, and to simplify notation write $c:=c_{(\ref{nn.L3})}.$
For all $n$ sufficiently large, say $n\geq N_{0\,},$ define
\begin{align}
\label{An-def}
\nonumber
A_{n}\,:=\biggl\{\Delta X_{s}\left(  \Big(\frac{c}{2}\,2^{-n},\frac{3c}%
{2}\,2^{-n}\Big)\right)   &  \geq2^{-(\bar{\eta}_{\mathrm{c}}+1)n}%
\,n^{1/(1+\beta)}\\
\text{ } &  \text{for some }s\in(t-2^{-2 n},t-2^{-2(n+1)}%
)\biggr\}.
\end{align}

Based on Lemma~\ref{n.L3} we will show in the following lemma that,
if $X_t(0)>0$ then there exist infinitely many jumps $\Delta X_s(x)$
which are greater than $((t-s)|x|)^{1/(1+\beta)}$ with $x\sim (t-s)^{1/2}$.
To be more precise, we show that $A_n$ occur infinitely often.

\begin{lemma}
\label{n.L4}
We have
\begin{equation*}
\mathbf{P}\!\left(  A_{n}\,\text{infinitely often}\ \big|\,D_{\theta}\right)
=1.
\end{equation*}
\end{lemma}

\begin{proof}
If $y\in\left(\frac{c}{2}(t-s)^{1/2},\,\frac{3c}{2}(t-s)^{1/2}\right)$ 
and $s\in(t-2^{-2n},\,t-2^{-2(n+1)})$, then
\begin{align}
\left( (t-s)y\log\bigl((t-s)^{-1}\bigr)\right)^{1/(1+\beta)}
\geq\left(2^{-2(n+1)}\,\frac{c}{2}\,2^{-n-1}2n\log2\right)^{1/(1+\beta)}\nonumber\\
=c_{(\ref{100})}^{-1} 2^{-(\bar{\eta}_{\mathrm{c}}+1)n}\,n^{1/(1+\beta)}.\label{100}%
\end{align}
This implies that
\begin{align}
A_{n}&\supseteq\Bigg\{\Delta X_{s}(y)
\geq c_{(\ref{100})}\left((t-s)y\log\bigl((t-s)^{-1}\bigr)\right)^{1/(1+\beta)}\label{101}\\
&  \qquad\text{for some }s\in(t-2^{-2n},\,t-2^{-2(n+1)})
\text{ and }y\in\Big(\frac{c}{2}\,(t-s)^{1/2},\,\frac{3c}{2}\,(t-s)^{1/2}\Big)\Bigg\}.\nonumber
\end{align}

Consequently, from (\ref{101}) we get
\begin{align*}
\bigcup_{n=N}^{\infty}A_{n}\supseteq\Biggl\{\Delta &X_{s}(y)
\geq c_{(\ref{100})}\left((t-s)y\log\bigl((t-s)^{-1}\bigr)\right)^{1/(1+\beta)}\label{101}\\
&\text{for some }s\in(t-2^{-2N},t)
\text{ and }y\in\Big(\frac{c}{2}\,(t-s)^{1/2},\,\frac{3c}{2}\,(t-s)^{1/2}\Big)\Bigg\}
\end{align*}
for all $N>N_{0}\vee\frac{1}{2}\log_{2}(t\wedge1/8)$.
Applying Lemma~\ref{n.L3} and using the monotonicity of the union in $N$, we
get
\begin{equation*}
\mathbf{P}\!\left(  \bigcup_{n=N}^{\infty}A_{n}\big|\,D_{\theta}\right)
=1\quad\text{for all }N\geq N_{0\,}.
\end{equation*}
This completes the proof.
\hfill$\square$
\end{proof}

\section{Dichotomy for densities}
\label{sec:dichotomy}
\subsection{Proof of Theorem \ref{T.dichotomy}(a)}
The non-random part $\mu\ast p_t(x)$ is continuous. The continuity of $Z_t(\cdot)$ follows from the classical
Kolmogorov criteria. Indeed, it suffices to show that there exist $\theta$, $q$ and $\delta$ as in
Lemma~\ref{L.moment_bound} such that $\delta q>1$. But this is immediate from the observation
$$
\sup_{\delta<\min\{1,(3-\theta/\theta)\},\theta\in(1+\beta,2),q\in(1,1+\beta)}\delta q=
\min\left\{1+\beta,2-\beta\right\}>1.
$$
\begin{remark}
Combining Lemma~\ref{L.moment_bound} with Corollary 1.2 in Walsh \cite{bib:wal86}, we infer
that $Z_t(\cdot)$ is H\"older continuous of all orders smaller than $\delta-1/q$. Noting
that
$$
\sup_{\delta<\min\{1,(3-\theta/\theta)\},\theta\in(1+\beta,2),q\in(1,1+\beta)}(\delta-1/q)
=\min\left\{1,\frac{2-\beta}{1+\beta}\right\}-\frac{1}{1+\beta},
$$
we see that $Z$ is H\"older continuous of all orders smaller than
$\min\{\beta,1-\beta\}/(1+\beta)$. In other words, we proved Theorem \ref{T.loc.Hold}
for $\beta\geq1/2$. 
\hfill$\diamond$
\end{remark}
\subsection{Proof of Theorem \ref{T.dichotomy}(b)}
Throughout this subsection we assume that $d>1$. Recall
that $t>0$ and $X_{0}=\mu\in\mathcal{M}_{\mathrm{f}}\backslash\{0\}$ are
fixed. We want to verify that for each version of the density function 
$X_{t}$ the property
\begin{equation}
\label{basicprop}
\left\Vert X_{t}\right\Vert _{B}=\infty\ \,
\mathbf{P}\text{-a.s. on the event}\,\bigl\{X_{t}(B)>0\bigr\} 
\end{equation}
holds whenever $B$ is a fixed open ball in $\mathbb{R}^{d}$. Having this
relation for every open ball we may prove Theorem \ref{T.dichotomy}(b)
by the following simple argument: Let fix $\omega$ outside a null set so
that \eqref{basicprop} is valid for any ball with rational center and
rational radius. If $U$ is an open set with $X_t(U)>0$ then there exists
a ball $B$ with rational center and rational radius such that
$B\subset U$ and $X_t(B)>0$. Consequently, 
$||X_t||_U(\omega)=||X_t||_B(\omega)=\infty$.

\vspace{6pt}

To get~(\ref{basicprop}) we first show that on
the event $\bigl\{X_{t}(B)>0\bigr\}$ there are always sufficiently ``big''
jumps of $X$ on $B$ that occur close to time $t$. This is done in Lemma~\ref{lem:3}
above. Then with the help of properties of the log-Laplace equation derived in
Lemma~\ref{lem:2} we are able to show that the ``big'' jumps are large enough
to ensure the unboundedness of the density at time $t$. Loosely speaking the
density is getting unbounded in the proximity of big jumps. As we have seen in
the previous section, the largest jump at time $s<t$ is of order $(t-s)^{1/(1+\beta)}$.
Suppose this jump occurs at spatial point $x$. 
Since a jump occuring at time $s$ is smeared out by the kernel $p_{t-s}$,
we have the following estimate for the value of the density at time $t$ and spatial point $x$:
\begin{equation}
\label{verylast1}
X_t(x)\approx (t-s)^{1/(1+\beta)}p_{t-s}(0)\approx(t-s)^{1/(1+\beta)-d/2}.
\end{equation}
From \eqref{verylast1} it is clear that the density should explode in any dimension $d>1$.
In the rest of the section we justify this heuristic.

Set $\varepsilon_{n}:=2^{-n}$, $n\geq1$. Then we choose open balls 
$B_{n}\uparrow B$ such that
\begin{equation}
\overline{B_{n}}\subset B_{n+1}\subset B\quad\text{and}\quad\sup_{y\in
B^{\mathrm{c}},\,x\in B_{n},\,0<s\leq\varepsilon_{n}}\,p_{s}(x-y)
\,\underset{n\uparrow\infty}{\longrightarrow}\,0. \label{Bs}%
\end{equation}
Fix $n\geq1$ such that $\,\varepsilon_{n}<t.$\thinspace\ Set, for brevity, 
$$
\tau_{n}:=\inf\left\{s\in(t-\varepsilon_n,t):
\Delta X_s(B_n)>(t-s)^{\frac{1}{1+\beta}}\log^{\frac{1}{1+\beta}}\left(\frac{1}{t-s}\right)\right\}.
$$
It follows from Lemma \ref{lem:3} that
\begin{equation}
\label{interpret}
\mathbf{P}(\tau_n=\infty)\leq \mathbf{P}(X(B_n)=0),\quad n\geq1.
\end{equation}

In order to obtain a lower bound for $\left\Vert X_{t}\right\Vert _{B}$ we use
the following inequality
\begin{equation}
\left\Vert X_{t}\right\Vert _{B}\geq\int_{B}\mathrm{d}y\ X_{t}(y)p_{u}(y-x),
\quad x\in B,\ u>0. \label{123}%
\end{equation}
On the event $\{\tau_{n}<t\}$, denote by $\zeta_{n}$ the spatial location in
$B_{n}$ of the jump at time $\tau_{n\,},$\thinspace\ and by $r_{n}$ the size
of the jump, meaning that $\Delta X_{\tau_{n}}=r_{n}\delta_{\zeta_{n}}%
.$\thinspace\ Then specializing~(\ref{123}),%
\begin{equation}
\left\Vert X_{t}\right\Vert _{B}\,\geq\int_{B}\mathrm{d}y\ X_{t}%
(y)\,p_{t-\tau_{n}}(y-\zeta_{n})\,\ \text{on the event}\,\ \{\tau
_{n}<t\}. \label{123'}%
\end{equation}
\textrm{F}rom the strong Markov property at time $\tau_{n\,},$\thinspace
\ together with the branching property of superprocesses, we know that
conditionally on $\{\tau_{n}<t\}$, the process $\{X_{\tau_{n}+u}:\,u\geq0\}$
is bounded below in distribution by $\{\widetilde{X}_{u}^{n}:\,u\geq0\}$,
where $\widetilde{X}^{n}$ is a super-Brownian motion with initial value
$r_{n}\delta_{\zeta_{n}}$. Hence, from (\ref{123'}) we get
\begin{align}
&  \mathbf{E}\exp\!\left\{  _{\!_{\!_{\,}}}-\left\Vert X_{t}\right\Vert
_{B}\right\}  \,\label{firstbound}\\[2pt]
&  \leq\,\mathbf{E}\,\mathsf{1}_{\{\tau_{n}<t\}}\exp\left\{  -\int
_{B}\mathrm{d}y\ X_{t}(y)\,p_{t-\tau_{n}}(y-\zeta_{n})\right\}
+\,\mathbf{P}(\tau_{n}=\infty)\nonumber\\
\,  &  \leq\,\mathbf{E}\,\mathsf{1}_{\{\tau_{n}<t\}}\mathbf{E}_{r_{n}%
\delta_{\zeta_{n}}}\exp\left\{  -\int_{B}\mathrm{d}y\ X_{t-\tau_{n}%
}(y)\,p_{t-\tau_{n}}(y-\zeta_{n})\right\}  +\,\mathbf{P}(\tau
_{n}=\infty).\nonumber
\end{align}
Note that on the event $\{\tau_{n}<t\}$, we have
\begin{equation}
r_{n}\geq(t-\tau_{n})^{\frac{1}{1+\beta}}\log^{\frac{1}{1+\beta}}\Big(\frac{1}{t-\tau_{n}}\Big)
=:h_{\beta}(t-\tau_{n}). \label{not.hbeta}
\end{equation}
We now claim that
\begin{equation}
\lim_{n\uparrow\infty}\sup_{0<s<\varepsilon_{n},\ x\in B_{n},\ r\geq h_{\beta}(s)}
\mathbf{E}_{r\delta_{x}}\exp\left\{-\int_{B}\mathrm{d}y\ X_{s}(y)p_{s}(y-x)\right\}
=0. \label{secondlimit}%
\end{equation}
To verify (\ref{secondlimit}), let $s\in(0,\varepsilon_{n})$, $x\in B_{n}$
and $r\geq h_{\beta}(s)$. Then, using the Laplace transition functional of the
superprocess we get
\begin{gather}
\mathbf{E}_{r\delta_{x}}\exp\left\{-\int_{B}\mathrm{d}y\ X_{s}(y)p_{s}(y-x)\right\}
=\exp\left\{-r\,v_{s,x}^{n}(s,x)\right\} \nonumber\\
\leq\exp\left\{-h_{\beta}(s)v_{s,x}^{n}(s,x)\right\},  \label{Lapla}%
\end{gather}
where the non-negative function $\,v_{s,x}^{n}=\bigl\{v_{s,x}^{n}(s^{\prime
},x^{\prime}):\,s^{\prime}>0,$\ $x^{\prime}\in\mathbb{R}^{d}\bigr\}$%
\thinspace\ solves the log-Laplace integral equation
\begin{align}
v_{s,x}^{n}(s^{\prime}, \,x^{\prime})\,=\,&\int_{\mathbb{R}^{d}}%
\mathrm{d}y\ p_{s^{\prime}}(y-x^{\prime})\,1_{B}(y)\,p_{s}(y-x)\label{equt:4}\\
&  -\int_{0}^{s^{\prime}}\mathrm{d}r^{\prime}\int_{\mathbb{R}^{d}}%
\mathrm{d}y\ p_{s^{\prime}-r^{\prime}}(y-x^{\prime})
\bigl(v_{s,x}^{n}(r^{\prime},y)\bigr)^{1+\beta}\nonumber
\end{align}
related to (\ref{logLap}).

\begin{lemma}
\label{lem:2} 
If $d>1$ then
\begin{equation}
\lim_{n\uparrow\infty}\Big(\inf_{0<s<\varepsilon_{n},\,x\in B_{n}}h_{\beta
}(s)\,v_{s,x}^{n}(s,x)\Big)\,=\,+\infty\,. \label{limitv}%
\end{equation}
\end{lemma}
\begin{proof}
We start with a
determination of the asymptotics of the first term at the right hand side of
the log-Laplace equation (\ref{equt:4}) at $(s^{\prime},x^{\prime})=(s,x)$.
Note that
\begin{align}
&  \int_{\mathbb{R}^{d}}\mathrm{d}y\ p_{s}(y-x)\,1_{B}(y)\,p_{s}(y-x)\label{equt:i1}\\
&  =\int_{\mathbb{R}^{d}}\mathrm{d}y\ p_{s}(y-x)\,p_{s}(y-x)
-\int_{B^{\mathrm{c}}}\mathrm{d}y\ p_{s}(y-x)\,p_{s}(y-x).\nonumber
\end{align}
In the latter formula line, the first term equals $p_{2s}(0)=Cs^{-d/2}$,
whereas the second one is bounded from above by
\begin{equation}
\sup_{0<s<\varepsilon_{n},\,x\in B_{n},\ y\in B^{\mathrm{c}}}
p_{s}(y-x)\underset{n\uparrow\infty}{\longrightarrow}0, \label{equt:i2}%
\end{equation}
where the last convergence follows by assumption~(\ref{Bs}) on $B_{n\,}%
.$\thinspace\ Hence from~(\ref{equt:i1}) and (\ref{equt:i2}) we obtain
\begin{equation}
\int_{\mathbb{R}^{d}}\mathrm{d}y\ p_{s}(y-x) 1_{B}(y)\,p_{s}(y-x)
=Cs^{-d/2}+\mathrm{o}(1)\,\ \text{as}\,\ n\uparrow\infty, \label{equt:i3}%
\end{equation}
uniformly in $s\in(0,\varepsilon_{n})$ and $x\in B_{n\,}$. 

To simplify notation, we write $v^{n}:=v_{s,x}^{n}.$ Next, since $v^{n}$ is
non-negative we drop the non-liner term in from~(\ref{equt:4}) to get the
upper bound
\begin{equation*}
v^{n}(s^{\prime},x^{\prime})\,\leq\,
\int_{\R^{d}}\mathrm{d}y\ p_{s^{\prime}}(y-x^{\prime})\,p_{s}(y-x)
=\,p_{s^{\prime}+s}(x-x^{\prime}). 
\end{equation*}
Then we have
\begin{align}
&\int_{0}^{s}\mathrm{d}r^{\prime}\int_{\mathbb{R}^{d}}\mathrm{d}y
\ p_{s-r^{\prime}}(y-x)\bigl(v^{n}(r^{\prime},y)\bigr)^{1+\beta}\label{equt:i5}\\
&\leq\int_{0}^{s}\mathrm{d}r^{\prime}%
\int_{\mathbb{R}^{d}}\mathrm{d}y\ p_{s-r^{\prime}}(y-x)
\bigl(p_{r^{\prime}+s}(x-y)\bigr)^{1+\beta}\nonumber\\
&\leq\bigl(p_{s}(0)\bigr)^{\beta}
\int_{0}^{s}\mathrm{d}r^{\prime}\int_{\mathbb{R}^{d}}\mathrm{d}y
\ p_{s-r^{\prime}}(y-x)\,p_{r^{\prime}+s}(x-y)\nonumber\\
&=\bigl(p_{s}(0)\bigr)^{\beta}
\int_{0}^{s}\mathrm{d}r^{\prime}\ p_{2s}(0)=Cs^{1-d(1+\beta)/2}.\nonumber
\end{align}
Summarizing, by (\ref{equt:4}), (\ref{equt:i3}) and (\ref{equt:i5}),
\begin{equation}
v^{n}(s,x)\geq Cs^{-d/2}+\mathrm{o}(1)-Cs^{1-d(1+\beta)/2}
\label{equt:i6}%
\end{equation}
uniformly in $s\in(0,\varepsilon_{n})$ and $x\in B_{n\,}$. According
to the general assumption $d<2/\beta,$ we conclude that the
right hand side of (\ref{equt:i6}) behaves like $Cs^{-d/2}$ as
$s\downarrow0,$ uniformly in $s\in(0,\varepsilon_{n})$. Now recalling
definition (\ref{not.hbeta}) as well as our assumption
that $d>1$ we immediately get
\begin{equation*}
\lim_{n\uparrow\infty}\,\inf_{0<s<\varepsilon_{n}}h_{\beta}(s)\,s^{-d/2}=+\infty.
\end{equation*}
By (\ref{equt:i6}), this implies (\ref{limitv}), and the proof of the lemma is
finished.
\hfill$\square$
\end{proof}

We are now in position to complete the proof of Theorem \ref{T.dichotomy}(b).
The claim (\ref{secondlimit}) readily follows from
estimate~(\ref{Lapla}) and (\ref{limitv}). Moreover, according to
(\ref{secondlimit}), by passing to the limit $n\uparrow\infty$ in the right
hand side of (\ref{firstbound}), and then using \eqref{interpret}, we arrive
at
\begin{equation*}
\mathbf{E}\exp\left\{  _{\!_{\!_{\,}}}-\left\Vert X_{t}\right\Vert
_{B}\right\}  \,\leq\,\limsup_{n\uparrow\infty}\mathbf{P}\left(  \tau
_{n}=\infty\right)  \,\leq\,\limsup_{n\uparrow\infty}\mathbf{P}\!\left(
_{\!_{\!_{\,}}}X_{t}(B_{n})=0\right)  \!.
\end{equation*}
Since the event $\bigl\{X_{t}(B)=0\bigr\}$ is the non-increasing limit as
$n\uparrow\infty$ of the events $\bigl\{X_{t}(B_{n})=0\bigr\}$ we get%
\begin{equation*}
\mathbf{E}\exp\!\left\{  _{\!_{\!_{\,}}}-\left\Vert X_{t}\right\Vert
_{B}\right\}  \,\leq\,\mathbf{P}\!\left(  _{\!_{\!_{\,}}}X_{t}%
(B)\,=\,0\right)  \!.
\end{equation*}
Since obviously $\left\Vert X_{t}\right\Vert _{B}=0$ if and only if
$X_{t}(B)=0$, we see that (\ref{basicprop}) follows from this last bound. The
proof of Theorem \ref{T.dichotomy}(b) is finished for $U=B$.
\section{Pointwise H\"older exponent at a given point: proof of Theorem~\ref{T.fixed}.}
\label{sec:fixed}
Let us first give a heuristic explanation for the value of $\bar{\eta}_{\rm c}$. According to
Lemmas \ref{6.00} and \ref{n.L3}, the maximal jump at time $s$ and spatial point $x$ near
point $z=0$ is of order $((t-s)|x|)^{1/(1+\beta)}$. Due to the scaling properties of the
heat kernel, the jump that has a decisive effect on the pointwise H\"older exponent at $z=0$
should occur at distance 
\begin{equation}
\label{verylast2}
|x|\approx(t-s)^{1/2}. 
\end{equation}
Then the size of this jump $r$ is of order 
$$
((t-s)|x|)^{1/(1+\beta)}\approx|x|^{3/(1+\beta)}.
$$
Therefore the convolution of the jump $r\delta_x$ with $p_{t-s}(x-\cdot)-p_{t-s}(0-\cdot)$
is of order
$$
|x|^{3/(1+\beta)}(p_{t-s}(0)-p_{t-s}(|x|))\approx|x|^{3/(1+\beta)-1}.
$$
In the last step we used \eqref{verylast2}. This leads then to the result that 
difference of values of the density at points $x$ and $0$ is of the same order.
Then the pointwise H\"older exponent at $0$ should be 
$$
\frac{3}{1+\beta}-1=\bar{\eta}_{\rm c}.
$$ 
This heuristic works for $\bar{\eta}_{\rm c}<1$. In the
case $\bar{\eta}_{\rm c}>1$ the density becomes differentiable. For that reason one has to convolute
$r\delta_x$ with $q_{t-s}(x-\cdot,0-\cdot)$, where $q$ is defined in \eqref{def_qq}. This
convolution is also of the order
\begin{equation}
\label{verylast3}
|x|^{3/(1+\beta)}|q_{t-s}(0,x)|\approx|x|^{3/(1+\beta)-1}.
\end{equation}
Again, we arrive at the same value of the pointwise H\"older exponent $\bar{\eta}_{\rm c}$. 

Since the case $\bar{\eta}_{\mathrm{c}}<1$ has been studied in \cite{FMW11}, we shall
concentrate here on the case $\bar{\eta}_{\mathrm{c}}>1$. In other words, we shall assume
that $\beta<1/2$. Under this assumption, the function
$\frac{\partial}{\partial y}p_{t-u}(x_2-y)$ is integrable with respect to
$M(\mathrm{d}u,\mathrm{d}y)$. Consequently, we may consider $Z_t(x_1,x_2)$ defined in
\eqref{def.Z.new} with $q$ defined in \eqref{def_qq}.

\subsection{Proof of the lower bound for $H_X(0)$.}
To get a lower bound for $H_Z(0)$ it suffices to show that, for every positive $\gamma$,
$$
\sup_{0<x<1}\frac{\left|Z_t(x,0)\right|}{|x|^{\overline{\eta}_{\rm c}-\gamma}}<\infty
$$
with probability one.

Let $\Delta Z_s(x_1,x_2)$ denote the jump of $Z(x_1,x_2)$:
$$
\Delta Z_s(x_1,x_2):=Z_s(x_1,x_2)-Z_{s-}(x_1,x_2).
$$
Denote
\begin{align*}
A_1^\varepsilon:=
\left\{|\Delta X_{s}|\leq c_{(\ref{inL8})}\,(t-s)^{(1+\beta)^{-1}-\gamma}\text{ for all }s<t\right\}
\cap\left\{V(B_1(0))\leq c_\varepsilon\right\}\\
\cap\left\{\Delta X_{s}(x)\leq c_{(\ref{in6.1})}\,\bigl((t-s)|x|\bigr)^{\frac{1}{1+\beta}-\gamma}\text{ for all }s<t\ \text{and}\ x\in
B_{2}(0)\right\}
\end{align*}
with $V$ defined in Lemma \ref{L.fixed.2}. According to Lemma \ref{L.fixed.2}, there exists
$c_\varepsilon$ such that $\mathbf{P}(V(B_1(0)>c_\varepsilon)<\varepsilon$. Combining this with
Lemma \ref{L8} and Corollary \ref{6.1}, we conclude that
\begin{equation}
\label{P_A}
\mathbf{P}(A_1^\varepsilon)>1-3\varepsilon.
\end{equation}
In the next lemma we derive an upper bound for jumps of $Z_s(0,x)$.
Afterwords, in Lemma \ref{L.fixed.5} we derive an upper bound for the values of $Z_s(0,x)$,
which confirms the upper bound part of \eqref{verylast3}.
\begin{lemma}
\label{L.fixed.4}
On the event $A_1^\varepsilon$ we have, for all $s\leq t$ and all $x\in\R$,
$$
\left|\Delta Z_s(x,0)\right|
\leq C |x|^{\overline{\eta}_{\rm c}-3\gamma}.
$$
\end{lemma}
\begin{proof}
Let $(y,s,r)$ be the point of an arbitrary jump of the measure $\mathcal{N}$ with $s\leq t$.
Then for the corresponding jump of $Z(x,0)$ we have the following bound
\begin{equation}
\label{L.fixed.4.1}
\left|\Delta Z_s(x,0)\right|
\leq r|q_{t-s}(x-y,-y)|.
\end{equation}
On $A_1^\varepsilon$ we have
$$
r\leq C(t-s)^{\frac{1}{1+\beta}-\gamma}|y|^{\frac{1}{1+\beta}-\gamma}.
$$
Applying \eqref{L.kernel1.4} with $\delta=\overline{\eta}_{\rm c}-1-3\gamma$ to $|q_{t-s}(x-y,-y)|$, 
we obtain
\begin{align}
\label{L.fixed.4.2}
\nonumber
&\left|\Delta Z_s(x,0)\right|\\
\nonumber
&\leq C (t-s)^{\frac{1}{1+\beta}-\gamma}|y|^{\frac{1}{1+\beta}-\gamma}
\frac{|x|^{\overline{\eta}_{\rm c}-3\gamma}}{(t-s)^{(\overline{\eta}_{\rm c}-3\gamma)/2}}
\left(p_{t-s}(-y/2)+p_{t-s}((x-y)/2)\right)\\
&=C|x|^{\overline{\eta}_{\rm c}-3\gamma}\left(\frac{|y|}{(t-s)^{1/2}}\right)^{\frac{1}{1+\beta}-\gamma}
\left(p_1\left(\frac{-y}{2(t-s)^{1/2}})\right)+p_1\left(\frac{x-y}{2(t-s)^{1/2}})\right)\right).
\end{align}
Assume first that $|x|\leq (t-s)^{1/2}$. If $|y|>2(t-s)^{1/2}$ then $|x-y|>|y|/2$ and,
consequently,
\begin{align*}
&\left(\frac{|y|}{(t-s)^{1/2}}\right)^{\frac{1}{1+\beta}-\gamma}
\left(p_1\left(\frac{-y}{2(t-s)^{1/2}})\right)+p_1\left(\frac{x-y}{2(t-s)^{1/2}})\right)\right)\\
&\hspace{1cm}\leq 2 \left(\frac{|y|}{(t-s)^{1/2}}\right)^{\frac{1}{1+\beta}-\gamma}
p_1\left(\frac{|y|}{4(t-s)^{1/2}})\right).
\end{align*}
Noting that the function on the right hand side is bounded, we conclude from \eqref{L.fixed.4.2}
that
\begin{equation}
\label{L.fixed.4.3}
\left|\Delta Z_s(x,0)\right|
\leq C |x|^{\overline{\eta}_{\rm c}-3\gamma}
\end{equation}
for $|x|\leq (t-s)^{1/2}$ and $|y|>2(t-s)^{1/2}$. Moreover, if $|y|\leq 2(t-s)^{1/2}$ then
$$
\left(\frac{|y|}{(t-s)^{1/2}}\right)^{\frac{1}{1+\beta}-\gamma}
\left(p_1\left(\frac{-y}{2(t-s)^{1/2}})\right)+p_1\left(\frac{x-y}{2(t-s)^{1/2}})\right)\right)
\leq 2^{1+\beta}.
$$
Consequently, \eqref{L.fixed.4.3} is valid also for $|y|\leq 2(t-s)^{1/2}$
and $|x|\leq (t-s)^{1/2}$.

Assume now that $|x|>(t-s)^{1/2}$ and $|y|>2(t-s)^{1/2}$. In this case we bound $q_{t-s}$
in a completely different way:
\begin{equation*}
|q_{t-s}(x-y,-y)|\leq |p_{t-s}(x-y)-p_{t-s}(-y)|+|x|\left|\frac{\partial}{\partial y}p_{t-s}(-y)\right|.
\end{equation*}
Applying now \eqref{L.kernel1.1} with $\delta=\eta_{\rm c}-2\gamma$ and \eqref{L.kernel1.2}, we see that
$|q_{t-s}(x-y,-y)|$ is bounded by
\begin{align*}
&C|x|^{\eta_{\rm c}-2\gamma}(t-s)^{-\frac{1}{1+\beta}+\gamma}
\left(p_1\left(\frac{-y}{2(t-s)^{1/2}})\right)+p_1\left(\frac{x-y}{2(t-s)^{1/2}})\right)\right)\\
&+C\frac{|x|}{t-s}p_1\left(\frac{-y}{2(t-s)^{1/2}})\right).
\end{align*}
Consequently,
\begin{align}
\label{L.fixed.4.4}
\nonumber
\left|\Delta Z_s(x,0)\right|
&\leq C|y|^{\frac{1}{1+\beta}-\gamma}|x|^{\eta_{\rm c}-2\gamma}
\left(p_1\left(\frac{-y}{2(t-s)^{1/2}})\right)+p_1\left(\frac{x-y}{2(t-s)^{1/2}})\right)\right)\\
&\hspace{1cm} +C|x||y|^{\frac{1}{1+\beta}-\gamma}(t-s)^{\frac{1}{1+\beta}-1-\gamma}
p_1\left(\frac{-y}{2(t-s)^{1/2}})\right).
\end{align}
Since $u^{\frac{1}{1+\beta}-\gamma}p_1(u)$ is bounded, the term in the second line in \eqref{L.fixed.4.4}
does not exceed
$$
C|x|(t-s)^{\frac{3}{2(1+\beta)}-1-\frac{3\gamma}{2}}.
$$
As a result, for $|x|>(t-s)^{1/2}$ we get
\begin{align}
\label{L.fixed.4.5}
|y|^{\frac{1}{1+\beta}-\gamma}(t-s)^{\frac{1}{1+\beta}-1-\gamma}
p_1\left(\frac{-y}{2(t-s)^{1/2}})\right)\leq
C|x|^{\frac{3}{1+\beta}-1-3\gamma}=C|x|^{\overline{\eta}_{\rm c}-3\gamma}.
\end{align}
By the same argument,
\begin{align}
\label{L.fixed.4.6}
|y|^{\frac{1}{1+\beta}-\gamma}|x|^{\eta_{\rm c}-2\gamma}
p_1\left(\frac{-y}{2(t-s)^{1/2}})\right)\leq
C|x|^{\overline{\eta}_{\rm c}-3\gamma}.
\end{align}
We next note that if $|y|>2|x|$ then $|y-x|>|y|/2$ and, consequently,
$$
p_1\left(\frac{x-y}{2(t-s)^{1/2}})\right)\leq p_1\left(\frac{-y}{4(t-s)^{1/2}})\right).
$$
Therefore,
\begin{align}
\label{L.fixed.4.7}
|y|^{\frac{1}{1+\beta}-\gamma}|x|^{\eta_{\rm c}-2\gamma}
p_1\left(\frac{x-y}{2(t-s)^{1/2}})\right)\leq
C|x|^{\overline{\eta}_{\rm c}-3\gamma}
\end{align}
for $|y|>2|x|$. But if $|y|\leq 2|x|$ then \eqref{L.fixed.4.7} is obvious. Combining
\eqref{L.fixed.4.4} -- \eqref{L.fixed.4.7} we conclude that \eqref{L.fixed.4.3} holds
for $|x|>(t-s)^{1/2}$. This completes the proof.
\hfill$\square$
\end{proof}
\begin{lemma}
\label{L.fixed.5}
For every fixed $x$ with $|x|<1$ we have
$$
\mathbf{P}\left(\left|Z_t(x,0)\right|>r|x|^{\overline{\eta}_{\rm c}-2\gamma},A_1^\varepsilon\right)
\leq \left(\frac{C}{r}|x|^\gamma\right)^{r|x|^{-\gamma}/C},\quad r>0.
$$
\end{lemma}
\begin{proof}
According to Lemma \ref{L9} there exist spectrally positive $(1+\beta)$-stable processes
$L^+(x)$ and $L^-(x)$ such that
\begin{equation}
\label{L.fixed.5.1}
Z_t(x,0)=L^+_{T_+(x)}-L^-_{T_-(x)},
\end{equation}
where
\begin{equation*}
T_\pm(x):=\int_0^t du \int_{\mathbb{R}}X_u(dy)
\left(\left(q_{t-u}(x-y,-y)\right)^{\pm}\right)^{1+\beta}.
\end{equation*}
By \eqref{L.fixed.3.2} with $\theta=1+\beta$ and $\delta=(1-2\beta-\gamma)/(1+\beta)$,
there exists $C=C(\varepsilon,\gamma)$ such that
\begin{equation}
\label{L.fixed.5.2}
T_\pm(x)\leq C|x|^{2-\beta-\varepsilon_1}\quad\text{on }\ A_1^\varepsilon.
\end{equation}
Applying Lemma \ref{L.fixed.4} and \eqref{L.fixed.5.2}, we obtain
\begin{align*}
&\mathbf{P}\left(L^\pm_{T_{\pm}(x)}(x)\geq r|x|^{\overline{\eta}_{\rm c}-2\gamma},A_1^\varepsilon\right)\\
&\leq \mathbf{P}\left(L^\pm_{T_\pm(x)}(x)\geq r|x|^{\overline{\eta}_{\rm c}-2\gamma},A_1^\varepsilon,
\sup_{s\leq T_\pm(x)} \Delta L^\pm_s\leq C|x|^{\overline{\eta}_{\rm c}-3\gamma}\right)\\
&\leq\mathbf{P}\left(\sup_{s\leq C|x|^{2-\beta-\varepsilon_1}}L^\pm_{s}(x)
\mathsf{1}\bigl\{\sup_{0\leq v\leq s}\Delta L^\pm_{v}\leq C|x|^{\overline{\eta}_{\rm c}-3\gamma}\bigr\}
\geq r|x|^{\overline{\eta}_{\rm c}-2\gamma}\right).
\end{align*}
Applying now Lemma \ref{L3}, we get
\begin{align*}
\mathbf{P}\left(L^\pm_{T_{\pm}(x)}(x)\geq r|x|^{\overline{\eta}_{\rm c}-2\gamma},A_1^\varepsilon\right)
\leq \left(\frac{C|x|^{2-\beta-\gamma}}{r|x|^{\overline{\eta}_{\rm c}-2\gamma}
|y|^{\beta(\overline{\eta}_{\rm c}-3\gamma)}}\right)^{r|x|^{-\gamma}/C}\\
=\left(\frac{C}{r}|x|^\gamma\right)^{r|x|^{-\gamma}/C}.
\end{align*}
Thus, the proof is finished.
\hfill$\square$
\end{proof}

\vspace{12pt}

Taking into account \eqref{P_A}, we see that the lower bound for
$H_{Z}(0)$ will be proven if we show that, for every $\varepsilon>0$,
\begin{equation}
\label{T.fixed.10}
\sup_{0<x<1}\frac{\left|Z_t(x,0)\right|}{|x|^{\overline{\eta}_{\rm c}-\gamma}}<\infty
\quad\text{on }\ A_1^\varepsilon.
\end{equation}
Fix some $q\in(0,1)$ and note that
\begin{equation*}
\left\{  \sup_{0<x<1}\frac{\left|Z_t(x,0)\right|}{x^{\overline{\eta}_{\rm c}-\gamma}}>k\right\}
\subseteq\bigcup_{n=1}^{\infty}\left\{  \sup_{x\in I_{n}}\left|Z_t(x,0)\right|
>\frac{k}{2^{q}}n^{-q(\overline{\eta}_{\rm c}-\gamma)}\right\},
\end{equation*}
where $I_{n}:=\{x:(n+1)^{-q}\leq x<n^{-q}\}$. Moreover, by the triangle
inequality,
\begin{equation*}
\left|Z_t(x,0)\right| \leq \left|Z_t(n^{-q},0)\right|
+\left|Z_t(x)-Z_t(n^{-q})\right|+(n^{-q}-x)|W|,\quad x\in I_{n},
\end{equation*}
where
$$
W:=\int_0^t\int_{\mathbb{R}}
M\left(\mathrm{d}(u,y)\right)\frac{\partial}{\partial y}p_{t-u}(-y).
$$
Consequently, for all $n\geq1$,
\begin{align*}
&  \left\{ \sup_{x\in I_{n}}\left|Z_t(x,0)\right|>\frac{k}{2^{q}}n^{-q(\bar\eta_{\rm c}-\gamma)}\right\}
\subseteq\left\{\sup_{x\in I_n}\bigl|Z_t(x)-Z_t(n^{-q})\bigr|>\frac{k}{3\cdot2^{q}}n^{-q}\right\}\\ 
&\hspace{1cm} \cup\left\{ \left|Z_t(n^{-q},0)\right|>\frac{k}{3\cdot2^{q}}n^{-q(\bar\eta_{\rm c}-\gamma)}\right\}
\cup\left\{|W|>\frac{kn^{-q(\bar\eta_{\rm c}-\gamma)}}{3\cdot2^{q}(n^{-q}-(n+1)^{-q})}\right \},
\end{align*}
Note that, for all $R>0$,
\begin{align*}
&  \left\{  \sup_{0<x<y<1}\frac{\bigl|Z_t(x)-Z_t(y)\bigr|}%
{|x-y|^{q(\bar\eta_{\rm c}-\gamma)/(q+1)}}\leq R\right\} \nonumber\\
&  \subseteq\left\{  \bigl|Z_t(x)-Z_t(n^{-q})\bigr|\leq
R\,q^{q(\bar\eta_{\rm c}-\gamma)/(q+1)}n^{-q\eta},\,\ x\in I_{n}\right\}  \!.
\end{align*}
This implies that
\begin{align*}
\left\{  \sup_{0<x<1}\frac{\left|Z_t(x,0)\right|}{x^{\overline{\eta}_{\rm c}-\gamma}}>k\right\}
\subseteq \left\{\sup_{0<x<y<1}\frac{\bigl|Z_t(x)-Z_t(y)\bigr|}{|x-y|^{q(\bar\eta_{\rm c}-\gamma)/(q+1)}}>c(q)k\right\}\\
\cup\bigcup_{n=1}^\infty \left\{ \left|Z_t(n^{-q},0)\right|>\frac{k}{3\cdot2^{q}}n^{-q(\bar\eta_{\rm c}-\gamma)}\right\}
\cup\left\{|W|>c(q)k\right\}.
\end{align*}
where $c(q)$ is some positive constant. 

If we choose $q$ so small that
\thinspace$(\bar\eta_{\rm c}-\gamma) q/(q+1)<\eta_{c\,},$\thinspace\ then
\begin{equation*}
\lim_{k\rightarrow\infty}\mathbf{P}\biggl(\sup_{0<x<y<1}
\frac{\bigl|Z_t(x)-Z_t(y)\bigr|}{|x-y|^{q(\bar\eta_{\rm c}-\gamma)/(q+1)}}>c(q)k\biggr)=0,
\end{equation*}
since, by Theorem \ref{T.loc.Hold}, $Z$ is locally H\"{o}lder continuous of every index 
smaller than $\eta_{c}$. Furthermore,
\begin{equation*}
\lim_{k\rightarrow\infty}\mathbf{P}(|W|>c(q)k)=0.
\end{equation*}
Finally, applying Lemma \ref{L.fixed.5}, conclude that
\begin{equation*}
\lim_{k\rightarrow\infty}\mathbf{P}\biggl(
\bigcup_{n=1}^{\infty}\left\{\left|Z_t(n^{-q},0)\right|>\frac{k}{3\cdot2^{q}}n^{-q(\bar\eta_{\rm c}-\gamma)}\right\}  \biggr)=0.
\end{equation*}
Thus, \eqref{T.fixed.10} is shown.
\subsection{Proof of the optimality of $\bar{\eta}_{\mathrm{c}}$.}

Now it is time to explain the detailed strategy of the optimality proof.
Define
\begin{align}
&  A_2^{\varepsilon}:=\label{Aep}\\
&  \Big\{\Delta X_{s}(y)\leq c_{(\ref{in6.1a})}%
\,\bigl((t-s)|y|\bigr)^{1/(1+\beta)}(f_{s,x})^{\ell}\text{ for all
}s<t\ \text{and}\ y\in B_{1/\mathrm{e}}(0)\Big\}\nonumber\\
&  \cap\,\Big\{\Delta X_{s}(y)\leq c_{\eqref{inL8}}(t-s)^{1/(1+\beta)-\gamma
}\ \text{for all }s<t\ \text{and }y\in\R\Big\}\cap\{V(B_2(0))\leq
c_{\varepsilon}\},\nonumber
\end{align}
where $f_{s,x\,}$ and $\ell$ are defined in~(\ref{def_f}) and (\ref{def.lambda1}), 
respectively. Further, according to Lemma \ref{L.fixed.2}, 
$\mathbf{P}(V(B_2(0))>c_\varepsilon)\to 0$ for any 
$c_\varepsilon\to\infty$ as $\varepsilon\to0$. 
Note that $D_{\theta}\uparrow\{X_{t}(0)>0\}$ as $\theta\downarrow0$.
Moreover, by Lemmata \ref{L8}, \ref{6.00} and \ref{L.fixed.2},
$A_2^\varepsilon\uparrow\Omega$ as $\varepsilon\downarrow0$. Hence, for the proof of
Theorem~\ref{T.fixed}(b) \emph{it is sufficient}\/ to show that%
\begin{equation*}
\mathbf{P}\biggl(\sup_{x\in B_{\epsilon}(0),\,x\neq0}
\frac{\bigl|{X}_{t}(x)-{X}_{t}(0)\bigr|}{|x|^{\bar{\eta}_{\mathrm{c}}}}\,=\,\infty
\,\bigg|\,D_{\theta}\cap A_2^\varepsilon\biggr)=1.
\end{equation*}
Since $\mu\ast p_t(x)$ is Lipschitz continuous at
$0$, the latter will follow from the equality%
\begin{align}
\mathbf{P}\Big(Z_t(c\,2^{-n-2},0)\geq   \ 2^{-\bar{\eta
}_{\mathrm{c}}n}\,n^{1/(1+\beta)-\varepsilon}\text{ }\label{Nr.}
\ \text{infinitely often}\,\Big|\,D_{\theta}\cap A_2^\varepsilon\Big)=1,
\end{align}
where we choose $c=c_{(\ref{nn.L3})}$.

To verify (\ref{Nr.}), we will again use the{LM ''the'' instead of ''our''}  method of representing
$Z$ as a time-changed stable process. To be more precise, applying
(\ref{L.fixed.5.1}) with $x=c\,2^{-n-2}$ (for $n$ sufficiently large) and using
$n$-dependent notation as $L_{n}^{\pm},T_{n,\pm}$ (and $\varphi_{n,\pm})$, we
have
\begin{equation*}
Z_t(c\,2^{-n-2},0)=L_{n}^{+}(T_{n,+})-L_{n}^{-}(T_{n,-}).
\end{equation*}

It follows from this representation that \eqref{Nr.} is
a consequence in the following statement.
\begin{proposition}
\label{very_new_prop}
For almost every $\omega\in D_\theta\cap A_2^\varepsilon$ there exists a subsequence $n_j$
such that 
$$
L_{n_j}^{+}(T_{n_j,+})\geq 2^{1-\bar{\eta}_{\rm c}n_j}n_j^{1/(1+\beta)-\varepsilon}
\text{ and }\quad
L_{n_j}^{-}(T_{n_j,-})\leq 2^{-\bar{\eta}_{\rm c}n_j}n_j^{1/(1+\beta)-\varepsilon}.
$$
\end{proposition}

\medskip

Let us define the following events
\[
B_{n}^{+}:=\bigl\{L_{n}^{+}(T_{n,+})\geq2^{1-\bar{\eta}_{\mathrm{c}}%
n}\,n^{1/(1+\beta)-\varepsilon}\bigr\},\quad B_{n}^{-}:=\bigl\{L_{n}%
^{-}(T_{n,-})\leq2^{-\bar{\eta}_{\mathrm{c}}n}\,n^{1/(1+\beta)-\varepsilon
}\bigr\}
\]
and
\begin{equation*}
B_{n}:=B_{n}^{+}\cap B_{n}^{-}.
\end{equation*}
Then, obviously, Proposition \ref{very_new_prop} will follow once we verify
\begin{equation}
\lim_{N\uparrow\infty}\mathbf{P}\biggl(\,\bigcup_{n=N}^{\infty}(B_{n}\cap
A_{n})\,\bigg|\,D_{\theta}\cap A_2^\varepsilon\,\biggr)=1, \label{n6}%
\end{equation}
where $A_n$ were defined in \eqref{An-def}.
Taking into account Lemma~\ref{n.L4}, we conclude that to get (\ref{n6}) we
have to show
\begin{equation}
\lim_{N\uparrow\infty}\mathbf{P}\biggl(\,\bigcup_{n=N}^{\infty}(B_{n}%
^{\mathrm{c}}\cap A_{n})\,\bigg|\,D_{\theta}\cap A_2^\varepsilon\biggr)=0.
\label{n7}%
\end{equation}
Let us explain briefly the meaning of \eqref{n6}. By Lemma \ref{n.L4} we know
that there exists a sequence of big jumps
$$
\Delta X_s\left(\frac{c}{2}2^{-n_j},\frac{3c}{2}2^{-n_j}\right)\geq
2^{-(\overline{\eta}_{\rm c}+1)n_j}n_j^{1/(1+\beta)}
$$
for some $s\in(t-2^{-2n_j},t-2^{-2(n_j+1)})$. \eqref{n6} implies that
these jumps guarantee big values of $L_{n_j}^{+}(T_{n_j,+})-L_{n_j}^{-}(T_{n_j,-})$
for some subsequence of $\{n_j\}$. And this is the main consequence of
Proposition \ref{very_new_prop}.

Now we will present two lemmas, from which (\ref{n7}) will follow immediately.
To this end, split
\begin{equation}
B_{n}^{\mathrm{c}}\cap A_{n}=(B_{n}^{+,\mathrm{c}}\cap A_{n})\cup
(B_{n}^{-,\mathrm{c}}\cap A_{n}). \label{114}%
\end{equation}

\begin{lemma}
\label{n.L5}We have%
\begin{equation*}
\lim_{N\uparrow\infty}\sum_{n=N}^{\infty}\mathbf{P}\bigl(B_{n}^{+,\mathrm{c}%
}\cap A_{n}\cap A_2^\varepsilon\bigr)=\,0.
\end{equation*}
\end{lemma}

The proof of this lemma is a word-for-word repetition of the proof of
Lemma~5.3 in \cite{FMW10} (it is even simpler as we do not need additional
indexing in $k$ here), and we omit it. The idea behind the proof is simple:
Whenever $X$ has a \textquotedblleft big\textquotedblright\ jump guaranteed
by $A_{n\,},$ this jump corresponds to the jump of $L_{n}^{+}$ and then it
is very difficult for a spectrally positive process $L_{n}^{+}$ to come down,
which is required by $B_{n}^{+,\mathrm{c}}$. 

Note that Lemma \ref{n.L5} alone is not enough to finish the proof of
Proposition \ref{very_new_prop}:
on $\{L_n^+\geq 2^{1-\bar{\eta}_{\rm c}n}n^{1/(1+\beta)-\varepsilon}\}$ we may still have
$Z_t(c2^{-n-2},0)\leq 2^{-\bar{\eta}_{\rm c}n}n^{1/(1+\beta)-\varepsilon}$
if $B_n^{-,c}$ occurs. In other words, this is the situation where the jump of $X$,
which leads to a large
value of $L_n^+$, can be compensated by a further jump of $X$. The next lemma
states that it can not happen that all the jumps guaranteed by $A_n$'s will be
compensated.
\begin{lemma}
\label{n.L5'}
We have
\begin{equation*}
\lim_{N\uparrow\infty}\sum_{n=N}^{\infty}\mathbf{P}\bigl(B_{n}^{-,\mathrm{c}%
}\cap A_{n}\cap A_2^\varepsilon\cap D_{\theta}\bigr)=0.
\end{equation*}
\end{lemma}

\medskip

Now we are ready to finish\\
\noindent{\it Proof of Proposition \ref{very_new_prop}.\ }
Combining Lemmata \ref{n.L5} and \ref{n.L5'}, we conclude that
there exists a subsequence $\{n_j\}$ with properties described in 
Proposition~\ref{very_new_prop}.
\hfill$\square$

\medskip

The remaining part of the paper will be devoted to the proof of
Lemma~\ref{n.L5'} and we prepare now for it.

One can easily see that $B_{n}^{-,\mathrm{c}}$ is a subset of a union of two
events (with the obvious correspondence):
\begin{align*}
B_{n}^{-,\mathrm{c}}\subseteq U_{n}^{1}\cup U_{n}^{2}:=\,  &  \bigl\{\Delta
L_{n}^{-}>2^{-\bar{\eta}_{\mathrm{c}}n}\,n^{1/(1+\beta)-2\varepsilon
}\bigr\}\nonumber\\
&  \cup\bigl\{\Delta L_{n}^{-}\leq2^{-\bar{\eta}_{\mathrm{c}}n}\,n^{1/(1+\beta
)-2\varepsilon},\ L_{n}^{-}(T_{n,-})>2^{-\bar{\eta}_{\mathrm{c}}%
n}\,n^{1/(1+\beta)-\varepsilon}\bigr\},
\end{align*}
where
\begin{equation*}
\Delta L_{n}^{-}:=\sup_{0<s\leq T_{n,-}}\Delta L_{n}^{-}(s).
\end{equation*}
The occurrence of the event $U_{n}^{1}$ means that $L_{n}^{-}$ has big jumps.
If $U_{n}^{2}$ occurs, it means that $L_{n}^{-}$ gets large without big jumps.
It is well-known that stable processes without big jumps can not achieve large
values. Thus, the statement of the next lemma is not surprising.

\begin{lemma}
\label{n.L6}
We have
\begin{equation*}
\lim_{N\uparrow\infty}\sum_{n=N}^{\infty}\mathbf{P}(U_{n}^{2}\cap
A_2^\varepsilon)=0.
\end{equation*}
\end{lemma}

We omit the proof of this lemma as well, since its crucial part related to
bounding of $\mathbf{P}(U_{n}^{2}\cap A_2^\varepsilon)$ is a repetition of the
proof of Lemma~5.6 in \cite{FMW10} (again
with obvious simplifications).

\begin{lemma}
\label{n.L7}
There exist constants $\rho$ and $\xi$ such
that, for all sufficiently large values of $n$,
\begin{equation*}
A_2^\varepsilon\cap A_{n}\cap U_{n}^{1}\,\subseteq\,A_2^\varepsilon\cap
E_{n}(\rho,\xi),
\end{equation*}
where
\begin{align*}
E_{n}(\rho,\xi):=  &  \left\{  _{\!_{\!_{\,_{_{\!_{\!_{_{\!_{\!_{_{\!_{\!_{\,}%
}}\,}}}}}}}}}}\mathrm{There}\text{ }\mathrm{exist}\text{ }\mathrm{at}\text{
}\mathrm{least}\text{\textrm{ }}\mathrm{two}\text{\textrm{ }}\mathrm{jumps}%
\text{\textrm{ }}\mathrm{of\;}M\ \mathrm{of}\text{\textrm{ }}\mathrm{the}%
\text{\textrm{ }}\mathrm{form\;}r\delta_{(s,y)}\;\mathrm{such}\text{\textrm{
}}\mathrm{that}\right. \nonumber\\
&  \quad r\geq\left(  (t-s)\max\bigl\{(t-s)^{1/2},|y|\bigr\}\right)
^{1/(1+\beta)}\log^{1/(1+\beta)-2\varepsilon}\bigl((t-s)^{-1}%
\bigr),\nonumber\\[3pt]
&  \quad|y|\leq(t-s)^{1/2}\log^{\xi}\bigl((t-s)^{-1}\bigr),\ \!\!\left.
_{\!_{\!_{\,_{_{\!_{\!_{_{\!_{\!_{_{\!_{\!_{\,}}}\,}}}}}}}}}}s\in\left[
t-2^{-2 n}\,n^{\rho},\ t-2^{-2 n}\,n^{-\rho}\right]  \right\}  \!.
\end{align*}
\end{lemma}

\begin{proof}
By the definition of $A_{n\,},$\thinspace\ there exists a jump of $M$ of the
form $r\delta_{(s,y)}$ with $r,s$ as in $E_{n}(\rho,\xi)$, and $y>c\,2^{-n-1}%
$. Furthermore, noting that $\varphi_{n,-}(y)=0$ for $y\geq c\,2^{-n-3},$ we
see that the jumps $r\delta_{(s,y)}$ of $M$ contribute to $L_{n}^{-}(T_{n,-})$
if and only if $y<c\,2^{-n-3}$. Thus, to prove the lemma it is sufficient to
show that $U_{n}^{1}$ yields the existence of at least one further jump of $M$
on the half-line $\{y<c\,2^{-n-3}\}$ with properties mentioned in the
statement. Denote
\begin{align}
D:=\bigg\{(r,s,y)\!:\  &  r\geq\!\left((t-s)\max\bigl\{(t-s)^{1/2},|y|\bigr\}\right)^{1/(1+\beta)}
\log^{1/(1+\beta)-2\varepsilon}\bigl((t-s)^{-1}\bigr),\nonumber\\
&  y\in\left(  -(t-s)^{1/2}\log^{\xi}\bigl((t-s)^{-1}\bigr),\,c\,2^{-n-3}%
\right)  \!,\ \nonumber\\
&  \qquad\qquad\qquad\qquad\qquad\qquad\quad\ \,s\in\left[  t-2^{-2n}\,n^{\rho},
\ t-2^{-2 n}\,n^{-\rho}\right]  \!\bigg\}. \label{setD}%
\end{align}
Then we need to show that $U_{n}^{1}$ implies the existence of a jump
$r\delta_{(s,y)}$ of $M$ with $(r,s,y)\in D.$

Note that
\begin{align*}
\lefteqn{D=D_{1}\cap D_{2}\cap D_{3}}\\
&  :=\left\{  (r,s,y):\;r\geq0,\ s\in\lbrack0,t],\;y\in\left(
-(t-s)^{1/2}\log^{\xi}\bigl((t-s)^{-1}\bigr),\,c\,2^{-n-3}\right)
\right\} \\
&  \quad\quad\cap\Big\{(r,s,y):\;r\geq0,\ y\in(-\infty,\,c\,2^{-n-3}%
),\ s\in\left[  t-2^{-2 n}\,n^{\rho},\ t-2^{-2 n}\,n^{-\rho}\right]
\!\Big\}\\
&  \quad\quad\cap\bigg\{(r,s,y):\;y\in(-\infty,\,c\,2^{-n-3}),\ s\in
\lbrack0,t],\\
&  \;\;\;\;\;\;\;\quad\qquad\ r\geq\left(  (t-s)\max\bigl\{(t-s)^{1/2
},|y|\bigr\}\right)  ^{1/(1+\beta)}\log^{1/(1+\beta)-2\varepsilon
}\bigl((t-s)^{-1}\bigr)\bigg\}.
\end{align*}
Therefore,%
\begin{equation*}
D^{\mathrm{c}}\cap\{y<c2^{-n-3}\}\,=\,\left(  D_{1}^{\mathrm{c}}%
\cap\{y<c\,2^{-n-3}\}\right)  \cup\left(  D_{1}\cap D_{2}^{\mathrm{c}}\right)
\cup\left(  D_{1}\cap D_{2}\cap D_{3}^{\mathrm{c}}\right)  \!,
\end{equation*}
where the complements are defined with respect to the set%
\begin{equation*}
\left\{  _{\!_{\!_{\,}}}(r,s,y):\;r\geq0,\ s\in\lbrack0,t],\;y\in
\R\right\}  \!.
\end{equation*}

We first show that any jumps of $M$ in $D_{1}^{\mathrm{c}}\cap\{y<c\,2^{-n-3}%
\}$ cannot be the course of a jump of $L_{n}^{-}$ such that $U_{n}^{1}$ holds.
Indeed, using the last inequality in Lemma~\ref{L.kernel2} with
$\delta=\bar{\eta}_{\mathrm{c}\,}$, we get for $y<c2^{-n-3}$ the inequality
\begin{align}
(q_{t-s}(c2^{-n-2},0))^-
&\leq C 2^{-\bar{\eta}_{\mathrm{c}}n}(t-s)^{-\bar{\eta}_{\mathrm{c}}/2}p_{t-s}(y/2)\nonumber\\
&\leq C\,2^{-\bar{\eta}_{\mathrm{c}}n}(t-s)^{-(1+\bar{\eta}_{\mathrm{c}})/2}
\exp\Big\{-\frac{y^2}{8(t-s)}\Big\}\nonumber\\
&\leq C\,2^{-\bar{\eta}_{\mathrm{c}}n}(t-s)^{1-\bar{\eta}_{\mathrm{c}}/2}|y|^{-3}, \label{n8}%
\end{align}
in the second step we used the scaling property of the kernel $p$, and in the last step we have
used the trivial bound $\mathrm{e}^{-x}\leq x^{-3/2}$.

Further, by~(\ref{Aep}), on the set $A_2^\varepsilon$ we have
\begin{equation}
\Delta X_{s}(y)\leq C\,\bigl(|y|(t-s)\bigr)^{1/(1+\beta)}(f_{s,y})^{\ell},
\quad|y|\leq1/\mathrm{e}, \label{n9}%
\end{equation}
and
\begin{equation}
\Delta X_{s}(y)\leq C\,(t-s)^{1/(1+\beta)-\gamma},\quad|y|>1/\mathrm{e},
\label{n10}%
\end{equation}
and recall that $f_{s,x}=\,\log\bigl((t-s)^{-1}\bigr)\,\mathsf{1}_{\{x\neq
0\}}\log\bigl(|x|^{-1}\bigr)$. Combining (\ref{n8}) and~(\ref{n9}), we
conclude that the corresponding jump of $L_{n}^{-},$ henceforth denoted by
\hspace{-1pt}$\Delta L_{n}^{-}[r\delta_{(s,y)}],$ is bounded by
\begin{equation*}
C\,2^{-\bar{\eta}_{\mathrm{c}}n}(t-s)^{1-\bar{\eta}_{\mathrm{c}}/2
+\frac{1}{1+\beta}}\log^{\frac{1}{1+\beta}+q}\bigl((t-s)^{-1}%
\bigr)|y|^{-3+\frac{1}{1+\beta}}\log^{\frac{1}{1+\beta}+q}%
\bigl(|y|^{-1}\bigr).
\end{equation*}
Since $|y|^{-3+\frac{1}{1+\beta}}\log^{\frac{1+\gamma}{1+\beta}}\bigl(|y|^{-1}\bigr)$
is monotone decreasing, we get, maximizing over $y$, for 
$y<-(t-s)^{1/2}\log^{\xi}\bigl((t-s)^{-1}\bigr)$ the bound
\begin{equation*}
\Delta L_{n}^{-}[r\delta_{(s,y)}]\leq C\,2^{-\bar{\eta}_{\mathrm{c}}n}%
\log^{\frac{2}{1+\beta}+2q-\xi(3-\frac{1}{1+\beta})}\bigl(|y|^{-1}%
\bigr).
\end{equation*}
Choosing $\xi\geq\frac{2+2q(1+\beta)}{3(1+\beta)-1}$, we
see that
\begin{equation}
\Delta L_{n}^{-}[r\delta_{(s,y)}]\leq C\,2^{-\bar{\eta}_{\mathrm{c}}n}%
,\quad|y|<1/\mathrm{e}. \label{nn1}%
\end{equation}
Moreover, if $y<-1/\mathrm{e}$, then it follows from (\ref{n8}) and
(\ref{n10}) that the jump $\Delta L_{n}^{-}[r\delta_{(s,y)}]$ is bounded by
\begin{equation}
C\,2^{-\bar{\eta}_{\mathrm{c}}n}(t-s)^{1-\bar{\eta}_{\mathrm{c}}/2
+\frac{1}{1+\beta}-\gamma}|y|^{-3}\leq C\,2^{-\bar{\eta}_{\mathrm{c}}%
n}. \label{nn2}%
\end{equation}
Combining (\ref{nn1}) and (\ref{nn2}), we see that all the jumps of $M$ in
$D_{1}^{\mathrm{c}}\cap\{y<c\,2^{-n-3}\}$ do not produce jumps of $L_{n}^{-}$
such that $U_{n}^{1}$ holds.

We next assume that $M$ has a jump $r\delta_{(s,y)}$ in $D_{1}\cap
D_{2}^{\mathrm{c}}$. If, additionally, $s\leq t-2^{-2 n}\,n^{\rho}$,
then, using the last inequality in Lemma~\ref{L.kernel2} with $\delta=1$, we get
\begin{equation*}
(q_{t-s}(c2^{-n-2},0))^-\leq C 2^{-2n}(t-s)^{-3/2}.
\end{equation*}
\textrm{F}rom this bound and (\ref{n9}) we obtain
\begin{gather*}
\Delta L_{n}^{-}[r\delta_{(s,y)}]
\leq C\,2^{-2n}(t-s)^{-3/2+\frac{1}{1+\beta}}
\log^{\frac{1}{1+\beta}+q}\bigl((t-s)^{-1}\bigr)|y|^{\frac{1}{1+\beta}}
\log^{\frac{1}{1+\beta}+q}\bigl(|y|^{-1}\bigr)\nonumber\\
\leq C\,2^{-2n}(t-s)^{\frac{3}{2}(\frac{1}{1+\beta}-1)}
\log^{\frac{2+\xi}{1+\beta}+2q}\bigl((t-s)^{-1}\bigr).
\end{gather*}
Using the assumption $t-s\geq2^{-2 n}\,n^{\rho}$, we arrive at the
inequality
\begin{equation*}
\Delta L_{n}^{-}[r\delta_{(s,y)}]
\leq C\,2^{-\bar{\eta}_{\mathrm{c}}}
n^{\frac{2+\xi}{1+\beta}+2q+\frac{\rho}{2}(\bar{\eta}_{\mathrm{c}}-2)}.
\end{equation*}
\textrm{F}rom this we see that if we choose 
$\rho\geq\frac{2(2q+(2+\xi)/(1+\beta))}{2-\bar{\eta}_{\mathrm{c}}}$ then
the jumps of $L_{n}^{-}$ are bounded by $C\,2^{-\bar{\eta}_{\mathrm{c}}n}$,
and hence $U_{n}^{1}$ does not occur.

Using (\ref{L.kernel1.1}) with $\delta=1$ and \eqref{L.kernel1.2}, one can easily derive
\begin{equation*}
(q_{t-s}(c2^{-n-2},0))^-\leq C2^{-n}(t-s)^{-1}.
\end{equation*}
Then for $y\in\left(-(t-s)^{1/2}\log^{\xi}\bigl((t-s)^{-1}\bigr),\ c\,2^{-n-3}\right)$ 
and $t-s\leq2^{-2 n}\,n^{-\rho}$ we have the inequality
\begin{align*}
\Delta L_{n}^{-}[r\delta_{(s,y)}]
&\leq C 2^{-n}(t-s)^{-1}(|y|(t-s))^{\frac{1}{1+\beta}}(f_{s,y})^\ell\\
&\leq C 2^{-n}(t-s)^{\frac{3}{2(1+\beta)}-1}\log^{\frac{2+\xi}{1+\beta}+2q}\bigl((t-s)^{-1}\bigr)\\
&= C 2^{-n}(t-s)^{(\bar{\eta}_{\mathrm{c}}-1)/2}\log^{\frac{2+\xi}{1+\beta}+2q}\bigl((t-s)^{-1}\bigr)\\
&\leq C 2^{-\bar{\eta}_{\mathrm{c}}n}n^{-\rho(\bar{\eta}_{\mathrm{c}}-1)/2+\frac{2+\xi}{1+\beta}+2q}.
\end{align*}
Choosing $\rho\geq\frac{2(\xi+2+2q(1+\beta))}{(1+\beta)(\bar{\eta}_{\mathrm{c}}-1)}$,
we conclude that $\Delta L_{n}^{-}[r\delta_{(s,y)}]\leq
C\,2^{-\bar{\eta}_{\mathrm{c}}n}$, and again $U_{n}^{1}$ does not occur.

Finally, it remains to consider the jumps of $M$ in $D_{1}\cap D_{2}\cap
D_{3\,}^{\mathrm{c}}.$ If the value of the jump does not exceed
$(t-s)^{\frac{3}{2(1+\beta)}}\log^{\frac{1}{1+\beta}-2\varepsilon
}\bigl((t-s)^{-1}\bigr)$, then it follows from \eqref{L.kernel1.4} with
$\delta=\bar{\eta}_{\mathrm{c}}-1$ that
\begin{align*}
\Delta L_{n}^{-}[r\delta_{(s,y)}]
&\leq C2^{-\bar{\eta}_{\mathrm{c}}n}(t-s)^{-(\bar{\eta}_{\mathrm{c}}-1)/2}
(t-s)^{\frac{3}{2(1+\beta)}}\log^{\frac{1}{1+\beta}-2\varepsilon}\bigl((t-s)^{-1}\bigr)\\
&\leq C2^{-\bar{\eta}_{\mathrm{c}}n}\log^{\frac{1}{1+\beta}-2\varepsilon}\bigl((t-s)^{-1}\bigr).
\end{align*}
Then, on $D_{2}$, that is, for $t-s>2^{-2 n}\,n^{-\rho}$,
\begin{equation}
\Delta L_{n}^{-}[r\delta_{(s,y)}]\leq C\,2^{-\bar{\eta}_{\mathrm{c}}%
n}\,n^{\frac{1}{1+\beta}-2\varepsilon}. \label{n11}%
\end{equation}
Furthermore, if $y<-(t-s)^{1/2}$ and the value of the jump is less than
$\bigl(|y|(t-s)\bigr)^{\frac{1}{1+\beta}}\log^{\frac{1}{1+\beta}-2\varepsilon
}\bigl((t-s)^{-1}\bigr)$, then, using (\ref{n8}), we get
\begin{gather*}
\Delta L_{n}^{-}[r\delta_{(s,y)}]\leq C\,2^{-\bar{\eta}_{\mathrm{c}}%
n}(t-s)^{1-\bar{\eta}_{\mathrm{c}}/2}\log^{\frac{1}{1+\beta}%
-2\varepsilon}\bigl((t-s)^{-1}\bigr)|y|^{-3+\frac{1}{1+\beta}%
}\nonumber\\
\leq C\,2^{-\bar{\eta}_{\mathrm{c}}n}\log^{\frac{1}{1+\beta}-2\varepsilon
}\bigl((t-s)^{-1}\bigr).
\end{gather*}
Then, on $D_{2}$, that is, for $t-s>2^{-2 n}\,n^{-\rho}$,
\begin{equation}
\Delta L_{n}^{-}[r\delta_{(s,y)}]\leq C\,2^{-\bar{\eta}_{\mathrm{c}}%
n}\,n^{\frac{1}{1+\beta}-2\varepsilon}. \label{n12}%
\end{equation}
By (\ref{n11}) and (\ref{n12}), we see that the jumps of $M$ in $D_{1}\cap
D_{2}\cap D_{3}^{\mathrm{c}}$ do not produce jumps such that $U_{n}^{1}$
holds. Combining all the above we conclude that to have $\Delta L_{n}%
^{-}[r\delta_{(s,y)}]>C\,2^{-\bar{\eta}_{\mathrm{c}}n}\,n^{\frac{1}{1+\beta
}-2\varepsilon}$ it is necessary to have a jump in $D_{1}\cap D_{2}\cap
D_{3\,}.$\thinspace\ Thus, the proof is finished.
\hfill$\square$
\end{proof}
Now we are ready to finish\\
\noindent{\it Proof of Lemma $\ref{n.L5'}$.\ }
In view of the Lemmas~\ref{n.L6} and \ref{n.L7},
it suffices to show that
\begin{equation}
\lim_{N\uparrow\infty}\sum_{n=N}^{\infty}\mathbf{P}\!\left(  _{\!_{\!_{\,}}%
}E_{n}(\rho,\xi)\cap A_2^\varepsilon\cap D_{\theta}\right)  =0.\label{4.70}%
\end{equation}
The intensity of the jumps in $D$ [the set defined in~(\ref{setD}) and
satisfying conditions in $E_{n}(\rho,\xi)$\thinspace] is given by
\begin{align}
\int_{t-2^{-2 n}n^{\rho}}^{t-2^{-2 n}n^{-\rho}}\mathrm{d}s &
\int_{|y|\leq(t-s)^{1/2}\log^{\xi}\bigl((t-s)^{-1}\bigr)}X_{s}%
(\mathrm{d}y)\label{n13}\\
&  \qquad\qquad\qquad\qquad\qquad\frac{\log^{2\varepsilon(1+\beta
)-1}\bigl((t-s)^{-1}\bigr)}{(t-s)\max\bigl\{(t-s)^{1/2},|y|\bigr\}}%
\,.\nonumber
\end{align}
Since in (\ref{4.70}) we are interested in a limit as $N\uparrow\infty,$ we
may assume that $n$ is such that $\,(t-s)^{1/2}\log^{\xi}\bigl((t-s)^{-1}%
\bigr)\leq1$\thinspace\ for $\,s\geq t-2^{-2 n}n^{\rho}.$\thinspace\ We
next note that
\begin{align*}
&  \int_{|y|\leq(t-s)^{1/2}}\frac{X_{s}(\mathrm{d}y)}{\max
\bigl\{(t-s)^{1/2},|y|\bigr\}}\\
&  =\,(t-s)^{-1/2}X_{s}\bigl((-(t-s)^{1/2},\,(t-s)^{1/2
})\bigr)\,\leq\ \theta^{-1}\nonumber
\end{align*}
on $D_{\theta\,}$. Further, for every $j\geq1$ satisfying $\,j\leq\log^{\xi
}\bigl((t-s)^{-1}\bigr),$
\begin{align*}
&  \int_{j(t-s)^{1/2}\leq|y|\leq(j+1)(t-s)^{1/2}}\frac
{X_{s}(\mathrm{d}y)}{\max\bigl\{(t-s)^{1/2},|y|\bigr\}}\\
&  \leq j^{-1}(t-s)^{-1/2}X_{s}\left(  \left\{  y:\;j(t-s)^{1/2}%
\leq|y|\leq(j+1)(t-s)^{1/2}\right\}  \right)  .\nonumber
\end{align*}
Since the set $\,\left\{  y:\;j(t-s)^{1/2}\leq|y|\leq(j+1)(t-s)^{1/2}\right\}$
is the union of two balls with radius $\frac{1}{2}\,(t-s)^{-1/2}$
and centers in $B_{2}(0),$ we can apply Lemma~\ref{n.L2} with $c=1$ to get%
\begin{equation*}
\int_{j(t-s)^{1/2}\leq|y|\leq(j+1)(t-s)^{1/2}}\frac{X_{s}%
(\mathrm{d}y)}{\max\bigl\{(t-s)^{1/2},|y|\bigr\}}\,\leq\,2\,\theta
^{-1}j^{-1}%
\end{equation*}
on $D_{\theta\,}.$\thinspace\ As a result, on the event $D_{\theta}$ we get
the inequality%
\begin{gather*}
\int_{|y|\leq(t-s)^{1/2}\log^{\xi}\left(  (t-s)^{-1}\right)  }%
\!X_{s}(\mathrm{d}y)\,\frac{1}{\max\bigl\{(t-s)^{1/2},|y|\bigr\}}%
\nonumber\\
\leq C\theta^{-1}\log\!\Big(\left\vert \log\bigl((t-s)^{-1}\bigr)\right\vert
\!\Big).
\end{gather*}
Substituting this into (\ref{n13}), we conclude that the intensity of the
jumps is bounded by
\begin{equation*}
C\theta^{-1}\int_{t-2^{-2 n}n^{\rho}}^{t-2^{-2 n}n^{-\rho}%
}\mathrm{d}s\ \frac{\log^{2\varepsilon(1+\beta)-1}\Big((t-s)^{-1}\log
\log\bigl((t-s)^{-1}\bigr)\Big)}{(t-s)}\,.
\end{equation*}
Simple calculations show that the latter expression is less than
\begin{equation*}
C\theta^{-1}n^{2\varepsilon(1+\beta)-1}\log^{1+2\varepsilon(1+\beta)}n.
\end{equation*}
Consequently, since $E_{n}(\rho,\xi)$ holds when there are two jumps in $D$,
we have
\begin{equation*}
\mathbf{P}\!\left(  _{\!_{\!_{\,}}}E_{n}(\rho,\xi)\cap A_2^\varepsilon\cap
D_{\theta}\right)  \leq C\theta^{-2}n^{4\varepsilon(1+\beta)-2}\log
^{2+4\varepsilon(1+\beta)}n.
\end{equation*}
Because $\varepsilon<1/8\leq1/4(1+\beta)$, the sequence $\mathbf{P}\!\left(
_{\!_{\!_{\,}}}E_{n}(\rho,\xi)\cap A_2^\varepsilon\cap D_{\theta}\right)  $ is
summable, and the proof of the lemma is complete.
\hfill$\square$
\section{Elements of the proof of Theorem~\ref{thm:mfractal}}
\label{sec:multifactal}
The spectrum of singularities of $X_t$ coincides with that of $Z$. Consequently,
to prove Theorem~\ref{thm:mfractal}, we have to determine Hausdorff dimensions
of the sets
\begin{align*}
 {\mathcal E}_{Z,\eta}&:=\{x\in (0,1): H_{Z}(x)=\eta\},\\
 \tilde{\mathcal E}_{Z,\eta}&:=\{x\in (0,1): H_{Z}(x)\leq\eta\}
\end{align*}
and this is done in the next two subsections.  

As usual, we give some heuristic arguments, which should explain the result~(\ref{mfr}). Using heuristic
arguments that led to \eqref{verylast3}, one can easily show that a jump of order $(t-s)^\nu$ occurring between
times $t-s$ and $t$ implies that (if there are no other ''big`` jumps nearby) the pointwise H\"older exponent
should be  $2\nu-1$ in the ball of radius $(t-s)^{1/2}$ centered at the spatial position of this jump. In other
words, in order to have $H_X(x)=\eta$ at a point $x$, a jump of order $(t-s)^{(\eta+1)/2}$ should appear, whose
distance to $x$ is less or equal to $(t-s)^{1/2}$. 

From the formula for the compensator we infer that the number of such jumps, $\mathbf{N}_{\eta}\,,$ is of order
\begin{align*}
\mathbf{N}_{\eta}&\approx\int_{t-s}^t du X_u((0,1))\int_{(t-s)^{(\eta+1)/2}}r^{-2-\beta}dr\\
&\approx(t-s)^{-(\eta+1)(1+\beta)/2}\int_{t-s}^t du X_u((0,1)).
\end{align*}
It turns out that the random measure $X_u$ can be replaced by the
Lebesgue measure multiplied by a random factor. (The proof
of this fact is one of the main technical difficulties.) As a
result, the number of jumps $\mathbf{N}_{\eta}$, leading to pointwise H\"older exponent 
$\eta$ at certain points, is of order
$$
(t-s)^{-(\eta+1)(1+\beta)/2+1}.
$$
Since every such jump 
effects the regularity in the ball
of radius $\mathbf{r}\approx(t-s)^{1/2}$, the corresponding Hausdorff dimension $\alpha$
of the set of points $x$, with $H_X(x)=\eta$, 
can be obtained from the relation
$$
\mathbf{N}_{\eta}\mathbf{r}^{\alpha}\approx ((t-s)^{1/2})^\alpha (t-s)^{-(\eta+1)(1+\beta)/2+1}\approx 1.
$$
This gives  
$$
\alpha=(\eta+1)(1+\beta)-2=(\eta-\eta_{\rm c})(1+\beta),
$$
which coincides with~(\ref{mfr}). 

In the rest of this section we will justify the above heuristics. 
\subsection{Upper bound for the Hausdorff dimension}
\label{subsec:5.2}
The aim of this section is to prove the following proposition.
\begin{proposition}
\label{prop:08_02}
For every $\eta\in[\eta_{\rm c},\overline{\eta}_{\rm c})$,
$$
{\rm dim}({\mathcal E}_{Z,\eta})\leq 
{\rm dim}(\tilde{\mathcal E}_{Z,\eta})\leq (1+\beta)(\eta-\eta_{\rm c}),\;\;\mathbf{P}-\text{\rm a.s.}
$$
\end{proposition}

\bigskip

\mbox{}\\
\noindent We need to introduce an additional notation. In what follows, for any\\
$\eta\in (\eta_{{\rm c}},\bar{\eta}_{{\rm c}})\setminus\{1\}$, we fix an
arbitrary small $\gamma=\gamma(\eta)\in(0,\frac{10^{-2}\eta_{\rm c}}{\alpha})$ such that
\begin{eqnarray*}
\gamma &<& 
\left\{\begin{array}{lcr}
 10^{-2}\min\{1-\eta,\eta\},&&{\rm if}\; \eta < 1, \\
 10^{-2}\min\{\eta-1, 2-\eta\},&&{\rm if}\; \eta> 1, 
\end{array}
\right.
\end{eqnarray*}
and define
\begin{align*}
S_\eta:=\Big\{x\in(0,1):&\text{ there exists a sequence } (s_n,y_n)\to(t,x)\\
& \text{ with }\Delta X_{s_n}(\{y_n\})\geq (t-s_n)^{\frac{1}{1+\beta}-\gamma}|x-y_n|^{\eta-\eta_{\rm c}}\Big\}.
\end{align*} 
To prove the above proposition we have to verify the following two lemmas.
\begin{lemma}
\label{P1}
For every $\eta\in(\eta_{\rm c},\overline{\eta}_{\rm c})\setminus\{1\}$ we have
$$
\mathbf{P}\left(H_{Z}(x)\geq\eta-4\gamma\text{ for all }x\in (0,1)\setminus S_{\eta}\right)=1.
$$
\end{lemma}
\begin{lemma}
\label{P2}
For every $\eta\in(\eta_{\rm c},\overline{\eta}_{\rm c})\setminus\{1\}$ we have
$$
{\rm dim}(S_{\eta})\leq (1+\beta)(\eta-\eta_{\rm c}),\quad \mathbf{P}-\text{\rm a.s.}
$$
\end{lemma}

The aim of the Lemma~\ref{P1} is to show that for any $\varepsilon>0$ sufficiently small,
outside the set $S_{\eta+4\gamma-\varepsilon}$ the H\"older exponent  is {\it larger}
than $\eta+\varepsilon$, and hence $\tilde{\mathcal E}_{Z,\eta}\subset S_{\eta+4\gamma-\varepsilon}$. Therefore 
Lemma~\ref{P2} gives immediately the upper bound on dimension of $\tilde{\mathcal E}_{Z,\eta}$. 
We can formalize it and immediately get

\medskip
\noindent {\it Proof of Proposition~\ref{prop:08_02}.\ } It follows easily from Lemma \ref{P1} 
that  $\tilde{\mathcal E}_{Z,\eta}\subset S_{\eta+4\gamma+\varepsilon}$ 
 for every $\eta\in(\eta_{\rm c},\overline{\eta}_{\rm c})\setminus\{1\}$ and  every $\varepsilon>0$ sufficiently small. Therefore,
$$
{\rm dim}(\tilde{\mathcal E}_{Z,\eta})\leq \lim_{\varepsilon\to0}{\rm dim}(S_{\eta+4\gamma+\varepsilon}).
$$
Using Lemma \ref{P2} we then get
$$
{\rm dim}(\tilde{\mathcal E}_{Z,\eta})\leq (1+\beta)(\eta+4\gamma-\eta_{\rm c}), \quad \mathbf{P}-\text{\rm a.s.}
$$
Since $\gamma$ can be chosen arbitrary small, the result for $\eta\neq1$ follows immediately.
The inequality for $\eta=1$ follows from the monotonicity in $\eta$ of the sets $\tilde{\mathcal E}_{Z,\eta}$.
\hfill$\square$
\medskip

Let $\varepsilon\in (0,\eta_{\rm c}/2)$ be arbitrarily small.
We introduce a new ``good`` event  $A_2^\varepsilon$ which will be
frequently used throughout the proofs. On this event, with high
probability, $V=V(B_2(0))$ from Lemma~\ref{L.fixed.2}  is bounded by a constant,
and there is a bound on the sizes of jumps. By Lemma \ref{L8}, there exists a constant 
$C_{(\ref{eq:08_02_2})}=C_{(\ref{eq:08_02_2})}(\varepsilon,\gamma)$ 
such that
\begin{eqnarray}
\label{eq:08_02_2}
\mathbf{P}(|\Delta X_s|> C_{(\ref{eq:08_02_2})} (t-s)^{(1+\beta)^{-1}-\gamma}\;\;{\rm for\; some}\; s<t)\leq \varepsilon/3. 
\end{eqnarray}
Then we fix another constant $C_{(\ref{eq:08_02_3})}=C_{(\ref{eq:08_02_3})}(\varepsilon,\gamma)$ such that 
\begin{eqnarray}
\label{eq:08_02_3}
\mathbf{P}(V\leq C_{(\ref{eq:08_02_3})})\geq 1-\varepsilon/3.
\end{eqnarray}
Recall that, by Theorem~1.2 in~\cite{FMW10}, $x\mapsto X_t(x)$ is $\mathbf{P}$-a.s. 
H\"older continuous with any exponent less than $\eta_{\rm c}$. Hence we can define  a constant 
$C_{(\ref{eq:08_02_20})}=C_{(\ref{eq:08_02_20})}(\varepsilon)$ such that
\begin{eqnarray}
\label{eq:08_02_20}
\mathbf{P}\left(\sup_{x_1,x_2\in (0,1), x_1\not=x_2}
 \frac{|X_t(x_1)-X_t(x_2)|}{|x_1-x_2|^{\eta_{\rm c}-\varepsilon}}\leq C_{(\ref{eq:08_02_20})}\right)\geq 1-\varepsilon/3.
\end{eqnarray}
Now we are ready to define
\begin{align}
\label{good_set}
A_3^{\varepsilon}&:=\{|\Delta X_s|\leq  C_{(\ref{eq:08_02_2})} (t-s)^{(1+\beta)^{-1}-\gamma}\;\;{\rm for\; all}\; s<t\}
\\
\nonumber
&\cap \{ V\leq C_{(\ref{eq:08_02_3})}\}\cap
 \left\{\sup_{x_1,x_2\in (0,1), x_1\not=x_2}
 \frac{|X_t(x_1)-X_t(x_2)|}{|x_1-x_2|^{\eta_{\rm c}-\varepsilon}}\leq C_{(\ref{eq:08_02_20})}\right\}. 
\end{align}
Clearly by~(\ref{eq:08_02_2}), (\ref{eq:08_02_3}) and (\ref{eq:08_02_20}), 
$\mathbf{P}(A_3^\varepsilon)\geq 1-\varepsilon.$ See (3.4) 
in~\cite{FMW10} for the analogous definition. 

Now we are ready to give the  proof of Lemma \ref{P2}.

\medskip

\noindent {\it Proof of Lemma \ref{P2}.\ } 
To every jump $(s,y,r)$ of the measure $\cN$ (in what follows in the paper we will 
usually call them 
simply ``jumps'') with 
$$
(s,y,r)\in D_{j,n}:=[t-2^{-j}, t-2^{-j-1})\times(0,1)\times[2^{-n-1}, 2^{-n})
$$ 
we assign the ball
\begin{equation}\label{balls}
B^{(s,y,r)}:=B\left(y,\left(\frac{2^{-n}}{(2^{-j-1})^{\frac{1}{1+\beta}-\gamma}}\right)^{1/(\eta-\eta_{\rm c})}\right).
\end{equation}
We used here the obvious notation $B(y,\delta)$ for the ball in $\R$ with the center at $y$ and radius $\delta$. Define
 $n_0(j):=j[\frac{1}{1+\beta}-\frac{\gamma}{4}]$. It follows from \eqref{eq:08_02_2} and \eqref{good_set} that, on 
$A_2^\varepsilon$, there are no jumps bigger than $2^{-n_0(j)}$ in the time interval $[t-2^{-j}, t-2^{-j-1})$.

It is easy to see that every point from $S_\eta$ is contained in infinitely many balls $B^{(s,y,r)}$.
Therefore, for every $J\geq1$, the set
$$\bigcup_{j\geq J, n\geq 1}\bigcup_{(s,y,r)\in D_{j,n}} B^{(s,y,r)} $$
covers $S_\eta$. From \eqref{eq:08_02_2} and \eqref{good_set} we conclude
that, on $A_2^\varepsilon$, there are no jumps bigger than
$C_{\eqref{eq:08_02_2}}2^{-(j+1)(\frac{1}{1+\beta}-\gamma)}$ in the time
interval $s\in [t-2^{-j},t-2^{-j-1})$ for any $j\geq1$. Define
$n_0(j):=j[\frac{1}{1+\beta}-\frac{\gamma}{4}]$. Clearly, there exists $J_0$ 
such that for all $j\geq J_0$ there are no jumps bigger than $2^{-n_0(j)}$
in the time interval $[t-2^{-j}, t-2^{-j-1})$. Hence, for every $J\geq J_0$, the
set
$$
S_\eta(J):=\bigcup_{j\geq J, n\geq n_0(j)}\bigcup_{(s,y,r)\in D_{j,n}} B^{(s,y,r)} 
$$
covers $S_\eta$ for every $\omega\in A_3^\varepsilon$.

It follows from the formula for the compensator that, on the event\\
$\left\{\sup_{s\leq t}X_s((0,1))\leq N\right\}$,
the intensity of jumps with
$(s,y,r)\in D_{j,n}$ is bounded by 
$$
N2^{-j-1}\int_{2^{-n-1}}^{2^{-n}}c_\beta r^{-2-\beta}dr=\frac{Nc_\beta(2^{1+\beta}-1)}{2(1+\beta)}2^{n(1+\beta)-j}=:\lambda_{j,n}.
$$
Therefore, the intensity of jumps with $(s,y,r)\in\cup_{n=n_0(j)}^{n_1(j)}D_{j,n}=:\tilde{D}_j$, where 
$n_1(j)=j[\frac{1}{1+\beta}+\frac{\gamma}{4}]$, is bounded by
$$
\sum_{n=n_0(j)}^{n_1(j)}\lambda_{j,n}\leq \frac{Nc_\beta 2^\beta}{(\beta+1)}2^{j(1+\beta)\gamma/4}=:\Lambda_j.
$$
The number of such jumps does not exceed $2\Lambda_j$ with the probability $1-e^{-(1-2\log2)\Lambda_j}$.
This is immediate from the exponential Chebyshev inequality applied to Poisson distributed random variables.
Analogously, the number of
jumps with $(s,y,r)\in D_{j,n}$ does not exceed $2\lambda_{j,n}$ with the probability
at least $1-e^{-(1-2\log2)\lambda_{j,n}}$.
Since
$$
\sum_j\left(e^{-(1-2\log2)\Lambda_j}+\sum_{n=n_1(j)}^\infty e^{-(1-2\log2)\lambda_{j,n}}\right)<\infty,
$$
we conclude, applying the Borel-Cantelli lemma, that, for almost every $\omega$ from the set
$A_2^\varepsilon\cap\left\{\sup_{s\leq t}X_s((0,1))\leq N\right\}$, 
there exists $J(\omega)$ such that for all $j\geq J(\omega)$ and $n\geq n_1(j)$,  
the numbers of jumps in $\tilde{D}_j$
and in $D_{j,n}$ are bounded by $2\Lambda_j$ and $2\lambda_{j,n}$ respectively.

The radius of every ball corresponding to the jump in $\tilde{D}_j$ is bounded by 
$r_j:=C2^{-\frac{3\gamma}{4(\eta-\eta_{\rm c})}j}$. Thus, one can easily see that
$$
\sum_{j=1}^\infty\left(2\Lambda_j r_j^\theta+\sum_{n=n_1(j)}^\infty 2\lambda_{j,n}
\left(\frac{2^{-n}}{(2^{-j-1})^{\frac{1}{1+\beta}-\gamma}}\right)^{\theta/(\eta-\eta_{\rm c})}\right)<\infty
$$
for every $\theta>(1+\beta)(\eta-\eta_{\rm c})$. This yields the desired bound for the Hausdorff dimension
for almost every $\omega\in A_3^\varepsilon\cap\left\{\sup_{s\leq t}X_s((0,1))\leq N\right\}$.
Letting $N\to\infty$ and $\varepsilon\to0$ completes the proof.
\hfill$\square$
\medskip

The proof of Lemma \ref{P1} is omitted since it goes similarly to the proof of Theorem~\ref{T.fixed}(a) and all the 
modifications come from the necessity of dealing with numerous random points. For details of the proof of Lemma \ref{P1},
we refer the interested reader to the proof of Lemma 3.2 in~\cite{MW14}.

\subsection{Lower bound for the Hausdorff dimension}
\label{sec:6.2}
The aim of this section is to describe the main steps in the proof
of the following proposition. The full proof of it is given 
in the Section~4 of~\cite{MW14}. 
\begin{proposition}
\label{prop:08_02_3}
For every $\eta\in(\eta_{\rm c},\overline{\eta}_{\rm c})\setminus\{1\}$,
$$
{\rm dim}({\mathcal E}_{Z,\eta})\geq
(1+\beta)(\eta-\eta_{\rm c}),\;\;\mathbf{P}-\text{\rm a.s. on}
\; \{X_t((0,1))>0\}.
$$
\end{proposition}
\begin{remark}
Clearly the above proposition together with Proposition~\ref{prop:08_02}
finishes the proof of Theorem~\ref{thm:mfractal}.
\end{remark}

The proof of the lower
bound is much more involved then the proof of the upper one.  Let us give 
short description of the strategy. 
First we state two lemmas that give some uniform estimates on "masses" of $X_s$ of dyadic intervals 
at times $s$ close to $t$. These lemmas imply that $X_s(dx)$ for times $s$ close to $t$
is very close to Lebesgue measure with the density bounded from above and away of zero. 
This is very helpful for constructing a set
$J_{\eta,1}$ with 
${\rm dim}(J_{\eta,1})\geq (\beta+1)(\eta-\eta_{\rm c})$, on
which we show existence of "big" jumps of $X$ that occur close to time $t$.
These jumps are "encoded" in the jumps of the auxiliary processes $L^+_{n,l,r}$ 
and they, in fact,  "may" destroy  the H\"older continuity of $X_t(\cdot)$
on $J_{\eta,1}$ for any  index greater or equal to $\eta$ (see
Proposition \ref{last_prop} and the proof of Proposition \ref{prop:08_02_3}).

\bigskip

In the next two lemmas we give some bounds for $X_s(I_k^{(n)})$, where 
$$ I_k^{(n)}:= [ k2^{-n}, (k+1)2^{-n}).$$

In what follows, fix some 
\begin{equation}
\label{eq:02_08_1}
 m>3/2,
\end{equation}
and let $\theta\in (0,1)$ be arbitrarily small.
Define 
\begin{align*}
O_n&:=\left\{\omega:\text{ there exists }k\in[0, 2^{n}-1]\text{ such that } \right.\\
&\left.\hspace*{1.2cm}
\sup_{s\in(t-2^{-2 n}n^{4m/3},t)}X_s(I^{(n)}_k)\geq 2^{-n}n^{4m/3}\right\}
\end{align*}
and 
\begin{align*}
B_n=B_n(\theta):=&\left\{\omega:\text{ there exists }k\in [0, 2^{n}-1]\text{ such that } 
\right.\\
&\left.\hspace*{0.4cm} I^{(n)}_k\cap\{x:X_t(x)\geq\theta\}\neq\emptyset\right.
\\
&\left.\hspace*{0.4cm}\text{ and }\inf_{s\in(t-2^{-2 n}n^{-2 m},t)}X_s(I^{(n)}_k)\leq 2^{-n}n^{-2m}\right\}.
\end{align*}
\begin{lemma}\label{lower0}
There exists a constant $C$ such that
$$
\mathbf{P}(O_n)\leq Cn^{-2m/3},\;\;n\geq 1.
$$
\end{lemma}

Recall the definition of $A_3^\varepsilon$ given in \eqref{good_set}.

\begin{lemma}\label{lower1}
There exists a constant $C=C(m)$ such that, for every $\theta\in(0,1)$,
$$
\mathbf{P}(B_n(\theta)\cap A_3^\varepsilon)\leq C\theta^{-1}n^{-\alpha m/3},\;\;n\geq \tilde n(\theta),
$$
for some $\tilde n(\theta)$ sufficiently  large. 
\end{lemma}
The proofs of the above lemmas are technical and hence we omit them.  
Let us just mention that the  proof of Lemma~\ref{lower0} is an almost word-by-word repetition of the 
proof of Lemma 5.5 in \cite{FMW10}, and for the proof of Lemma~\ref{lower1} we refer the reader to the proof of 
Lemma~6.7 in~\cite{MW14}.

\subsubsection{Analysis of the set of jumps which destroy the H\"older continuity}
\label{sec:4.2}
In this subsection we construct a set $J_{\eta,1}$ such that its
Hausdorff dimension is bounded from below by  $(\beta+1)(\eta-\eta_{\rm c})$
and in the vicinity of each $x\in J_{\eta,1}$ there are jumps of $X$
which destroy the H\"older continuity at $x$ for any index greater than $\eta$.

We first introduce $J_{\eta,1}$ and prove the lower bound for its
dimension. 
Set
$$
q:=\frac{5m}{(\beta+1)(\eta-\eta_{\rm c})}$$
and define
\begin{align*}
&A_{k}^{(n)}:=\left\{\Delta X_s(I^{(n)}_{k-2n^q-2})\geq 2^{-(\eta+1)n}\right.\\
&\hspace{1cm}\left.\text{ for some }
s\in[t-2^{-2 n}n^{-2 m},t-2^{-2(n+1)}(n+1)^{-2 m})\right\},
\end{align*}
$$
J^{(n)}_{k,r}:=\left[\frac{k}{2^n}-(n^q2^{-n})^r,\frac{k+1}{2^n}+(n^q2^{-n})^r\right].
$$

Let us introduce the following notation.
For a Borel set $B$ and an event $E$ define a random set
$$
B{\mathsf 1}_E(\omega):=
\left\{
\begin{array}{cl}
B,\quad \omega\in E,\\
\emptyset,\quad\omega\notin E.
\end{array}
\right.
$$
Now we are ready to define random sets
$$
J_{\eta,r}:=\limsup_{n\to\infty}
\bigcup_{k=2n^q+2}^{2^n-1} J^{(n)}_{k,r}{\mathsf 1}_{A_k^{(n)}},
\quad r>0.
$$
As we have mentioned already we are interested in getting the lower bound
on Hausdorff dimension of $J_{\eta,1}$. The standard procedure
for this is as follows. First  show that a bit "inflated" set
$J_{\eta,r}$, for certain $r\in (0,1)$, contains open intervals.
This would imply a lower bound $r$ on the Hausdorff dimension of 
$J_{\eta,1}$ (see  Lemma \ref{lower2} and Theorem 2 from 
\cite{Jaff99} where a similar strategy was implemented). Thus  to get a
sharper bound on Hausdorff dimension of $J_{\eta,1}$ one should
try to take $r$ as large as possible. In the next lemma we show that,
in fact, one can choose $r=(\beta+1)(\eta-\eta_{\rm c})$.
\begin{lemma}
\label{lower2}
On the event $A_3^\varepsilon$,
$$
\{x\in (0,1):\,X_t(x)\geq\theta\}\subseteq J_{\eta,(\beta+1)(\eta-\eta_{\rm c})}\,,\; \pr-{\rm a.s.} 
$$
for every $\theta\in(0,1)$.
\end{lemma}
\begin{proof}
Fix an arbitrary $\theta\in(0,1)$.
We estimate the probability of the event $E_n\cap A_3^\varepsilon$, where
$$
E_n:=\left\{\omega:\{x\in (0,1):\,X_t(x)\geq\theta\}\subseteq 
\bigcup_{k=2n^q+2}^{2^n-1} J^{(n)}_{k,(\beta+1)(\eta-\eta_{\rm c})}{\mathsf 1}_{A_k^{(n)}}\right\}.
$$
To prove the lemma it is enough to show that the sequence $\mathbf{P}(E_n^c\cap A_3^\varepsilon)$
is summable.
It follows from Lemma \ref{lower1} that, for all 
$n\geq \tilde n(\theta)$,
\begin{align}
\label{lower3.1}
\nonumber
\mathbf{P}(E_n^c\cap A_3^\varepsilon)&\leq \mathbf{P}(E_n^c\cap B_n\cap A_3^\varepsilon)+
 \mathbf{P}(E_n^c\cap B_n^{c}\cap A_3^\varepsilon)\\
&\leq C\theta^{-1}n^{-2m/3}+\mathbf{P}(E_n^c\cap B_n^{c}\cap A_3^\varepsilon).
\end{align}
For any $k=0,\ldots,2^n-1,$ the compensator measure $\widehat{N}(dr,dy,ds)$
of the random measure $\mathcal{N}(dr,dy,ds)$ 
(the jump measure for $X$ --- see Lemma \ref{L.mart.dec}), 
on
\begin{align*} 
& \mathcal{J}^{(n)}_1\times I^{(n)}_k\times \mathcal{J}^{(n)}_2
\\
&:= [2^{-(\eta+1)n},\infty)\times I^{(n)}_k\times [t-2^{-2 n}n^{-2m},t-2^{-2(n+1)}(n+1)^{-2m}),
\end{align*}
is given by the formula
\begin{equation}
\label{eq30_2}
 1\{(r,y,s)\in \mathcal{J}^{(n)}_1\times I^{(n)}_k\times \mathcal{J}^{(n)}_2\}n(dr)
    X_s(dy) ds.
\end{equation}
If 
$$k\in K_{\theta}:= \{l: I^{(n)}_l\cap\{x\in (0,1):X_t(x)\geq\theta\}\neq\emptyset\},$$ 
then, 
by the definition of $B_n$, we have
\begin{equation}
\label{eq30_1}
X_s(I^{(n)}_k)\geq 2^{-n}n^{-2m},\quad\text{ 
for $s\in\mathcal{J}^{(n)}_2$, on the event $A_3^\varepsilon\cap B_n^{c}$.}
\end{equation}
Define the measure $\widehat{\Gamma}(dr,dy,ds)$ on  $\R_+\times (0,1)\times \R_+\,,$ as follows, 
\begin{align}
 \label{eq:02_08_6}
\widehat{\Gamma}(dr,dy,ds):= n(dr) n^{-2m} dy ds. 
\end{align}
Then, by~(\ref{eq30_2}) and~(\ref{eq30_1}), on $A_3^\varepsilon\cap B_n^{c}$, and on the set 
$$
\mathcal{J}^{(n)}_1\times \{y\in (0,1):X_t(y)\geq\theta\}\times \mathcal{J}^{(n)}_2
$$
we have the following  bound 
$$
\widehat{\Gamma}(dr,I_k^{(n)},\mathcal{J}^{(n)}_2)\leq
\widehat{\mathcal{N}}(dr,I_k^{(n)},\mathcal{J}^{(n)}_2), k\in K_\theta.
$$
By standard arguments it is easy to construct the Poisson point process $\Gamma(dr,dx,ds)$ 
on $\R_+\times (0,1)\times \R_+$ with intensity measure $\widehat\Gamma$ given by~(\ref{eq:02_08_6}) on the whole space 
$\R_+\times (0,1)\times \R_+$ such that on $A_3^\varepsilon\cap B_n^{c}$, 
\begin{eqnarray*}
 \Gamma(dr,I_k^{(n)},\mathcal{J}^{(n)}_2)\leq\mathcal{N}(dr,I_k^{(n)},\mathcal{J}^{(n)}_2)
\end{eqnarray*}
for $r\in \mathcal{J}^{(n)}_1$ and $k\in K_\theta$.

Now, define 
$$\xi_k^{(n)}= {\mathsf 1}_{\left\{\Gamma\left(\mathcal{J}^{(n)}_1\times I^{(n)}_{k-2n^q-2}\times \mathcal{J}^{(n)}_2
\right)\geq 1\right\}},\quad k\geq 2n^q+2.
$$
Clearly, on $A_3^\varepsilon\cap B_n^{c}$ and for $k$ such that $k-2n^q-2\in K_{\theta}$,
$$ \xi_{k}^{(n)}\leq {\mathsf 1}_{ A_k^{(n)}}.
$$  
Moreover, by construction $\{\xi_k^{(n)}\}_{k=2n^q+2}^{2^n+2n^q+1}$ is a collection of independent 
 identically distributed Bernoulli random variables with success probabilities
\begin{eqnarray*}
p^{(n)}&:=& \widehat\Gamma\left(\mathcal{J}^{(n)}_1\times I^{(n)}_{k-2n^q-2}\times 
\mathcal{J}^{(n)}_2
\right)\\
&=&
C 2^{(\eta-\eta_{\rm c})(1+\beta)n-n}n^{-4m}.
\end{eqnarray*}
From the above  coupling with the Poisson point process $\Gamma$, it is 
easy to see that
\begin{align}
\label{eq:02_08_7}
\mathbf{P}(E_n^c\cap B_n^{c}\cap A_3^\varepsilon)\leq \mathbf{P}(\tilde{E}_n^{c}),
\end{align}
where
$$
\tilde{E}_n:=\left\{(0,1)\subseteq\bigcup_{k=2n^q+2}^{2^n+2n^q+1}J^{(n)}_k{\mathsf 1}_{\{\xi^{(n)}_k=1\}}\right\}.
$$

Let $L{(n)}$ denote the length of the longest run of zeros
in the sequence $\{\xi_k^{(n)}\}_{k=2n^q+2}^{2^n+2n^q+1}\,.$ 
Clearly,
$$
\mathbf{P}(\tilde{E}_n^c)\leq\mathbf{P}(L^{(n)}\geq 2^{n-(\beta+1)(\eta-\eta_{\rm c})n}n^{5m})
$$
and it is also obvious that
$$\mathbf{P}(L^{(n)}\geq j)\leq 2^{n}p^{(n)}(1-p^{(n)})^j\,, \forall j\geq 1.$$ Use this with the fact that,
 by~(\ref{eq:02_08_1}), 
 $m>1$,
to get that
\begin{align}
\label{eq:02_08_8}
\mathbf{P}(\tilde{E}_n^c)
\leq\exp\left\{-\frac{1}{2}n^{m}\right\}
\end{align}
for all $n$ sufficiently large. Combining \eqref{lower3.1}, \eqref{eq:02_08_7}
and \eqref{eq:02_08_8}, we conclude that the sequence 
$\mathbf{P}(E_n^c\cap A_3^\varepsilon)$ is summable.
Applying Borel-Cantelli, we complete the proof of the lemma. \hfill$\square$
\end{proof}

Define
$$
h_\eta(x):=x^{(\beta+1)(\eta-\eta_{\rm c})}\log^2\frac{1}{x}
$$
and
$$
\mathcal{H}_\eta(A):=\lim_{\epsilon\to0}
\inf\left\{\sum_{j=1}^\infty h_\eta(|I_j|), A\in\bigcup_{j=1}^\infty I_j\text{ and }|I_j|\leq\epsilon\right\}.
$$
Combining Lemma \ref{lower2} and Theorem 2 from \cite{Jaff99}, one can easily get
\begin{corollary}
\label{lower3}
On the event $A_3^\varepsilon\cap\{X_t((0,1))>0\}$,
$$
\mathcal{H}_\eta(J_{\eta,1})>0,\;\; \pr-{\rm a.s.}
$$
and, consequently, on $A_3^\varepsilon\cap\{X_t((0,1))>0\}$, 
$$
{\rm dim}(J_{\eta,1})\geq (\beta+1)(\eta-\eta_{\rm c}),\quad \pr-{\rm a.s.}
$$
\end{corollary}
\begin{proof}
Fix any $\theta\in(0,1)$.
If $\omega\in A_2^\varepsilon$ is such that $B_\theta:=\{x\in(0,1):X_t(x)\geq\theta\}$
is not empty, then by the local H\"older continuity of $X_t(\cdot)$ there exists an
open interval $(x_1(\omega),x_2(\omega))\subset B_{\theta/2}$. Moreover, in view of
Lemma~\ref{lower2}, 
$$
(x_1(\omega),x_2(\omega))\subset J_{\eta,(\beta+1)(\eta-\eta_{\rm c})}(\omega),\quad \pr-{\rm a.s.}
$$
on the event $A_2^\varepsilon\cap\{B_\theta \text{ is not empty}\}$. Thus, we may apply
Theorem 2 from \cite{Jaff99} to the set $(x_1(\omega),x_2(\omega))$, which gives
$$
\mathcal{H}_\eta((x_1(\omega),x_2(\omega))\cap J_{\eta,1})>0,\quad \pr-{\rm a.s.}
$$
on the event $A_2^\varepsilon\cap\{B_\theta \text{ is not empty}\}$.
Thus,
$$
{\rm dim}((x_1(\omega),x_2(\omega))\cap J_{\eta,1})\geq (\beta+1)(\eta-\eta_{\rm c}),\quad \pr-{\rm a.s.}
$$
on the event $A_2^\varepsilon\cap\{B_\theta \text{ is not empty}\}$. Due to the monotonicity
of $\mathcal{H}_\eta(\cdot)$ and ${\rm dim}(\cdot)$, we conclude that 
$\mathcal{H}_\eta(J_{\eta,1})>0$ and
${\rm dim}(J_{\eta,1})\geq (\beta+1)(\eta-\eta_{\rm c})$, $\pr$-a.s.
on the event $A_2^\varepsilon\cap\{B_\theta \text{ is not empty}\}$. Noting that
${\mathsf 1}_{\{B_\theta \text{ is not empty}\}}\uparrow{\mathsf 1}_{\{X_t(0,1)>0\}}$ as $\theta\downarrow0$,
$\pr$-a.s., we complete the proof. \hfill$\square$
\end{proof}

Now we turn to the second part of the present subsection. 
By construction of $J_{\eta,1}$ we know that  to the 
left of every point $x\in J_{\eta,1}$ there 
exist big jumps of $X$ at time $s$ ``close'' to $t$: such 
jumps are  defined by  the events $A^{(n)}_k$. We would like to show 
that these jumps  will result in destroying the H\"older continuity of any index greater than $\eta$ 
at the point $x$. To this end, we will introduce auxiliary processes $L^{\pm}_{n,y,x}$  that are indexed by points 
 $(y,x)$ on a 
grid {\it finer} than $\{k2^{-n}, k=0,1,\ldots\}$. That is, take some integer $Q>1$ (note, that eventually
$Q$ will be chosen large enough, depending on $\eta$). Define
$$
Z^{\eta}_s(x_1,x_2):=
\int_0^s\int_{\R} M\left(\mathrm{d}(u,y)\right) p^{\eta}_{t-u}(x_1-y,x_2-y),\;s\in [0,t],
$$
where
$$
p^{\eta}_s(x,y):=
\left\{ \begin{array}{rcl}
   p_{s}(x)-p_{s}(y),&&{\rm if}\; \eta\leq 1,\\
  p_{s}(x)-p_{s}(y)-(x-y)\frac{\partial p_s(y)}{\partial y},&&{\rm if}\; \eta\in 
  (1,\bar\eta_{\rm c}).
\end{array}
\right.
$$

According to \eqref{Levy.rep} and \eqref{def.Z.new}, for every $x,y\in  2^{-Qn}\Z$,
there exist spectrally positive  $(1+\beta)$-stable processes $L^\pm_{n,y,x}$ such that
\begin{eqnarray}
\label{eq:18_10_2}
Z^\eta_s(y,x)&=&L^+_{n,y,x}(T^{n,y,x}_+(s))-
 L^-_{n,y,x}(T^{n,y,x}_{-}(s))\;\;\\
\nonumber 
&=:&\bL^+_{n,y,x}-\bL^-_{n,y,x},
\end{eqnarray}
where
\begin{eqnarray}
\nonumber
T^{n,y,x}_\pm(s)=\int_0^s du \int_{\mathbb{R}}X_u(dz)
\left(\left({p}^{\eta}_{t-u}(y-z,x-z)\right)^{\pm}\right)^{1+\beta},\;\;s\leq t.
\end{eqnarray}
In what follows 
let $[z]$ denote the integer part of $z$ for $z\in \R$.
The crucial ingredient for the proof of the lower bound is the following proposition. 
\begin{proposition}
\label{last_prop}
Fix arbitrary 
$\eta\in(\eta_{\rm c},\bar{\eta}_{\rm c})\setminus\{1\}$ and $Q>1$.
For $\pr$-a.s. $\omega$ on $A_3^\varepsilon$, there exists a set 
$\mathbf{G}_{\eta}\in  [0,1]$ with 
\begin{equation}
\label{last_eq}
{\rm dim}(\mathbf{G}_{\eta})<(\beta+1)(\eta-\eta_{\rm c})
\end{equation}
 such that the following holds.  
For $\pr$-a.s. $\omega$ on $A_3^\varepsilon$, for every $x\in  J_{\eta,1}\setminus \mathbf{G}_{\eta}$,
there exist a (random) sequences 
$$ n_j=n_j(x),\;\; j\geq 1,$$ and $$(x_{n_j}, y_{n_{j}})= (x_{n_j}(x), y_{n_{j}}(x)),\;\; j\geq 1$$
with 
\begin{eqnarray*}
x_{n_j}&=& 2^{-Qn_j}[2^{Qn_j}x],\;\;j\geq 1,\\
 | y_{n_j} - x_{n_j}| &\leq& C n_j^q2^{-n_j},\;\;j\geq 1,
\end{eqnarray*}
such that 
$$
\bL^+_{n_j,y_{n_j},x_{n_j}}\geq n_j^{m}2^{-\eta n_j},\quad
\bL^-_{n_j,y_{n_j},x_{n_j}}\leq 2^{-(\eta n_j-1)},
$$
for all $n_j$ sufficiently large.
\end{proposition}
 
Note that in the above proposition we do not give  precise definition of $y_{n_j}$, but it is chosen in a way 
that it is  ``close'' to the spatial position of a ``big'' jump that is supposed to destroy the H\"older continuity 
at $x$. As for the point $x_{n_j}$, it is chosen to be ``close'' to $x$ itself.

Similarly to Proposition \ref{very_new_prop}, Proposition~\ref{last_prop} deals
with possible compensation effects. In contrast to the case of fixed points
considered in Proposition~\ref{very_new_prop}, we cannot
show that the compensation  does not happen. But we derive an upper bound for the Hausdorff dimension of
the set $\mathbf{G}_{\eta}$, on which such a compensation may occur. The dimension of $\mathbf{G}_{\eta}$
turns to be strictly smaller than $(\beta+1)(\eta-\eta_{\rm c})$, see \eqref{last_eq}.

The proof of Proposition \ref{last_prop} is rather technical and uses the same ideas as the proof of
Proposition \ref{very_new_prop}. The major additional difficulty comes from the need to consider
random sets. To overcome it we use Borel-Cantelli arguments. We refer the reader to Section~4.3 in~\cite{MW14} where the proofs of 
results leading to Proposition~\ref{last_prop} are given. 
Now we will 
explain how Proposition~\ref{last_prop} implies the proof of Proposition~\ref{prop:08_02_3} for the case of $\eta<1$ (the proof for the case of $\eta>1$ 
goes along the similar lines, see Section 4.4 in~\cite{MW14}). 
 \medskip

 \noindent\emph{Proof of Proposition}~\ref{prop:08_02_3} {\emph {for}} $\eta<1$.\ \  
Fix arbitrary 
$\eta\in(\eta_{\rm c},\min(\bar{\eta}_{\rm c},1))$. 
Also fix 
$$
Q=\left[4\frac{\eta}{\eta_{\rm c}}+2\right],
$$
where as usual $[z]$ denotes the integer part of $z$.
Let $\mathbf{G}_{\eta}$ be as in Proposition~\ref{last_prop}.
If  $x$ is an  arbitrary point in $J_{\eta,1}\setminus \mathbf{G}_{\eta}$, 
then let 
$\{n_j(x)\}_{j\geq 1}$ and 
$\{(x_{n_j}(x), y_{n_{j}}(x))\}_{j\geq 1}$ be the  sequences constructed in 
Proposition \ref{last_prop}.
Then 
Proposition \ref{last_prop} implies that,
\begin{equation}\label{compens2.8}
\liminf_{j\to\infty}2^{(\eta+\delta)n_j}
\left|Z^\eta_t\left(y_{n_j}(x),x_{n_j}(x)\right)\right|=\infty,
\quad \forall x\in  J_{\eta,1}\setminus \mathbf{G}_{\eta},\;  \pr-{\rm a.s.}\;\; {\rm on}\,\; A_3^\varepsilon.
\end{equation}
for any $\delta>0$. Recall that, $X_t(\cdot)$ and $Z_t(\cdot)$ are H\"older continuous
with any  exponent less than $\eta_{\rm c}$ at every point of $(0,1)$. Therefore, recalling that
$Q>4\frac{\eta}{\eta_{\rm c}}$, we have
\begin{align}
\nonumber
&\lim_{j\rightarrow\infty}\sup_{x\in (0,1)}2^{(\eta+\delta)n_j}\left|Z^\eta_t(x,x_{n_j}(x))\right|\\
&\hspace{1cm}=\lim_{j\rightarrow \infty} C(\omega) 2^{-\frac{1}{2}Q\eta_{\rm c} n_j}2^{(\eta+\delta)n_j}
\label{compens2.1}
=0,\quad \pr-{\rm a.s.}\;\; {\rm on}\,\; A_3^\varepsilon.
\end{align}
Therefore, for any $x$ in $J_{\eta,1}\setminus \mathbf{G}_{\eta}$, we have
\begin{eqnarray}
\label{eq:triang}
 |Z^\eta_t(y_{n_j}(x),x)|\geq |Z^\eta_t(y_{n_j}(x),x_{n_j}(x))|-|Z^\eta_t(x_{n_j}(x),x)|,\;\;j\geq 1. 
\end{eqnarray}
Therefore, combining (\ref{compens2.8}) and (\ref{compens2.1}), (\ref{eq:triang})
we conclude that
\begin{equation}
\label{compens2.1a}
H_{Z}(x)\leq \eta,\ \text{ for all }x\in  J_{\eta,1}\setminus 
\mathbf{G}_{\eta}\,, 
\;\; \pr-{\rm a.s.}\;\; {\rm on}\,\; A_3^\varepsilon.
\end{equation}

We know, by Lemma~\ref{P1}, that 
$$
H_{Z}(x)\geq \eta-2\gamma-2\rho\ \text{ for all }x\in (0,1)\setminus S_{\eta-2\rho},
\;\;\pr-{\rm a.s.},
$$ 
This and (\ref{compens2.1a}) imply that ${\rm on}\,\; A_3^\varepsilon, \pr-{\rm a.s.},$
\begin{equation}
\label{compens2.0}
\eta-2\gamma-2\rho\leq H_{Z}(x)\leq \eta\ \text{ for all }x\in ( J_{\eta,1}\setminus
S_{\eta-2\rho})\setminus \mathbf{G}_{\eta}. 
\end{equation}
It follows easily from Lemma \ref{P2}, Corollary \ref{lower3} and Lemma \ref{compens} that
${\rm on}\,\; A_3^\varepsilon$
$$
{\rm dim}\Big(( J_{\eta,1}\setminus S_{\eta-2\rho})\setminus \mathbf{G}_{\eta}\Big)\geq(\beta+1)(\eta-\eta_{\rm c}),\;\;
\pr-{\rm a.s.}
$$
Thus, by (\ref{compens2.0}),
$$
{\rm dim}\{x:H_{Z}(x)\leq\eta\}\geq(\beta+1)(\eta-\eta_{\rm c}), \;\;{\rm on}\,\; A_3^\varepsilon,\; \pr-{\rm a.s.}.
$$
It is clear that
\begin{align*}
\{x:H_{Z}(x)=\eta\}\cup\bigcup_{n=n_0}^\infty\{x:H_{Z}(x)\in(\eta-n^{-1},\eta-(n+1)^{-1}]\}\\
=\{x:\eta-n_0^{-1}\leq H_{Z}(x)\leq\eta\}.
\end{align*}
Consequently,
\begin{align*}
&\mathcal{H}_\eta(\{x:\eta-n_0^{-1}\leq H_{Z}(x)\leq\eta\})\\
&\hspace{1cm}=\mathcal{H}_\eta(\{x:H_{Z}(x)=\eta\})\\
&\hspace{2cm}+\sum_{n=n_0}^\infty\mathcal{H}_\eta(\{x:H_{Z}(x)\in(\eta-n^{-1},\eta-(n+1)^{-1}]\}).
\end{align*}
Since the dimensions of $S_{\eta-2\rho}$ and $\mathbf{G}_{\eta}$ are smaller than $\eta$, the $\mathcal{H}_\eta$-measure of these
sets equals zero. Applying Corollary \ref{lower3}, we then conclude that on $ A_3^\varepsilon$ 
$$\mathcal{H}_\eta((J_{\eta,1}\setminus S_{\eta-2\rho})\setminus \mathbf{G}_{\eta})>0, \pr-{\rm a.s.}$$
 And in view of (\ref{compens2.0}), 
$\mathcal{H}_\eta(\{x:\eta-n_0^{-1}\leq H_{Z}(x)\leq\eta\})>0$. Furthermore, it follows from Proposition~\ref{prop:08_02},
that  dimension of the set $\{x:H_{Z}(x)\in(\eta-n^{-1},\eta-(n+1)^{-1}]\}$ is bounded from above by 
$(\beta+1)(\eta-(n+1)^{-1}-\eta_{\rm c})$. 
Hence, the definition of
 $\mathcal{H}_{\eta}$ immediately yields 
$$\mathcal{H}_\eta(\{x:H_{Z}(x)\in(\eta-n^{-1},\eta-(n+1)^{-1}]\})=0, \;\; {\rm on}\; A_3^\varepsilon,\; \pr-{\rm a.s.},$$
for all $n\geq n_0$. As a result we have
\begin{align}
\label{eq:02_08_10}
\mathcal{H}_\eta(\{x:H_{Z}(x)=\eta\})>0\;\; \pr-{\rm a.s.}\;\; {\rm on}\,\; A_3^\varepsilon.
\end{align}
Since $\varepsilon>0$ was arbitrary, this implies that~(\ref{eq:02_08_10}) is satisfied on the whole 
probability space $\pr$-a.s.  
From this we get  that $${\rm dim}\{x:H_{Z}(x)=\eta\}\geq(\beta+1)(\eta-\eta_{\rm c}), \ \pr-{\rm a.s.}$$
\hfill$\square$

\begin{appendix}
 \section{Estimates for the transition kernel of the one-dimensional Brownian motion}
We start with the following estimates for $p_t$ which are taken from Rosen \cite{Ros87}.
\begin{lemma}
\label{L.kernel1} 
Let $d=1$.
For each $\delta\in(0,1]$ there exists a constant $C$ such that
\begin{eqnarray}
\label{L.kernel1.1}
&&\bigl|p_{t}(x)-p_{t}(y)\bigr|
\leq C\frac{|x-y|^{\delta}}{t^{\delta/2}}\bigl(p_{t}(x/2)+p_{t}(y/2)\bigr),\\
\label{L.kernel1.2}
&&\left|\frac{\partial p_{t}(x)}{\partial x}\right|
\leq Ct^{-1/2}p_{t}(x/2),\\
\label{L.kernel1.3}
&&\left|\frac{\partial p_{t}(x)}{\partial x}-\frac{\partial p_{t}(y)}{\partial y}\right|
\leq C\frac{|x-y|^{\delta}}{t^{(1+\delta)/2}}\,\bigl(p_{t}(x/2)+p_{t}(y/2)\bigr),\\
\label{L.kernel1.4}
&&\left|p_{t}(x)-p_{t}(y) -(x-y)\frac{\partial p_{t}(y)}{\partial y} \right|
\leq C\frac{|x-y|^{1+\delta}}{t^{(1+\delta)/2}}\,\bigl(p_{t}(x/2)+p_{t}(y/2)\bigr)
\end{eqnarray}
for all $t>0$ and $x,y\in\R$.
\end{lemma}
The next lemma is a simple corollary of the previous one.
\begin{lemma}
\label{L.kernel2}
Let $d=1$. If $\theta\in(1,3)$ and $\delta\in(0,1]$
satisfy $\delta<(3-\theta)/\theta,$ then
\begin{align*}
&  \int_{0}^{t}\mathrm{d}s\int_{\mathbb{R}}\mathrm{d}y\ p_{s}(y)
\bigl|p_{t-s}(x_{1}-y)-p_{t-s}(x_{2}-y)\bigr|^{\theta}\\
&  \leq C (1+t)|x_{1}-x_{2}|^{\delta\theta}
\bigl(p_{t}(x_{1}/2)+p_{t}(x_{2}/2)\bigr),\quad t>0,\ \,x_{1},x_{2}\in
\R.
\end{align*}
\end{lemma}

\begin{proof}
By Lemma~\ref{L.kernel1}, for every $\delta\in\lbrack0,1]$,
\begin{align*}
&  \bigl|p_{t-s}(x_{1}-y)-p_{t-s}(x_{2}-y)\bigr|^{\theta}\\
&  \leq\,C\,\frac{|x_{1}-x_{2}|^{\delta\theta}}{(t-s)^{\delta\theta/2}}
\Big(p_{t-s}\bigl((x_{1}-y)/2\bigr)+p_{t-s}\bigl((x_{2}-y)/2\bigr)\Big)^{\theta},
\end{align*}
$t>s\geq0$, $x_{1},x_{2},y\in\R$. Noting that
$p_{t-s}(\cdot)\leq C\,(t-s)^{-1/2}$, we obtain
\begin{align}
\label{L.kernel2.1}
&  \bigl|p_{t-s}(x_{1}-y)-p_{t-s}(x_{2}-y)\bigr|^{\theta}\\
&  \leq C\frac{|x_{1}-x_{2}|^{\delta\theta}}{(t-s)^{(\delta\theta+\theta-1)/2}}
\Big(p_{t-s}\bigl((x_{1}-y)/2\bigr)+p_{t-s}\bigl((x_{2}-y)/2\bigr)\Big),\nonumber
\end{align}
$t>s\geq0$, $x_{1},x_{2},y\in\R$. Therefore,
\begin{align*}
&  \int_{0}^{t}\mathrm{d}s\int_{\R}\mathrm{d}y p_{s}(y)
\bigl|p_{t-s}(x_{1}-y)-p_{t-s}(x_{2}-y)\bigr|^{\theta}
\leq C|x_{1}-x_{2}|^{\delta\theta}\\
&  \times\int_{0}^{t}\mathrm{d}s\ (t-s)^{-(\delta\theta+\theta-1)/2}
\int_{\R}\mathrm{d}y\ p_{s}(y)
\Big(p_{t-s}\bigl((x_{1}-y)/2\bigr)+p_{t-s}\bigl((x_{2}-y)/2\bigr)\Big).
\end{align*}
By scaling of the kernel $p$,
\begin{align*}
&  \int_{\R}\mathrm{d}y p_{s}(y) p_{t-s}\bigl((x-y)/2\bigr)
=\frac{1}{2}\int_{\R}\mathrm{d}y\ p_{s/4}(y/2)p_{t-s}\bigl((x_{2}-y)/2\bigr)\\
&  =\frac{1}{2} p_{s/4+t-s}(x/2)
=\frac{1}{2}\bigl(s/4+t-s\bigr)^{-1/2}\,p_{1}\bigl((s/4+t-s)^{-1/2}x/2\bigr)\\
&  \leq t^{-1/2}p_{1}(t^{-1/2}x/2)=p_{t}(x/2),
\end{align*}
since $t/4\leq s+t/4-s\leq t$.

As a result we have the inequality
\begin{align*}
&  \int_{0}^{t}\mathrm{d}s\int_{\R}\mathrm{d}y p_{s}(y)
\bigl|p_{t-s}(x_{1}-y)-p_{t-s}(x_{2}-y)\bigr|^{\theta}\\
&  \hspace{0.5cm}\leq C |x_{1}-x_{2}|^{\delta\theta}
\bigl(p_{t}(x_{1}/2)+p_{t}(x_{2}/2)\bigr)
\int_{0}^{t}\mathrm{d}s\ s^{-(\delta\theta+\theta-1)/2}.
\end{align*}
Noting that the latter integral is bounded by $C(1+t),$ since
$\,(\delta\theta+\theta-1)<2$, we get the desired inequality.
\hfill$\square$
\end{proof}

\section{Probability inequalities for a spectrally positive stable process}
Let $L$ be a spectrally positive stable process of index $\kappa$ with Laplace transform
given by \eqref{Laplace}.
Let $\,\Delta L_{s}:=L_{s}-L_{s-}>0$\thinspace\ denote the jumps of $\,L.$

\begin{lemma}
\label{L3}
We have
\begin{equation*}
\mathbf{P}\Bigl(\,\sup_{0\leq u\leq t}L_{u}\mathsf{1}\bigl\{\sup_{0\leq v\leq
u}\Delta L_{v}\leq y\bigr\}\geq x\Bigr)\,\leq\,\Bigl(\frac{C\,t}{xy^{\kappa
-1}}\Bigr)^{\!x/y},\quad t>0,\ \,x,y>0.
\end{equation*}
{}
\end{lemma}
\begin{proof}
Since for $r>0$ fixed, $\{L_{r t}:\,t\geq0\}$ is equal to
$r^{1/\kappa}L$ in law, for the proof we may assume that $t=1.$ Let $\{\xi
_{i}:\,i\geq1\}$ denote a family of independent copies of $L_{1\,}.$%
\thinspace\ Set
\begin{equation*}
W_{ns}\,:=\,\sum_{1\leq k\leq ns}\xi_{k},\quad L_{s}^{(n)}\,:=\,n^{-1/\kappa
}W_{ns},\quad0\leq s\leq1,\ \,n\geq1.
\end{equation*}
Denote by $D_{[0,1]}$ the Skorokhod space of c\`{a}dl\`{a}g functions
$\,f:[0,1]\rightarrow\R.$\thinspace\ For fixed $y>0,$ let
$H:D_{[0,1]}\mapsto\R$ be defined by
\begin{equation*}
H(f)\,=\,\sup_{0\leq u\leq1}f(u)\,\mathsf{1}\Big\{\sup_{0\leq v\leq u}\Delta
f(v)\leq y\Big\},\quad f\in D_{[0,1]\,}.
\end{equation*}
It is easy to verify that $H$ is continuous on the set $D_{[0,1]}\setminus
J_{y}$, where $J_{y}$\thinspace:$=$\thinspace$\bigl\{f\in D_{[0,1]}:\,\Delta
f(v)=y\text{ for some }v\in\lbrack0,1]\bigr\}$. Since $\mathbf{P}(L\in
J_{y})=0$, from the invariance principle (see, e.g., Gikhman and Skorokhod
\cite{GS69}, Theorem~9.6.2) for $L^{(n)}$ we conclude that%
\begin{equation*}
\mathbf{P}\bigl(H(L)\geq x\bigr)\,=\,\lim_{n\uparrow\infty}\mathbf{P}%
\bigl(H(L^{(n)})\geq x\bigr),\quad x>0.
\end{equation*}
Consequently, the lemma will be proved if we show that
\begin{gather}
\mathbf{P}\Bigl(\,\sup_{0\leq u\leq1}W_{nu}\mathsf{1}\bigl\{\max_{1\leq k\leq
nu}\xi_{k}\leq yn^{1/\kappa}\bigr\}\geq xn^{1/\kappa}\Bigr)\,\nonumber\\
\leq\,\Bigl(\frac{C}{xy^{\kappa-1}}\Bigr)^{\!x/y},\quad x,y>0,\ \,n\geq1.
\label{L3.1}%
\end{gather}
To this end, for fixed $\,y^{\prime},h\geq0,$\thinspace\ we consider the
sequence%
\begin{equation*}
\Lambda_{0}:=1,\quad\Lambda_{n}\,:=\,\mathrm{e}^{hW_{n}}\mathsf{1}%
\bigl\{\max_{1\leq k\leq n}\xi_{k}\leq y^{\prime}\bigr\},\quad n\geq1.
\end{equation*}
It is easy to see that
\begin{equation*}
\mathbf{E}\{\Lambda_{n+1}\,|\,\Lambda_{n}=\mathrm{e}^{hu}\}\,=\,\mathrm{e}%
^{hu}\,\mathbf{E}\{\mathrm{e}^{hL_{1}};\,L_{1}\leq y^{\prime}\}\quad\text{for
all }\,u\in\R%
\end{equation*}
and that
\begin{equation*}
\mathbf{E}\{\Lambda_{n+1}\,|\,\Lambda_{n}=0\}\,=\,0.
\end{equation*}
In other words,
\begin{equation}
\mathbf{E}\{\Lambda_{n+1}\,|\,\Lambda_{n}\}\,=\,\Lambda_{n}\,\mathbf{E}%
\{\mathrm{e}^{hL_{1}};\,L_{1}\leq y^{\prime}\}. \label{L3.2}%
\end{equation}
This means that $\{\Lambda_{n}:\ n\geq1\}$ is a supermartingale
(submartingale) if $h$ satisfies $\mathbf{E}\{\mathrm{e}^{hL_{1}};\,L_{1}\leq
y^{\prime}\}\leq1\,\ $(respectively $\,\mathbf{E}\{\mathrm{e}^{hL_{1}}%
;L_{1}\leq y^{\prime}\}\geq1\,$). If $\Lambda_{n}$ is a submartingale, then by
Doob's inequality,
\begin{equation*}
\mathbf{P}\bigl(\max_{1\leq k\leq n}\Lambda_{k}\geq\mathrm{e}^{hx^{\prime}%
}\bigr)\,\leq\,\mathrm{e}^{-hx^{\prime}}\,\mathbf{E}\Lambda_{n\,},\quad
x^{\prime}>0.
\end{equation*}
But if $\Lambda_{n}$ is a supermartingale, then
\begin{equation*}
\mathbf{P}\bigl(\max_{1\leq k\leq n}\Lambda_{k}\geq\mathrm{e}^{hx^{\prime}%
}\bigr)\,\leq\,\mathrm{e}^{-hx^{\prime}}\,\mathbf{E}\Lambda_{0\,}%
=\,\mathrm{e}^{-hx^{\prime}},\quad x^{\prime}>0.
\end{equation*}
\textrm{F}rom these
inequalities and (\ref{L3.2}) we get
\begin{equation}
\mathbf{P}\bigl(\max_{1\leq k\leq n}\Lambda_{k}\geq\mathrm{e}^{hx^{\prime}%
}\bigr)\,\leq\,\mathrm{e}^{-hx^{\prime}}\max\Bigl\{1,\bigl(\mathbf{E}%
\{\mathrm{e}^{hL_{1}};\,L_{1}\leq y^{\prime}\}\bigr)^{n}\Bigr\}. \label{L3.3}%
\end{equation}
It was proved by Fuk and Nagaev \cite{FN71} (see the first formula in
the proof of Theorem~4 there) that
\[
\mathbf{E}\{\mathrm{e}^{hL_{1}};\,L_{1}\leq y^{\prime}\}\,\leq\,1+\hspace
{0.2pt}h\mathbf{E}\{L_{1\,};\,L_{1}\leq y^{\prime}\}\hspace{1pt}%
+\,\frac{\mathrm{e}^{hy^{\prime}}-1-hy^{\prime}}{(y^{\prime})^{2}%
}\,V(y^{\prime}),\,\ \,h,y^{\prime}>0,
\]
where $\,V(y^{\prime}):=\int_{-\infty}^{y^{\prime}}\mathbf{P}(L_{1}%
\in\mathrm{d}u)\,u^{2}>0.$\thinspace\ Noting that the assumption
$\mathbf{E}L_{1}=0$ yields that $\,\mathbf{E}\{L_{1\,};\,L_{1}\leq y^{\prime
}\}\leq0,$\thinspace\ we obtain
\begin{equation}
\mathbf{E}\{\mathrm{e}^{hL_{1}};\,L_{1}\leq y^{\prime}\}\,\leq\,1+\frac
{\mathrm{e}^{hy^{\prime}}-1-hy^{\prime}}{(y^{\prime})^{2}}\,V(y^{\prime
}),\quad h,y^{\prime}>0. \label{Nag}%
\end{equation}
Now note that%
\begin{align}
\Big\{\max_{1\leq k\leq n}W_{k}\mathsf{1}\{\max_{1\leq i\leq k}\xi_{i}  &
\leq y^{\prime}\}\geq x^{\prime}\Big\}\,=\,\Big\{\max_{1\leq k\leq
n}\mathrm{e}^{hW_{k}}\,\mathsf{1}\{\max_{1\leq i\leq k}\xi_{i}\leq y^{\prime
}\}\geq\mathrm{e}^{hx^{\prime}}\Big\}\nonumber\label{eq:2.30}\\
\,  &  =\,\bigl\{\max_{1\leq k\leq n}\Lambda_{k}\geq\mathrm{e}^{hx^{\prime}%
}\bigr\}.
\end{align}
Thus, combining (\ref{eq:2.30}), (\ref{Nag}), and (\ref{L3.3}), we get
\begin{gather}
\mathbf{P}\Bigl(\max_{1\leq k\leq n}W_{k}\mathsf{1}\{\max_{1\leq i\leq k}%
\xi_{i}\leq y^{\prime}\}\geq x^{\prime}\Bigl)\,\leq\,\mathbf{P}\bigl(\max
_{1\leq k\leq n}\Lambda_{k}\geq\mathrm{e}^{hx^{\prime}}\bigr)\nonumber\\
\,\leq\,\exp\Bigl\{-hx^{\prime}+\frac{\mathrm{e}^{hy^{\prime}}-1-hy^{\prime}%
}{(y^{\prime})^{2}}\,n\,V(y^{\prime})\Bigr\}\nonumber.
\end{gather}
Choosing $h:=(y^{\prime})^{-1}\log\bigl(1+x^{\prime}y^{\prime}/n\,V(y^{\prime
})\bigr)$, we arrive, after some elementary calculations, at the bound
\begin{equation*}
\mathbf{P}\Bigl(\max_{1\leq k\leq n}W_{k}\mathsf{1}\{\max_{1\leq i\leq k}%
\xi_{i}\leq y^{\prime}\}\geq x^{\prime}\Bigl)\,\leq\Bigl(\,\frac
{\mathrm{e}\,n\,V(y^{\prime})}{x^{\prime}y^{\prime}}\Bigr)^{\!x^{\prime
}/y^{\prime}},\quad x^{\prime},y^{\prime}>0.
\end{equation*}
Since $\mathbf{P}(L_{1}>u)\sim C\,u^{-\kappa}$ as $u\uparrow\infty$, we have
$\,V(y^{\prime})\leq C\,(y^{\prime})^{2-\kappa}$ for all $y^{\prime}>0.$
Therefore,
\begin{equation}
\mathbf{P}\Bigl(\max_{1\leq k\leq n}W_{k}\mathsf{1}\{\max_{1\leq i\leq k}%
\xi_{i}\leq y^{\prime}\}\geq x^{\prime}\Bigl)\,\leq\Bigl(\frac{Cn}{x^{\prime
}(y^{\prime})^{\kappa-1}}\Bigr)^{\!x^{\prime}/y^{\prime}},\quad x^{\prime
},y^{\prime}>0. \label{L3.4}%
\end{equation}
Choosing finally $\,x^{\prime}=xn^{1/\kappa},$\thinspace\ $y^{\prime
}=yn^{1/\kappa},$\thinspace\ we get (\ref{L3.1}) from (\ref{L3.4}). Thus, the
proof of the lemma is complete.\hfill$\square$
\end{proof}

\begin{lemma}
\label{L.small.values}There is a constant
$\,c_{\kappa}$ such that%
\begin{equation*}
\mathbf{P}\Big(\inf_{u\leq t}L_{u}<-x\Big)\leq\,\exp\!\Big\{-c_{\kappa}%
\,\frac{x^{\kappa/(\kappa-1)}}{t^{1/(\kappa-1)}}\Big\},\quad x,t>0.
\end{equation*}

\end{lemma}

\begin{proof}
It is easy to see that for all $h>0,$%
\begin{equation*}
\mathbf{P}\Big(\inf_{u\leq t}L_{u}<-x\,\Big)=\,\mathbf{P}\Big(\sup_{s\leq
t}\mathrm{e}^{-hL_{u}}>\mathrm{e}^{hx}\Big).
\end{equation*}
Applying Doob's inequality to the submartingale \thinspace$t\mapsto
\mathrm{e}^{-hL_{t}}$, we obtain
\begin{equation*}
\mathbf{P}\Big(\inf_{u\leq t}L_{u}<-x\,\Big)\leq\,\mathrm{e}^{-hx}%
\,\mathbf{E}\,\mathrm{e}^{-hL_{t}}.
\end{equation*}
Taking into account definition (\ref{Laplace}), we have%
\begin{equation*}
\mathbf{P}\Big(\inf_{u\leq t}L_{u}<-x\,\Big)\leq\,\exp\!\Big\{-hx+th^{\kappa
}\Big\}.
\end{equation*}
Minimizing the function $h\mapsto-hx+th^{\kappa}$, we get the inequality in
the lemma with $c_{\kappa}=(\kappa-1)/\bigl(\kappa\bigr)^{\kappa/(\kappa-1)}$.
\hfill$\square$
\end{proof}
\end{appendix}






\end{document}